\documentclass[10pt,a4paper]{article}
\usepackage[utf8]{inputenc}
\usepackage[american]{babel}
\usepackage[all]{xy}
\usepackage{enumitem}
\usepackage[dvipsnames]{xcolor}
\usepackage{graphicx}
\usepackage{stmaryrd}
\usepackage{amsfonts,amstext,amssymb,amsmath,amsthm,pspicture,bigints}
\usepackage{fullpage}
\usepackage{moreverb}
\usepackage{pifont}
\usepackage{pst-all}
\usepackage[left=2cm,right=2cm,top=2cm,bottom=2cm]{geometry}
\usepackage{esint}
\usepackage{fourier-orns}
\usepackage{mathrsfs}
\usepackage{array}
\newcommand{\T}{\mathbb{T}}
\usepackage{hyperref}
\newcommand{\ii }{{\rm i} }
\usepackage{authblk}
\newcolumntype{C}[1]{>{\centering\arraybackslash}b{#1}}
\newcolumntype{R}[1]{>{\raggedleft\arraybackslash}b{#1}}
\newcolumntype{L}[1]{>{\raggedright\arraybackslash}b{#1}}
\newcolumntype{M}[1]{>{\centering}m{#1}}
\newtheorem{theo}{Theorem}[section]
\newtheorem{defin}{Definition}[section]
\newtheorem{lem}{Lemma}[section]
\newtheorem{prop}{Proposition}[section]

\newtheorem{remark}{Remark}[section]

\usepackage{bbm}
\numberwithin{equation}{section}

\date{}
\title{Periodic and quasi-periodic Euler-$\alpha$ flows close to Rankine vortices}
\author{Emeric Roulley\thanks{Univ Rennes, CNRS, IRMAR – UMR 6625, F-35000 Rennes, France\\
E-mail address : emeric.roulley@univ-rennes1.fr}}
\begin{document}
	\maketitle
	\begin{abstract}
		In the present contribution, we first prove the existence of $\mathbf{m}$-fold simply-connected V-states close to the unit disc for Euler-$\alpha$ equations. These solutions are implicitly obtained as bifurcation curves from the circular patches. We also prove the existence of quasi-periodic in time vortex patches close to the Rankine vortices provided that the scale parameter $\alpha$ belongs to a suitable Cantor-like set of almost full Lebesgue measure. The techniques used to prove this result are borrowed from the Berti-Bolle theory in the context of KAM for PDEs. 
	\end{abstract}
\tableofcontents
\section{Euler-$\alpha$ equations and main results}
	We consider the Euler-$\alpha$ planar model with parameter $\alpha\geqslant 0$ given by
	\begin{equation}\label{Euler alpha equations}
		\left\lbrace\begin{array}{ll}
		\partial_{t}\mathbf{u}+(\mathbf{v}\cdot\nabla) \mathbf{u}+(\nabla \mathbf{v})^{\top}\mathbf{u}+\nabla \pi=0 & \textnormal{in }\mathbb{R}_{+}\times\mathbb{R}^{2}\\
		\nabla\cdot \mathbf{v}=\nabla\cdot \mathbf{u}=0 & \\
		\mathbf{u}=\mathbf{v}-\alpha^{2}\Delta \mathbf{v} & \\
		\mathbf{v}(0,\cdot)=\mathbf{v}_0. & \\
	\end{array}\right.
\end{equation}
This model is a regularization of Euler equations describing the flow of an incompressible fluid on spatial scales larger than the length scale parameter $\alpha.$ It has been introduced in the context of averaged fluid models, see \cite{HMR98,HMR98-1}. In the literature, $\mathbf{v}$ and $\mathbf{u}$ are called the \textit{filtered} and \textit{unfiltered velocities}, respectively. Actually, $\mathbf{u}$ corresponds to the velocity field of the fluid. Notice that $(\nabla \mathbf{v})^{\top}$ denotes the transpose of the Jacobi matrix for $\mathbf{v}$, namely
$$(\nabla \mathbf{v})^{\top}=\left(\frac{\partial \mathbf{v}_j}{\partial x_i}\right)_{1\leqslant i,j\leqslant 2}$$
The pressure term $\pi$ is linked to the pressure field $p$ of the fluid through the relation
$$\pi\triangleq p-\tfrac{1}{2}|\mathbf{v}|^2-\tfrac{\alpha^2}{2}|\nabla \mathbf{v}|^2.$$
Remark that we formally recover the classical Euler model, by taking $\alpha=0$ in the above set of equations. The rigorous justification of this convergence can be found in \cite{LT10}. In the sequel, we assume that $\alpha>0.$ Let us consider the \textit{unfiltered vorticity} $\boldsymbol{\omega}$ defined by
\begin{equation}\label{def vort}
	\boldsymbol{\omega}\triangleq \nabla^{\perp}\cdot \mathbf{u}=\nabla^{\perp}\cdot(\textnormal{Id}-\alpha^2\Delta)\mathbf{v},\qquad\nabla^{\perp}\triangleq \begin{pmatrix}
		-\partial_2\\
		\partial_1
	\end{pmatrix}.
\end{equation}
Applying the operator $\nabla^{\perp}\cdot$ to the first equation in \eqref{Euler alpha equations}, we find, after straightforward computations using the divergence-free conditions, that $\boldsymbol{\omega}$ is a solution to the following active scalar equation
\begin{equation}\label{active scalar eq}
	\partial_{t}\boldsymbol{\omega}+\mathbf{v}\cdot\nabla\boldsymbol{\omega}=0.
\end{equation}
Such kind of nonlinear and nonlocal transport PDE has been widely studied during the past decade in fluid dynamics, especially regarding the periodic motions and more recently the quasi-periodic ones. We shall discuss in this introduction two types of active scalar models in the form \eqref{active scalar eq} which will be of interest in the sequel. The first example is given by the classical velocity-vorticity formulation of 2D Euler equations, namely
$$\partial_{t}\boldsymbol{\omega}^{\textnormal{\tiny{E}}}+\mathbf{v}^{\textnormal{\tiny{E}}}\cdot\nabla\boldsymbol{\omega}^{\textnormal{\tiny{E}}}=0,\qquad\mathbf{v}^{\textnormal{\tiny{E}}}=\nabla^{\perp}\boldsymbol{\Psi}^{\textnormal{\tiny{E}}},\qquad\Delta\boldsymbol{\Psi}^{\textnormal{\tiny{E}}}=\boldsymbol{\omega}^{\textnormal{\tiny{E}}}.$$
In this case, the potential velocity $\boldsymbol{\Psi}^{\textnormal{\tiny{E}}}$ is obtained as the convolution with the Green function associated to the Laplace problem set in the whole plane, namely
\begin{equation}\label{Psi Euler intro}
	\boldsymbol{\Psi}^{\textnormal{\tiny{E}}}=\mathbf{G}^{\textnormal{\tiny{E}}}\ast\boldsymbol{\omega}^{\textnormal{\tiny{E}}},\qquad\mathbf{G}^{\textnormal{\tiny{E}}}\triangleq\tfrac{1}{2\pi}\log(|\cdot|).
\end{equation}
The second example is given by quasi-geostrophic shallow-water equations with parameter $\lambda>0$, shorten in what follows into $(QGSW)_{\lambda}$, which is a geophysical asymptotic model describing the circulation of the atmosphere at large time and space scales \cite[p.220]{V17}. This model is given by
$$\partial_t\boldsymbol{q}^{\textnormal{\tiny{SW}}}+\mathbf{v}_{\lambda}^{\textnormal{\tiny{SW}}}\cdot\nabla\boldsymbol{q}^{\textnormal{\tiny{SW}}}=0,\qquad\mathbf{v}_{\lambda}^{\textnormal{\tiny{SW}}}=\nabla^{\perp}\boldsymbol{\Psi}_{\lambda}^{\textnormal{\tiny{SW}}},\qquad(\Delta-\lambda^2)\boldsymbol{\Psi}_{\lambda}^{\textnormal{\tiny{SW}}}=\boldsymbol{q}^{\textnormal{\tiny{SW}}}.$$
In this case, the potential velocity $\boldsymbol{\Psi}_{\lambda}^{\textnormal{\tiny{SW}}}$ is obtained as the convolution with the Green function associated to the Helmoltz problem set in the whole plane, namely
\begin{equation}\label{Psi QGSW intro}
	\boldsymbol{\Psi}_{\lambda}^{\textnormal{\tiny{SW}}}=\mathbf{G}_{\lambda}^{\textnormal{\tiny{SW}}}\ast\boldsymbol{q}^{\textnormal{\tiny{SW}}},\qquad\mathbf{G}_{\lambda}^{\textnormal{\tiny{SW}}}\triangleq-\tfrac{1}{2\pi}K_0(\lambda|\cdot|),
\end{equation}
where $K_0$ is the modified Bessel function of second kind. We refer to the Appendix \ref{appendix Bessel} for a presentation of modified Bessel functions and some of their useful properties. The parameter $\lambda$ is called \textit{Rossby deformation length} or \textit{Rossby radius} and quantifies the rotation/stratification balance for the fluid. Few results are known regarding this model and we may refer to \cite{DHR19,R21} for the mathematical context of interest in the sequel and to \cite{DJ20,DP11} for interesting numerical simulations in the physical literature.\\

This paper is based on the following fundamental relation, which can be found for instance in \cite{BLT08,OS99}, giving a nice expression of the Euler-$\alpha$ velocity field $\mathbf{v}$. 
\begin{equation}\label{Psi Ea intro}
	\mathbf{v}=\nabla^{\perp}\boldsymbol{\Psi},\qquad\boldsymbol{\Psi}\triangleq\mathbf{G}\ast\boldsymbol{\omega},\qquad\mathbf{G}\triangleq\mathbf{G}^{\textnormal{\tiny{E}}}-\mathbf{G}_{\frac{1}{\alpha}}^{\textnormal{\tiny{SW}}}.
\end{equation}
Due to the importance of this formula for our work, we provide in Section \ref{sec Ea dyn} its detailed justification. In our approach, we take advantage of this explicit link with Euler and QGSW equations to prove the existence of periodic and quasi-periodic vortex patch motions for the Euler-$\alpha$ equations. In particular, we shall make appeal to the computations carried out in \cite{BHM21,B82,DHR19,HR21-1,HMV13,HR21} relatively to these contexts. For the Euler-$\alpha$ model, the global existence and uniqueness of classical solutions has been obtained in \cite{B99}. The well-posedness of weak solutions in the space of bounded Radon measures for \eqref{Euler alpha equations} has been proved in \cite{OS99}. Indeed, the formula \eqref{Psi Ea intro} and \eqref{exp K0} imply that the vector field $\mathbf{v}$ is less singular than the Euler one, which allows to reach the class of measures for the well-posedness. An other consequence is that the Yudovitch theory also applies in this context and the weak solutions are Lagrangian. This fact is the starting point to be able to look for vortex patches. Before presenting the notion of vortex patch solutions together with the main results of this study, we may end this presentation by briefly mentioning that the 2D Euler-$\alpha$ model has also been intensively studied in the case of subset of $\mathbb{R}^{2}$ domains with suitable boundary conditions. We may for instance refer the reader to \cite{BI17,BIFL20,BIFL22,FLTZ15,XZ21}.\\

Now, we shall present the notion of vortex patches and discuss some historical background about the periodic and quasi-periodic settings. The techniques used to find these two kind of solutions are completely different. The periodic solutions are obtained by using bifurcation theory from the stationary solutions and the quasi-periodic ones appear in the context of KAM/Nash-Moser theories. Considering an open bounded simply-connected domain $D_0$ in $\mathbb{R}^2$ then the function $\mathbf{1}_{D_0}$ provides an initial datum in the Yudovitch class. Therefore,
we have existence and uniqueness of a global weak solution which, according to the transport structure of \eqref{active scalar eq}, takes the following form
$$\boldsymbol{\omega}(t,\cdot)=\mathbf{1}_{D_t},$$
where $D_t$ is the image of $D_0$ by the flow map associated to the velocity field $\mathbf{v}$, namely
$$D_t=\boldsymbol{\Phi}_t(D_0),\qquad\boldsymbol{\Phi}_t(x)=x+\int_{0}^{t}\mathbf{v}\big(s,\boldsymbol{\Phi}_s(x)\big)ds.$$
The resulting solution $t\mapsto\mathbf{1}_{D_t}$ is called a \textit{vortex patch} with initial patch $\mathbf{1}_{D_0}.$ Given any parametrization $z(t,\cdot):\mathbb{T}\rightarrow\partial D_t$ of the boundary of the patch at time $t,$ it is well-known since the works \cite{HMV13,HMV15} that for an active scalar equation, as in our case, the particles on the boundary move with the flow and remain at the boundary. Hence, at least in the smooth case, one can write the following equation called \textit{vortex patch equation}.
\begin{equation}\label{Lag-For}
	\Big(\partial_tz(t,\theta)-\mathbf{v}\big(t,z(t,\theta)\big)\Big)\cdot\mathbf{n}\big(t,z(t,\theta)\big)=0,
\end{equation}
where $\mathbf{n}(t,z(t,\theta))$ is the outward normal vector to the boundary $\partial D_{t}$ of $D_{t}$ at the point $z(t,\theta)$. Since the dynamics is planar, it is more convenient to use the complex notation. In particular, the Euclidean structure of $\mathbb{R}^{2}$ is transposed into the complex world in the following way :
$$\forall u=a+\ii b\in\mathbb{C},\quad \forall v=a'+\ii b'\in\mathbb{C},\qquad u\cdot v\triangleq \langle u,v\rangle_{\mathbb{R}^{2}}=\mbox{Re}\left(u\overline{v}\right)=aa'+bb'.$$
Remarking that $\mathbf{n}(t,z(t,\theta))$ is real-proportional to $\ii\partial_{\theta}z(t,\theta)$, then we get the complex form of the contour dynamics motion \eqref{Lag-For}, 
\begin{equation}\label{complex vp eq}
	\mbox{Im}\Big(\partial_{t}z(t,\theta)\overline{\partial_{\theta}z(t,\theta)}\Big)=\mbox{Im}\Big(\mathbf{v}(t,z(t,\theta))\overline{\partial_{\theta}z(t,\theta)}\Big).
\end{equation}

\noindent $\blacktriangleright$ \textbf{Uniformly rotating solutions.}\\
Uniformly rotating vortex patches, also called \textit{V-states}, form a particular subclass of vortex patch solutions taking the form
\begin{equation}\label{unif rot dom}
	D_t=e^{\ii\Omega t}D_0.
\end{equation}
Any solution of this form is rotating with a time independent angular velocity $\Omega\in\mathbb{R}$ around its center of mass fixed at the origine. If $\Omega=0,$ then it is stationary, otherwise it is time periodic with period $\tfrac{2\pi}{\Omega}.$ Historically, the first example of V-states was provided by Kirchhoff \cite{K74} who showed that for Euler equations any ellipse with semi-axis $a$ and $b$ uniformly rotates for an angular velocity $\Omega=\tfrac{ab}{(a+b)^2}.$ Observe from \eqref{Psi Ea intro}, \eqref{Psi Euler intro} and \eqref{Psi QGSW intro} that the stream function $\boldsymbol{\Psi}$ admits a radial Green kernel. Therefore, any radial initial profile would generate a stationary solution. In particular, the Rankine vortices given by the discs provide such examples. The purpose of this study is to prove the existence of periodic (and quasi-periodic) vortex patch structures close to these equilibrium states. Due to the invariance by dilation of the model, it is sufficient to search for these type of solutions close to the unit disc. The analytical study of existence of V-states close to the unit disc for Euler equations goes back to the work of Burbea \cite{B82}. We also refer to \cite{HMV13} for a more recent and rigorous point of view. Indeed, combining Crandall-Rabinowitz's Theorem \ref{Crandall-Rabinowitz theorem} and complex analysis, one can find the existence of branches of $\mathbf{m}$-fold V-states bifurcating from the unit disc at the angular velocities
\begin{equation}\label{Burbea frequencies}
	\boldsymbol{\Omega}_{\mathbf{m}}^{\textnormal{\tiny{E}}}\triangleq\frac{\mathbf{m}-1}{2\mathbf{m}}.
\end{equation}
Later on, a lot of attention has been paid to such type of solutions for different nonlinear transport fluid models like Euler equations in the plane or in the unit disc, generalised surface quasi-geostrophic equations $(SQG)_{\alpha}$ and $(QGSW)_{\lambda}$ equations. Also, several topological and regularity settings were explored and we may refer to
\cite{CCG16,DHR19,G20,G21,G19, GPSY20,HH15,HH21,HHH18,HHHM15,HMW20,HW21,HHMV16,HM16, HM16-2,HM17,HMV13,HMV15,R17,R21}. Nevertheless, it seems that the Euler-$\alpha$ model has been set aside from these studies. Hence, we propose here to study the emergence of simply-connected V-states for this model. For this task, due to the decomposition \eqref{Psi Ea intro}, we may emphasize the result in \cite{DHR19}, where they found the following angular velocities related to the modified Bessel functions for which the bifurcation of simply-connected V-states from the unit disc occurs
\begin{equation}\label{DHR frequencies}
	\boldsymbol{\Omega}_{\mathbf{m}}^{\textnormal{\tiny{SW}}}(\lambda)\triangleq I_1(\lambda)K_1(\lambda)-I_{\mathbf{m}}(\lambda)K_{\mathbf{m}}(\lambda).
\end{equation}
Now, we shall present the first theorem proved in this study and dealing with periodic simply-connected patches bifurcating from the Rankine vortices.
\begin{theo}\label{thm Ea sc}
	Let $\alpha>0$ and $\beta\in(0,1).$ For any $\mathbf{m}\in\mathbb{N}^*,$ there exists a branch of $\mathbf{m}$-fold V-states of regularity $C^{1+\beta}$ bifurcating from the unit disc at the angular velocity
		\begin{equation}\label{def eigenvalues Easc}
			\boldsymbol{\Omega}_{\mathbf{m}}^{\textnormal{\tiny{E}}}(\alpha)\triangleq\frac{\mathbf{m}-1}{2\mathbf{m}}-\Big(I_1\left(\tfrac{1}{\alpha}\right)K_1\left(\tfrac{1}{\alpha}\right)-I_{\mathbf{m}}\left(\tfrac{1}{\alpha}\right)K_{\mathbf{m}}\left(\tfrac{1}{\alpha}\right)\Big)=\boldsymbol{\Omega}_{\mathbf{m}}^{\textnormal{\tiny{E}}}-\boldsymbol{\Omega}_{\mathbf{m}}^{\textnormal{\tiny{SW}}}\left(\tfrac{1}{\alpha}\right)
		\end{equation}
for Euler-$\alpha$ equations \eqref{Euler alpha equations}.
\end{theo}
\begin{remark}
	\begin{enumerate}[label=(\roman*)]
		\item Observe that the bifurcation frequencies \eqref{def eigenvalues Easc} are exactly the superposition of those of Euler \eqref{Burbea frequencies} and $(QGSW)_{\frac{1}{\alpha}}$ \eqref{DHR frequencies} models which is in accordance with the remark in Section \ref{sec Ea dyn} and the structure of the vortex patch equation \eqref{complex vp eq}.
		\item The following convergences hold true.
		$$\boldsymbol{\Omega}_{\mathbf{m}}^{\textnormal{\tiny{E}}}(\alpha)\underset{\mathbf{m}\rightarrow\infty}{\longrightarrow}\tfrac{1}{2}-I_1\left(\tfrac{1}{\alpha}\right)K_1\left(\tfrac{1}{\alpha}\right)\triangleq \boldsymbol{\Omega}_{\infty}^{\textnormal{\tiny{E}}}(\alpha),\qquad\boldsymbol{\Omega}_{\mathbf{m}}^{\textnormal{\tiny{E}}}(\alpha)\underset{\alpha\rightarrow0}{\longrightarrow}\tfrac{\mathbf{m}-1}{2\mathbf{m}}=\boldsymbol{\Omega}_{\mathbf{m}}^{\textnormal{\tiny{E}}}.$$
		\item The case $\mathbf{m}=1$ corresponds to a translation of the trivial solution, see Remark \ref{rem triv sol}.
	\end{enumerate}
\end{remark}
The Theorem \ref{thm Ea sc} is proved by using Crandall-Rabinowitz's Theorem \ref{Crandall-Rabinowitz theorem} in the spirit of the previous works mentioned above. For that purpose, we reformulate the vortex patch equation \eqref{complex vp eq} in the uniformly rotating framework with conformal mappings. We chose to use the conformal functional setting in order to be able to take advantage of the computations already done in \cite{DHR19,HMV13}. Nevertheless, one could also use the polar parametrization similarly to the next result on quasi-periodic solutions. Introducing, for an initial domain $D_0$ close to the unit disc $\mathbb{D}$, the conformal parametrization $\Phi:\mathbb{C}\setminus\overline{\mathbb{D}}\rightarrow\mathbb{C}\setminus\overline{D_0}$ in the form
$$\Phi(z)=z+f(z),\qquad f(z)=\sum_{n=0}^{\infty}\frac{a_n}{z^n},\qquad a_n\in\mathbb{R},$$
we can reformulate the vortex patch equation \eqref{complex vp eq} in the uniformly rotating context as the following equation
$$\forall w\in\mathbb{T},\quad F_{\alpha}(\Omega,f)(w)=0,\qquad F_{\alpha}(\Omega,f)(w)\triangleq \textnormal{Im}\left\lbrace\Big(\Omega\Phi(w)+I^{\textnormal{\tiny{E}}}(f)(w)+I^{\textnormal{\tiny{SW}}}(f)(\alpha,w)\Big)\overline{w}\overline{\Phi'(w)}\right\rbrace,$$
with 
$$I^{\textnormal{\tiny{E}}}(f)(w)\triangleq \fint_{\mathbb{T}}\Phi'(\tau)\log\big(|\Phi(w)-\Phi(\tau)|\big)d\tau,\qquad I^{\textnormal{\tiny{SW}}}(f)(\alpha,w)\triangleq \fint_{\mathbb{T}}\Phi'(\tau)K_{0}\left(\tfrac{1}{\alpha}|\Phi(\tau)-\Phi(w)|\right)d\tau.$$
Remarking that the disc provides a line of solutions $F_{\alpha}(\Omega,0)=0$ for $\Omega\in\mathbb{R}$, we shall apply the local bifurcation theory to find non-trivial solutions. In Proposition \ref{prop lin op Ea}, we prove that the functional $F_{\alpha}$ is of class $C^1$ with respect to some Hölder regularity spaces and its linearized operator at the equilibrium has the Fredholm property and expresses as the following Fourier multiplier
$$\forall w\in\mathbb{T},\quad d_{f}F_{\alpha}(\Omega,0)(w)=\sum_{n=0}^{\infty}a_n(n+1)\Big(\boldsymbol{\Omega}_{n+1}^{\textnormal{\tiny{E}}}(\alpha)-\Omega\Big)\textnormal{Im}(w^{n+1}).$$
The one dimensional kernel condition for applying bifurcation theory is ensured by the strict monotonicity of the sequence $\big(\boldsymbol{\Omega}_{n}^{\textnormal{\tiny{E}}}(\alpha)\big)_{n\in\mathbb{N}^*}$ which is checked using some refined estimates on modified Bessel functions, see Lemma \ref{lem mono ev Ea}. Finally, the transversality condition is a direct consequence of the Fourier decomposition of the linearized operator and is checked in Proposition \ref{prop hyp CR Ea}-(iv).\\

\noindent $\blacktriangleright$ \textbf{Quasi-periodic in time vortex patches.}\\
This topic is rather new in vortex patch dynamics and the tools used are borrowed from KAM/Nash-Moser theories \cite{A63,BB15,K54,M62}. Recall that a function $r:\mathbb{R}\rightarrow\mathbb{R}$ is said to be \textit{quasi-periodic} if there exists $d\in\mathbb{N}^*,$ $\widehat{r}=\widehat{r}(\varphi):\mathbb{T}^{d}\rightarrow\mathbb{R}$ continuous and $\omega\in\mathbb{R}^d$ such that
\begin{equation}\label{def QP}
	r(t)=\widehat{r}(\omega t),\qquad\forall l\in\mathbb{Z}^d\setminus\{0\},\quad\omega\cdot l\neq 0.
\end{equation} 
The existence of quasi-periodic vortex patches close to the unit disc for the QGSW equations have been proved in \cite{HR21}. These solutions are obtained by selecting the Rossby deformation length in a massive Cantor set. Similarly, in \cite{HHM21}, the authors used the parameter inside the equations to generate quasi-periodic vortex patches near the Rankine vortices for $(SQG)_{\alpha}$ equations for suitable selected values of $\alpha.$ Here our work follows the same idea relying on the free parameter $\alpha$ to obtain these solutions. We mention two other works recently obtained for Euler equations. The first one concerns the quasi-periodic patches close to the Rankine vortices in the unit disc presented in \cite{HR21-1}. The second result, which can be found in \cite{BHM21}, is relative to quasi-periodic patches close to the Kirchhoff ellipses. In both cases, this is a geometrical parameter which, when taken among a Cantor-like set of admissible parameters, allows to find quasi-periodic solutions. Our second result reads as follows.

\begin{theo}\label{thm QPS Ea}
		Let $0<\alpha_0<\alpha_1$ and  $\mathbb{S}\subset\mathbb{N}^*.$ There exists  $\varepsilon_{0}\in(0,1)$ small enough such that for every choice of amplitudes ${\mathtt{a}}=(\mathtt{a}_{j})_{j\in\mathbb{S}}\in(\mathbb{R}_{+}^{*})^{|\mathbb{S}|}$ enjoying the smallness condition
		$$|{\mathtt{a}}|\leqslant\varepsilon_{0},$$ 
		there exists a Cantor-like set $\mathscr{C}_{\infty}\subset(\alpha_{0},\alpha_{1})$ with asymptotically full Lebesgue measure as ${\mathtt{a}}\rightarrow 0,$ i.e.
		$$\lim_{{\mathtt{a}}\rightarrow 0}|\mathscr{C}_{\infty}|=\alpha_{1}-\alpha_{0},$$
		such that for any $\alpha\in\mathscr{C}_{\infty}$, the  equation \eqref{active scalar eq} admits a time quasi-periodic vortex patch solution
		with diophantine frequency vector ${\omega}_{\tiny{\textnormal{pe}}}(\alpha,{\mathtt{a}})\triangleq (\omega_{j}(\alpha,{\mathtt{a}}))_{j\in\mathbb{S}}\in\mathbb{R}^{|\mathbb{S}|}$ and taking the form
		\begin{align}
			\boldsymbol{\omega}(t,\cdot)=\mathbf{1}_{D_t},\qquad D_t&=\Big\{\ell e^{\ii\theta},\quad\theta\in[0,2\pi],\quad 0\leqslant \ell\leqslant R(t,\theta)\Big\},\qquad R(t,\theta)=\sqrt{1+2r(t,\theta)},\label{ansatz thm}\\
			r(t,\theta)&=\sum_{j\in\mathbb{S}}{\mathtt{a}_{j}}\cos\big(j\theta+\omega_{j}(\alpha,{\mathtt{a}})t\big)+\mathtt{p}\big({\omega}_{\tiny{\textnormal{pe}}}(\alpha,\mathtt{a})t,\theta\big).\nonumber
		\end{align}
		The diophantine frequency vector satisfies the following asymptotic
		$${\mathtt{\omega}}_{\tiny{\textnormal{pe}}}(\alpha,{\mathtt{a}})\underset{{\mathtt{a}}\rightarrow 0}{\longrightarrow}(-\Omega_{j}^{\textnormal{\tiny{E}}}(\alpha))_{j\in\mathbb{S}},$$
		where $\Omega_{j}^{\textnormal{\tiny{E}}}(\alpha)$ are the equilibrium frequencies defined by
		$$\Omega_{j}^{\textnormal{\tiny{E}}}(\alpha)\triangleq j\Big(\Omega+\boldsymbol{\Omega}_{|j|}^{\textnormal{\tiny{E}}}(\alpha)\Big),$$
		with $\boldsymbol{\Omega}_{|j|}^{\textnormal{\tiny{E}}}(\alpha)$ as in \eqref{def eigenvalues Easc}. In addition, the perturbation $\mathtt{p}:\T^{|\mathbb{S}|+1}\to\mathbb{R}$ is an even function satisfying for some large index of regularity $s,$
		$$\|\mathtt{p}\|_{H^{{s}}(\mathbb{T}^{|\mathbb{S}|+1},\mathbb{R})}\underset{{\mathtt{a}}\rightarrow 0}{=}o(|{\mathtt{a}}|).$$
\end{theo}
Now we shall briefly mention the key steps of the proof of Theorem \ref{thm QPS Ea} which are similar to the ones of \cite{HHM21,HR21-1,HR21}. The core of the proof of Theorem \ref{thm QPS Ea} relies on Berti-Bolle theory \cite{BB15} and \cite[Sec. 6]{HHM21}. We mention here some results which have been obtained using this theory \cite{BBMH18,BBM16,BBM16-1,BFM21,BFM21-1,BKM21,BM18,FG20,FGP18}. Plugging the ansatz \eqref{ansatz thm} into the vortex patch equation \eqref{complex vp eq} provides a nonlinear Hamiltonian transport PDE for the radial deformation $r$ in the form
\begin{equation}\label{eq r intro}
	\partial_{t}r=\partial_{\theta}\nabla\mathscr{H}(r),
\end{equation}
with Hamiltonian $\mathscr{H}$ related to the kinetic energy and the angular momentum. This equation satisfies the reversibility property, namely if $(t,\theta)\mapsto r(t,\theta)$ is a solution, then so is $(t,\theta)\mapsto r(-t,-\theta).$ After a rescalling $r\mapsto\varepsilon r$ the equation \eqref{eq r intro} can be seen as a quasilinear perturbation of its linearization at the equilibrium ($\varepsilon=0$), see \eqref{pertham}. As mentioned in Lemmata \ref{lem lin eq QP Ea} and \ref{lemsoleqEa}, this latter is an integrable Hamiltonian system, namely for any finite set of Fourier modes $\mathbb{S}\subset\mathbb{N}^*$ of cardinal $d\in\mathbb{N}^*$, the linearized equation at the Rankine patch admits, for almost every $\alpha$ in $(\alpha_0,\alpha_1),$ reversible quasi-periodic solutions with $d$-dimensional frequency vector $-\omega_{\textnormal{Eq}}(\alpha)\triangleq\big(-\Omega_{j}^{\textnormal{\tiny{E}}}(\alpha)\big)$ in the form
$$r(t,\theta)=\sum_{j\in\mathbb{S}}r_j\cos\big(j\theta-\Omega_j^{\textnormal{\tiny{E}}}(\alpha)t\big),\qquad r_j\in\mathbb{R}^*.$$
The property \eqref{def QP} for $\omega_{\textnormal{Eq}}(\alpha)$ holds for almost every $\alpha$ by imposing some Diophantine conditions. The measure of the corresponding set is a consequence of the transversality condition in Lemma \ref{lem trsvrslt Ea}-(i) together with the Rüssmann Lemma \ref{lem Russmann measure}. Notice that the transversality condition is itself a consequence of the non-degeneracy of the function $\alpha\mapsto\omega_{\textnormal{Eq}}(\alpha)$ on $[\alpha_0,\alpha_1],$ see Lemma \ref{lem non-deg Ea}. The introduction of the Diophantine conditions imply the invertibility of the linearized operator at the equilibrium state but with a fixed loss of derivatives. Therefore, to find quasi-periodic solutions for the nonlinear model, we may apply a Nash-Moser scheme. For this aim we reformulate the problem in terms of embedded tori. This is done by splitting the phase space $L_0^2(\mathbb{T})$ in \eqref{phase space} associated with \eqref{eq r intro} into tangential $L_{\overline{\mathbb{S}}}$ and normal $L_\perp^2$ subspaces, see \eqref{ph sp splt}. Then we introduce in \eqref{action-angle var} the action-angle variables $(I,\vartheta)$ by using a symplectic polar change of coordinates for the Fourier coefficients on the tangential subspace. Thus, any function $r\in L^2\big(\mathbb{T}_{\varphi}^d,L_0^2(\mathbb{T}_{\theta})\big)$ can be related to an embedded torus, namely
$$r(\varphi,\cdot)=\mathtt{A}\big(i(\varphi)\big),\qquad i:\begin{array}[t]{rcl}
	\mathbb{T}^d & \rightarrow & \mathbb{T}^d\times\mathbb{R}^d\times L_\perp^2\\
	\varphi & \mapsto & \big(\vartheta(\varphi),I(\varphi),z(\varphi)\big)
\end{array},\qquad\mathtt{A}:\mathbb{T}^d\times\mathbb{R}^d\times L_\perp^2\rightarrow L_0^2(\mathbb{T}).$$
Then, the search of reversible quasi-periodic solutions to \eqref{eq r intro} is equivalent to looking for reversible tori solutions to
\begin{equation}\label{scrF intro}
	\mathscr{F}(i,\kappa ,\alpha,\omega,\varepsilon)=0,
\end{equation}
where $\mathscr{F}$ is a nonlinear functional whose complete expression can be found in \eqref{def scrF}. This provides a true solution of the original problem for the particular value
\begin{equation}\label{rigid ttc}
	\kappa =-\omega_{\textnormal{Eq}}(\alpha).
\end{equation}
To find non-trivial solutions of \eqref{scrF intro}, we apply a Nash-Moser scheme. At each step, we shall find an approximate right inverse, with nice tame estimates, of the linearized operator $d_{i,\kappa }\mathscr{F}(i_0)$ at the reversible torus $i_0$ of the current step. Applying the Berti-Bolle theory \cite{BB15} and \cite[Sec. 6]{HHM21}, one can conjugate the linearized operator $d_{i,\kappa }\mathscr{F}(i_0)$ by a suitable diffeomorphism of the toroidal phase space $\mathbb{T}^d\times\mathbb{R}^d\times L_{\perp}^{2}$ in order to obtain a triangular system in the action-angle-normal variables up to nice error terms. Then to solve the triangular system, it is sufficient to find an almost approximate right inverse for the linearized operator in the normal directions denoted $\mathscr{L}_{\perp}$ and related to the linearized operator $\mathcal{L}_{\varepsilon r}$ of \eqref{eq r intro} through the relation \eqref{hLom-2}. To do so, we conjugate $\mathscr{L}_{\perp}$ to a constant coefficients operator up to error terms. This is the content of the Section \ref{sec cons aai}. First, in Proposition \ref{prop reduc trans} we use a KAM reduction process with quasi-periodic symplectic change of variables to reduce the transport part of $\mathcal{L}_{\varepsilon r}.$ Then, in Proposition \ref{prop projnor} we look at the localization effects in the normal directions to recover the reduction of the transport part for $\mathscr{L}_{\perp}.$ This provides a diagonal operator of order 1 plus a remainder of order $(0,-1)$ in the variables $(\varphi,\theta).$ Finally, in Proposition \ref{prop redrem} we use a KAM reduction process in the Toeplitz operators class to reduce the remainder term. Each KAM reduction occurs for suitable restrictions of the parameters $(\alpha,\omega)$ to a Cantor-like set. In addition, the inversion of the final diagonal operator also requires an extraction of parameters. The Nash-Moser process constructs a non-trivial solution 
$(\alpha,\omega)\mapsto(i_\infty(\alpha,\omega),\kappa _{\infty}(\alpha,\omega))$ with reversible torus $i_{\infty}$ modulo restriction of the parameters to a Cantor set in $(\alpha,\omega)$ obtained gathering all the restrictions of all the steps for the contruction of the almost approximate right inverses of the linearized operators. Then, coming back to \eqref{rigid ttc}, we rigidify the frequency $\omega$ into $\omega(\alpha,\varepsilon)$ so that
\begin{equation}\label{rigid ttc-bis}
	\kappa _{\infty}\big(\alpha,\omega(\alpha,\varepsilon)\big)=-\omega_{\textnormal{Eq}}(\alpha).
\end{equation}
We mention that the introduction of the free-parameter $\kappa $ was required to apply the Berti-Bolle theory along the scheme, more precisely to invert the triangular system. The condition \eqref{rigid ttc-bis} gives a final Cantor set in $\alpha$ only that we must check it is not empty. This latter fact is obtained in Proposition \ref{prop meas Cant} by estimating the Lebesgue measure of the set through the Rüssmann Lemma \ref{lem Russmann measure} together with the perturbed transversality conditions.
\section{The stream function associated to the filtered velocity}\label{sec Ea dyn}
This short section is devoted to the justification of the formula \eqref{Psi Ea intro} which is the fundamental relation for this work. Recall that the identity \eqref{Psi Ea intro} was already observable in the literature, for instance in \cite{BLT08,OS99}. The divergence-free property for $\mathbf{v}$ and $\mathbf{u}$  in \eqref{Euler alpha equations} implies the existence of stream functions $\boldsymbol{\Psi}$ and $\psi$ such that
$$\mathbf{v}=\nabla^{\perp}\boldsymbol{\Psi},\qquad \mathbf{u}=\nabla^{\perp}\psi$$
and the goal of this section is to find a nice expression for the velocity potential $\boldsymbol{\Psi}$. Notice that, according to the third equation in \eqref{Euler alpha equations}, the stream functions are linked through the relation
$$\psi=(1-\alpha^2\Delta)\boldsymbol{\Psi}.$$
Then, \eqref{def vort} implies
$$\boldsymbol{\omega}=\Delta\psi=(1-\alpha^2\Delta)\Delta\boldsymbol{\Psi}.$$
We deduce from \eqref{Psi Euler intro} and \eqref{Psi QGSW intro} that 
$$\boldsymbol{\Psi}=\tfrac{-1}{\alpha^2}\mathbf{G}^{\textnormal{\tiny{E}}}\ast \mathbf{G}_{\frac{1}{\alpha}}^{\textnormal{\tiny{SW}}}\ast\boldsymbol{\omega}\triangleq\mathbf{G}\ast\boldsymbol{\omega}.$$
Now our goal is to compute $\mathbf{G}.$ One can write
\begin{align*}
	\mathbf{G}(x)&=\tfrac{1}{4\pi^2\alpha^2}\int_{\mathbb{R}^2}K_{0}\left(\tfrac{|y|}{\alpha}\right)\log(|x-y|)\,dA(y)\\
	&=\tfrac{1}{8\pi^2\alpha^2}\int_{\mathbb{R}^2}K_{0}\left(\tfrac{|y|}{\alpha}\right)\log\big(|x-y|^2\big)dA(y)\\
	&=\tfrac{1}{8\pi^2\alpha^2}\int_{\mathbb{R}^2}K_{0}\left(\tfrac{|y|}{\alpha}\right)\log\left(|x|^2+|y|^2-2x\cdot y\right)dA(y),
\end{align*}
where $dA$ denotes the planar Lebesgue measure. Writing $x=(R\cos(\theta),R\sin(\theta))$, then a polar change of variables yields
\begin{align*}
	\mathbf{G}(x)&=\tfrac{1}{8\pi^2\alpha^2}\int_{0}^{\infty}\int_{0}^{2\pi}rK_{0}\left(\tfrac{r}{\alpha}\right)\log\left(R^2+r^2-2Rr\cos(\theta-\eta)\right) \,drd\eta\\
	&=\tfrac{\log(R)}{2\pi\alpha}\int_{0}^{\infty}\tfrac{r}{\alpha}K_0\left(\tfrac{r}{\alpha}\right)dr+\tfrac{1}{8\pi^2\alpha}\int_{0}^{\infty}\tfrac{r}{\alpha}K_0\left(\tfrac{r}{\alpha}\right)\mathtt{I_{P}}\left(\tfrac{r}{R}\right)dr,
\end{align*}
where $\mathtt{I_{P}}$ denotes the Poisson integral defined by
$$\forall x\in\mathbb{R},\quad \mathtt{I_{P}}(x)=\int_{-\pi}^{\pi}\log\big(1+x^2-2x\cos(\eta)\big)d\eta.$$  
It is well-known, see for instance \cite{C02}, that the Poisson integral admits the following explicit formula
$$\left\lbrace\begin{array}{ll}
	\mathtt{I_{P}}(x)=0 & \textnormal{if }|x|\leqslant1\\
	\mathtt{I_{P}}(x)=4\pi\log(|x|) & \textnormal{if }|x|>1.
\end{array}\right.$$
Therefore,
\begin{align*}
	\mathbf{G}(x)&=\tfrac{\log(R)}{2\pi\alpha}\int_{0}^{\infty}\tfrac{r}{\alpha}K_0\left(\tfrac{r}{\alpha}\right)dr+\tfrac{1}{2\pi\alpha}\int_{R}^{\infty}\tfrac{r}{\alpha}K_0\left(\tfrac{r}{\alpha}\right)\log\left(\tfrac{r}{R}\right)dr\\
	&=\tfrac{\log(R)}{2\pi\alpha}\int_{0}^{R}\tfrac{r}{\alpha}K_0\left(\tfrac{r}{\alpha}\right)dr+\tfrac{1}{2\pi\alpha}\int_{R}^{\infty}\tfrac{r}{\alpha}K_0\left(\tfrac{r}{\alpha}\right)\log\left(r\right)dr.
\end{align*}
A change of variables together with \eqref{Bessel and anti-derivatives}, \eqref{asymp small arg} and \eqref{asymp large z} give
$$\tfrac{\log(R)}{2\pi\alpha}\int_{0}^{R}\tfrac{r}{\alpha}K_0\left(\tfrac{r}{\alpha}\right)dr=\tfrac{\log(R)}{2\pi}\int_{0}^{\frac{R}{\alpha}}uK_0(u)du=\tfrac{\log(R)}{2\pi}\Big(1-\tfrac{R}{\alpha}K_{1}\left(\tfrac{R}{\alpha}\right)\Big)$$
and
\begin{align*}
	\tfrac{1}{2\pi\alpha}\int_{R}^{\infty}\tfrac{r}{\alpha}K_0\left(\tfrac{r}{\alpha}\right)\log\left(r\right)dr&=\tfrac{\log(\alpha)}{2\pi}\int_{\frac{R}{\alpha}}^{\infty}uK_0\left(u\right)du+\tfrac{1}{2\pi}\int_{\frac{R}{\alpha}}^{\infty}uK_0\left(u\right)\log\left(u\right)dr\\
	&=\tfrac{R\log(\alpha)}{2\pi\alpha}K_1\left(\tfrac{R}{\alpha}\right)+\tfrac{1}{2\pi}\int_{\frac{R}{\alpha}}^{\infty}uK_0\left(u\right)\log\left(u\right)du.
\end{align*}
An integration by parts in the last integral together with \eqref{asymp large z} lead to
\begin{align*}
	\tfrac{1}{2\pi}\int_{\frac{R}{\alpha}}^{\infty}uK_0\left(u\right)\log\left(u\right)du&=\tfrac{R}{2\pi\alpha}K_1\left(\tfrac{R}{\alpha}\right)\log\left(\tfrac{R}{\alpha}\right)+\tfrac{1}{2\pi}\int_{\frac{R}{\alpha}}^{\infty}K_1\left(u\right)du\\
	&=\tfrac{R}{2\pi\alpha}K_1\left(\tfrac{R}{\alpha}\right)\log\left(R\right)-\tfrac{R\log\left(\alpha\right)}{2\pi\alpha}K_1\left(\tfrac{R}{\alpha}\right)+\tfrac{1}{2\pi}K_0\left(\tfrac{R}{\alpha}\right).
\end{align*}
Gathering the foregoing computations and reminding that $R=|x|$ we obtain
\begin{equation}\label{dec E+qgsw}
	\mathbf{G}(x)=\tfrac{1}{2\pi}\log(|x|)+\tfrac{1}{2\pi}K_0\left(\tfrac{|x|}{\alpha}\right)=\mathbf{G}^{\textnormal{\tiny{E}}}(x)-\mathbf{G}_{\frac{1}{\alpha}}^{\textnormal{\tiny{SW}}}(x).
\end{equation}
Notice that we recover a radial function as a convolution of two radial functions. In particular, any radial profil gives a stationary solution for the active scalar equation \eqref{active scalar eq} satisfied by $\boldsymbol{\omega}.$
\section{Periodic rigid motion}\label{sec per}
This section is devoted to the proof of the Theorem \ref{thm Ea sc} which is an application of the Crandall-Rabinowitz's Theorem \ref{Crandall-Rabinowitz theorem}. This latter is applied to a reformulation of the vortex patch equation \eqref{complex vp eq} in the uniformly rotating framework.
\subsection{Function spaces and reformulation of the vortex patch equation}\label{sec funct Ea}
The goal of this subsection is to set up the contour dynamics equation for the Euler-$\alpha$ V-states near the Rankine vortices. Consider an ansatz \eqref{unif rot dom} with simply-connected domain $D_0$ and rotating uniformly with some angular velocity $\Omega.$ At time $t\geqslant 0,$ the boundary $\partial D_t$ can be parametrized by
\begin{equation}\label{prmz rotating boundary}
	z(t,\theta)=e^{\ii\Omega t}z(0,\theta),\qquad\theta\in[0,2\pi],\qquad z(0,0)=z(0,2\pi).
\end{equation}
With this parametrization, the left hand side of \eqref{complex vp eq} becomes
\begin{equation}\label{simpl vp eq 1}
	\mbox{Im}\Big(\partial_{t}z(t,\theta)\overline{\partial_{\theta}z(t,\theta)}\Big)=\Omega\,\mbox{Re}\Big( z(0,\theta)\overline{\partial_{\theta}z(0,\theta)}\Big).
\end{equation}
As for the computation of the right hand-side of \eqref{complex vp eq}, it is obtained from the Biot-Savart law. Note that the velocity potential $\boldsymbol{\Psi}$ given by \eqref{Psi Ea intro} writes in this context
$$\boldsymbol{\Psi}(t,z)=\tfrac{1}{2\pi}\int_{D_t}\log(|z-\xi|)dA(\xi)+\tfrac{1}{2\pi}\int_{D_t}K_{0}\left(\tfrac{1}{\alpha}|z-\xi|\right)dA(\xi).$$
Notice that the real notation $x\in\mathbb{R}^2$ has been replaced by the complex notation $z\in\mathbb{C}.$
To get the Biot-Savart law we shall use Stokes' Theorem in complex notation 
$$2\ii\int_{D}\partial_{\overline{\xi}}f(\xi,\overline{\xi})dA(\xi)=\int_{\partial D}f(\xi,\overline{\xi})d\xi.$$
Hence, by making the identification   $2\ii\partial_{\overline{z}}=\nabla^{\perp}$ leading to  $\mathbf{v}(t,z)=2\ii\partial_{\overline{z}}\boldsymbol{\Psi}(t,z)$ one deduces, after a regularization procedure similar to the one explained in the proof of \cite[Lem. 2.1]{HR21-1}, that
\begin{equation}\label{velocity Ealpha-1}
	\mathbf{v}(t,z)=-\frac{1}{2\pi}\int_{\partial D_{t}}\log(|z-\xi|)d\xi-\frac{1}{2\pi}\int_{\partial D_{t}}K_{0}\left(\tfrac{1}{\alpha}|z-\xi|\right)d\xi.
\end{equation}
One easily obtains from \eqref{prmz rotating boundary}, \eqref{velocity Ealpha-1} and change of variables,
\begin{align}\label{sym velocity v Ea}
	\mathbf{v}\big(t,z(t,\theta)\big)=e^{\ii\Omega t}\mathbf{v}\big(0,z(0,\theta)\big).
\end{align}
Thus, combining \eqref{sym velocity v Ea} and \eqref{prmz rotating boundary}, we obtain
\begin{equation}\label{simpl vp eq 2}
	\mbox{Im}\Big( \mathbf{v}\big(t,z(t,\theta)\big)\overline{\partial_{\theta}z(t,\theta)}\Big)=\mbox{Im}\Big( \mathbf{v}\big(0,z(0,\theta)\big)\overline{\partial_{\theta}z(0,\theta)}\Big).
\end{equation}
Gathering \eqref{simpl vp eq 1} and \eqref{simpl vp eq 2}, the equation \eqref{complex vp eq} becomes
\begin{equation}\label{vpe rot}
	\Omega\mbox{Re}\Big( z(0,\theta)\overline{\partial_{\theta}z(0,\theta)}\Big)=\mbox{Im}\Big( \mathbf{v}(0,z(0,\theta))\overline{\partial_{\theta}z(0,\theta)}\Big).
\end{equation}
Therefore, denoting $z'$ a tangent vector to the boundary $\partial D_{0}$ at the point $z,$ one gets from \eqref{velocity Ealpha-1} and \eqref{vpe rot} that
for any $z\in\partial D_{0},$
\begin{equation}\label{general boundary eq Ealpha}
	\Omega\mbox{Re}\left(z\overline{z'}\right) + \mbox{Im}\left(\left[\frac{1}{2\pi}\int_{\partial D_{0}}\log\big(|z-\xi|\big)d\xi+\frac{1}{2\pi}\int_{\partial D_{0}}K_{0}\big(\tfrac{1}{\alpha}|z-\zeta|\big)d\zeta\right]\overline{z'}\right)=0.
\end{equation}
In accordance with the previous works in the field beginning with the one of Burbea \cite{B82}, we should rewrite the equation \eqref{general boundary eq Ealpha} by using conformal mappings. For that purpose we shall now present the function spaces used throughout this first part on periodic solutions. Notice that we shall identify $2\pi$-periodic $g:\mathbb{R}\rightarrow\mathbb{C}$ functions with functions $f:\mathbb{T}\rightarrow\mathbb{C}$ defined on the torus $\mathbb{T}=\mathbb{R}/2\pi\mathbb{Z}\cong\mathbb{U}$ through the relation 
$$f(w)=g(\theta),\quad w=e^{\ii\theta}.$$
In this part, we shall consider the following notation for the mean value line integral of any continuous function $f$ defined on the torus $\mathbb{T}$
$$\fint_{\mathbb{T}}f(\tau)d\tau\triangleq \frac{1}{2\ii\pi}\int_{\mathbb{T}}f(\tau)d\tau\triangleq \frac{1}{2\pi}\int_{0}^{2\pi}f\big(e^{\ii\theta}\big)e^{\ii\theta}d\theta.$$
Now, we introduce the Hölder spaces on the unit circle. Given $\beta\in(0,1),$ we denote by $C^{\beta}(\mathbb{T})$ the space of continuous functions $f$ such that
$$\| f\|_{C^{\beta}(\mathbb{T})}\triangleq \| f\|_{L^{\infty}(\mathbb{T})}+\sup_{\underset{\tau\neq w}{(\tau,w)\in\mathbb{T}^{2}}}\frac{|f(\tau)-f(w)|}{|\tau-w|^{\beta}}<+\infty$$
and we denote by $C^{1+\beta}(\mathbb{T})$ the space of $C^{1}$ functions with $\beta$-Hölder continuous derivative such that
$$\| f\|_{C^{1+\beta}(\mathbb{T})}\triangleq \| f\|_{L^{\infty}(\mathbb{T})}+\Big\| \frac{df}{dw}\Big\|_{C^{\beta}(\mathbb{T})}<+\infty.$$
For $\beta\in(0,1),$ we set
\begin{align*}
	X^{1+\beta}&\triangleq \left\lbrace f\in C^{1+\beta}(\mathbb{T})\quad\textnormal{s.t.}\quad\forall w\in\mathbb{T},\,f(w)=\sum_{n=0}^{+\infty}f_{n}\overline{w}^{n},\,f_{n}\in\mathbb{R}\right\rbrace,\\
	Y^{\beta}&\triangleq\left\lbrace g\in C^{\beta}(\mathbb{T})\quad\textnormal{s.t.}\quad\forall w\in\mathbb{T},\,g(w)=\sum_{n=1}^{+\infty}g_{n}e_{n}(w),\,g_{n}\in\mathbb{R}\right\rbrace,\qquad e_{n}(w)\triangleq \mbox{Im}(w^{n}).
\end{align*}
Remark that the $\mathbf{m}$-fold symmetry property can be translated in the functional spaces as follows
\begin{align*}
	X_{\mathbf{m}}^{1+\beta}&\triangleq\left\lbrace f\in X^{1+\beta}\quad\textnormal{s.t.}\quad\forall w\in\mathbb{T},\,f(w)=\sum_{n=1}^{+\infty}f_{n\mathbf{m}-1}\overline{w}^{n\mathbf{m}-1}\right\rbrace,\\
	Y_{\mathbf{m}}^{\beta}&\triangleq\left\lbrace g\in Y^{\beta}\quad\textnormal{s.t.}\quad\forall w\in\mathbb{T},\,g(w)=\sum_{n=1}^{+\infty}g_{n\mathbf{m}}e_{n\mathbf{m}}(w)\right\rbrace.
\end{align*}
We shall also consider the following balls of radius $r>0$ in $X^{1+\beta}$ and $X_{\mathbf{m}}^{1+\beta}$, respectively
$$B_{r}^{1+\beta}\triangleq \Big\{ f\in X^{1+\beta}\quad\textnormal{s.t.}\quad\| f\|_{C^{1+\beta}(\mathbb{T})}<r\Big\},\qquad B_{r,\mathbf{m}}^{1+\beta}\triangleq B_{r}^{1+\beta}\cap X_{\mathbf{m}}^{1+\beta}.$$
\noindent Based on the Riemann mapping Theorem, we may parametrize the boundary of $D_0$ by considering the conformal mapping $\Phi:\mathbb{C}\setminus\overline{\mathbb{D}}\rightarrow\mathbb{C}\setminus\overline{D_0}$ given by
$$\Phi(z)\triangleq z+f(z),\qquad f(z)=\sum_{n=0}^{\infty}\frac{a_n}{z^n},\qquad a_n\in\mathbb{R}.$$
Here, $f$ is in $B_r^{1+\beta}$ for small $r$. We have $\Phi(\mathbb{T})=\partial D_0$ and we mention that the link between the regularity of the mapping and of the boundary is given by Kellogg-Warschawski's Theorem \cite{W35,P92}. For $w\in\mathbb{T},$ a tangent vector to the boundary $\partial D_0$ at the point $z=\Phi(w)$ is given by
$$\overline{z'}=-\ii\overline{w}\overline{\Phi'(w)}.$$
Plugging this into \eqref{general boundary eq Ealpha} and using the change of variables $\xi=\Phi(\tau)$ give
\begin{equation}\label{def F Ea}
	\forall w\in\mathbb{T},\quad F_{\alpha}(\Omega,f)(w)=0,\qquad F_{\alpha}(\Omega,f)(w)\triangleq \textnormal{Im}\left\lbrace\Big(\Omega\Phi(w)+I^{\textnormal{\tiny{E}}}(f)(w)+I^{\textnormal{\tiny{SW}}}(f)(\alpha,w)\Big)\overline{w}\overline{\Phi'(w)}\right\rbrace,
\end{equation}
with 
$$I^{\textnormal{\tiny{E}}}(f)(w)\triangleq \fint_{\mathbb{T}}\Phi'(\tau)\log\big(|\Phi(w)-\Phi(\tau)|\big)d\tau,\qquad I^{\textnormal{\tiny{SW}}}(f)(\alpha,w)\triangleq \fint_{\mathbb{T}}\Phi'(\tau)K_{0}\left(\tfrac{1}{\alpha}|\Phi(w)-\Phi(\tau)|\right)d\tau.$$
Observe that the functional $F_{\alpha}$ makes appear a term $I^{\textnormal{\tiny{E}}}$ associated with the Euler dynamics and a term $I^{\textnormal{\tiny{SW}}}$ corresponding to the QGSW equations.
\subsection{Regularity aspects and structure of the linearized operator}
Our next task is to study some regularity properties for the functional $F_{\alpha}$ defined in \eqref{def F Ea}, look for the structure of its linearized operator at the Rankine vortex and check some monotonicity property for its spectrum. We first remark that the linearized operator at the equilibrium state acts as a Fourier multiplier according to the functions spaces introduced in Section \ref{sec funct Ea}. More precisely, we have the following result.
\begin{prop}\label{prop lin op Ea}
	Let $\alpha>0.$ Then the following properties hold true.
	\begin{enumerate}
	\item There exists $r>0$ such that for any $\beta\in(0,1)$, the following hold true.
	\begin{enumerate}[label=(\roman*)]
		\item $F_{\alpha}:\mathbb{R}\times B_r^{1+\beta}\rightarrow Y^{\beta}$ is well-defined and of class $C^1.$
		\item For any $\mathbf{m}\in\mathbb{N}^*,$ the restriction $F_{\alpha}:\mathbb{R}\times B_{r,\mathbf{m}}^{1+\beta}\rightarrow Y_{\mathbf{m}}^{\beta}$ is well-defined.
		\item The partial derivative $\partial_{\Omega}d_fF_{\alpha}:\mathbb{R}\times B_{r}^{1+\beta}\rightarrow\mathcal{L}\left(X^{1+\beta},Y^\beta\right)$ exists and is continuous.
		\item For any $\Omega\in\mathbb{R},$ one has $F_{\alpha}(\Omega,0)=0.$
	\end{enumerate}
	\item Let $\Omega\in\mathbb{R}\setminus\big\{\tfrac{1}{2}-I_{1}\left(\tfrac{1}{\alpha}\right)K_{1}\left(\tfrac{1}{\alpha}\right)\big\}.$ Then the operator $d_{f}F_{\alpha}(\Omega,0):X^{1+\beta}\rightarrow Y^{\beta}$ is Fredholm with index $0$.\vspace{0.05cm}\\
	In addition, for $h\in X^{1+\beta}$ writing
	$$\forall w\in\mathbb{T},\quad h(w)=\sum_{n=0}^{\infty}a_n\overline{w}^{n},\quad a_n\in\mathbb{R},$$
we have
\begin{equation}\label{lin op Ea}
	\forall w\in\mathbb{T},\quad d_{f}F_{\alpha}(\Omega,0)(w)=\sum_{n=0}^{\infty}a_n(n+1)\Big(\tfrac{n}{2(n+1)}-\boldsymbol{\Omega}_{n+1}^{\textnormal{\tiny{SW}}}\left(\tfrac{1}{\alpha}\right)-\Omega\Big)e_{n+1}(w),
\end{equation}
with $\boldsymbol{\Omega}_{n}^{\textnormal{\tiny{SW}}}$ as in \eqref{DHR frequencies}.
\end{enumerate}
\end{prop}
\begin{proof}
	\textbf{1. (i)} The proof of the regularity is now classical. We refer the reader to \cite{HMV13} for the computations associated with the Euler part and to \cite{DHR19} for the computations associated with the QGSW part.\\
	\textbf{(ii)} For $f\in B_{r,\mathbf{m}}^{1+\beta},$ the following identities hold
	$$\forall w\in\mathbb{T},\qquad\Phi\left(e^{\frac{2\ii\pi}{\mathbf{m}}}w\right)=e^{\frac{2\ii\pi}{\mathbf{m}}}\Phi\left(w\right),\qquad\Phi'\left(e^{\frac{2\ii\pi}{\mathbf{m}}}w\right)=\Phi'\left(w\right).$$
	Therefore, the change of variables $\tau\mapsto e^{\frac{2\ii\pi}{\mathbf{m}}}\tau$ implies
	$$I^{\textnormal{\tiny{E}}}(f)\left(e^{\frac{2\ii\pi}{\mathbf{m}}}w\right)=e^{\frac{2\ii\pi}{\mathbf{m}}}I^{\textnormal{\tiny{E}}}(f)(w),\qquad I^{\textnormal{\tiny{SW}}}(f)\left(\alpha,e^{\frac{2\ii\pi}{\mathbf{m}}}w\right)=e^{\frac{2\ii\pi}{\mathbf{m}}}I^{\textnormal{\tiny{SW}}}(f)(\alpha,w).$$
	Consequently,
	$$\forall w\in\mathbb{T},\quad F_{\alpha}(\Omega,f)\big(e^{\frac{2\ii\pi}{\mathbf{m}}}w\big)=F_{\alpha}(\Omega,f)(w).$$
	This symmetry property proves the desired result.\\
	\textbf{(iii)} One readily has
	$$\partial_{\Omega}d_fF_{\alpha}(\Omega,f)[h](w)=\textnormal{Im}\left\lbrace\overline{h'(w)}\overline{w}\Phi(w)+h(w)\overline{w}\overline{\Phi'(w)}\right\rbrace.$$
	Hence, for $(f_{j},g_{j})\in(B_{r}^{1+\alpha})^{2}$ and $h\in C^{1+\alpha}(\mathbb{T})$, we get 
	$$\Big\|\partial_{\Omega}d_{f}F_{\alpha}(\Omega,f_{j})[h]-\partial_{\Omega}d_{f}F_{\alpha}(\Omega,g_{j})[h]\Big\|_{C^{\alpha}(\mathbb{T})}\lesssim\| f_{j}-g_{j}\|_{C^{1+\alpha}(\mathbb{T})}\| h\|_{C^{1+\alpha}(\mathbb{T})}.$$
	This proves the continuity of $\partial_{\Omega}d_{f}F_{\alpha}:\mathbb{R}\times B_{r}^{1+\alpha}\rightarrow\mathcal{L}(X^{1+\alpha},Y^{\alpha}).$\\
	\textbf{(iv)} For $w\in\mathbb{T},$ using the change of variable $\tau\mapsto w\tau$, the fact that $|w|=1$ and the definition of the line integral, we get \begin{align*}
		I^{\textnormal{\tiny{E}}}(0)(w)&=\fint_{\mathbb{T}}\log\big(|w-\tau|\big)d\tau\\
		&=w\fint_{\mathbb{T}}\log\big(|1-\tau|\big)d\tau\\
		&=\frac{w}{2\pi}\int_{0}^{2\pi}\log\big(|1-e^{\ii\theta}|\big)e^{\ii\theta}d\theta.
	\end{align*}
	Now remark that
	$$|1-e^{\ii\theta}|^{2}=2\big(1-\cos(\theta)\big)=4\sin^{2}\big(\tfrac{\theta}{2}\big).$$
	In particular, this latter quantity is even in $\theta.$ Hence, we infer
	\begin{equation}\label{IE0}
		I^{\textnormal{\tiny{E}}}(0)(w)=\frac{w}{4\pi}\int_{0}^{2\pi}\log\big(\sin^2\big(\tfrac{\theta}{2}\big)\big)\cos(\theta)d\theta=-\frac{w}{2}.
	\end{equation}
	The last identity is a consequence of the following formula which can be found for instance in \cite[Lem. A.3]{CCG16}. 
	\begin{equation}\label{int diego}
		\tfrac{1}{2\pi}\int_{0}^{2\pi}\log\big(\sin^2\left(\tfrac{\theta}{2}\right)\big)\cos(n\theta)d\theta=-\frac{1}{n}.
	\end{equation}
Similarly,
	\begin{equation}\label{ISW0}
		I^{\textnormal{\tiny{SW}}}(0)(w)=\frac{w}{2\pi}\int_{0}^{2\pi}K_0\big(\tfrac{2}{\alpha}\big|\sin\big(\tfrac{\theta}{2}\big)\big|\big)\cos(\theta)d\theta=wI_1\big(\tfrac{1}{\alpha}\big)K_1\big(\tfrac{1}{\alpha}\big).
	\end{equation}
	The last identity follows from the following identity which can be found in the proof of \cite[Lem. 3.2]{HR21}. 
	\begin{equation}\label{int Bes SW}
		\frac{1}{2\pi}\int_{0}^{2\pi}K_0\big(\tfrac{2}{\alpha}\sin\big(\tfrac{\theta}{2}\big)\big)\cos(n\theta)d\theta=I_{n}\left(\tfrac{1}{\alpha}\right)K_{n}\left(\tfrac{1}{\alpha}\right).
	\end{equation}
Inserting \eqref{IE0} and \eqref{ISW0} into \eqref{def F Ea} and using $w\overline{w}=|w|^2=1$ gives the desired result.\\
	\textbf{2.} We can write
	\begin{align*}
		d_fF_{\alpha}(\Omega,0)[h](w)&=\mathcal{I}[h](w)+\mathcal{K}[h](w),
	\end{align*}
with
\begin{align*}
	\mathcal{I}[h](w)&\triangleq\big(\Omega-\tfrac{1}{2}+I_{1}\left(\tfrac{1}{\alpha}\right)K_1\left(\tfrac{1}{\alpha}\right)\big)\textnormal{Im}\left\lbrace\overline{h'(w)}\right\rbrace,\\
	\mathcal{K}[h](w)&\triangleq\textnormal{Im}\left\lbrace\big(\Omega h(w)+d_fI^{\textnormal{\tiny{E}}}(0)[h](w)+d_fI^{\textnormal{\tiny{SW}}}(0)[h](w)\big)\overline{w}\right\rbrace.
\end{align*}
One obviously has that $\mathcal{I}:X^{1+\beta}\rightarrow Y^{\beta}$ is an isomorphism provided that $\Omega\neq\tfrac{1}{2}-I_{1}\left(\tfrac{1}{\alpha}\right)K_{1}\left(\tfrac{1}{\alpha}\right).$ One readily has
\begin{align*}
	d_{f}I^{\textnormal{\tiny{E}}}(0)[h](w)&=\fint_{\mathbb{T}}h'(\tau)\log\big(|w-\tau|\big)d\tau+\fint_{\mathbb{T}}\frac{h(w)-h(\tau)}{2\big(w-\tau\big)}d\tau+\fint_{\mathbb{T}}\frac{\overline{h(w)}-\overline{h(\tau)}}{2\big(\overline{w}-\overline{\tau}\big)}d\tau
\end{align*}
and 
\begin{align*}
	d_{f}I^{\textnormal{\tiny{SW}}}(0)[h](w)&=\fint_{\mathbb{T}}h'(\tau)K_0\big(\tfrac{1}{\alpha}|w-\tau|\big)d\tau\\
	&\quad+\fint_{\mathbb{T}}K_0'\big(\tfrac{1}{\alpha}|w-\tau|\big)\frac{\big(h(w)-h(\tau)\big)\big(\overline{w}-\overline{\tau}\big)}{2\alpha|w-\tau|}d\tau\\
	&\quad+\fint_{\mathbb{T}}K_0'\big(\tfrac{1}{\alpha}|w-\tau|\big)\frac{\big(\overline{h(w)}-\overline{h(\tau)}\big)\big(w-\tau\big)}{2\alpha|w-\tau|}d\tau.
\end{align*}
Now, combining \eqref{Bessel derivatives} and \eqref{exp K0}, we can write
$$K_{0}(z)=-\log(z)+\mathtt{F}_1(z), \qquad K_0'(z)=-K_1(z)=-\tfrac{1}{z}+\mathtt{F}_2(z),\qquad \mathtt{F}_1,\mathtt{F}_1',\mathtt{F}_2, \mathtt{F}_2'\textnormal{ bounded at }0.$$
Gathering the foregoing computations leads to
\begin{align*}
	d_{f}I^{\textnormal{\tiny{E}}}(0)[h](w)+d_{f}I^{\textnormal{\tiny{SW}}}(0)[h](w)&=\fint_{\mathbb{T}}h'(\tau)\mathtt{F}_{1}\big(\tfrac{1}{\alpha}|w-\tau|\big)d\tau+\fint_{\mathbb{T}}\mathtt{F}_2\big(\tfrac{1}{\alpha}|w-\tau|\big)\frac{\big(h(w)-h(\tau)\big)\big(\overline{w}-\overline{\tau}\big)}{2\alpha|w-\tau|}d\tau\\
	&\quad+\fint_{\mathbb{T}}\mathtt{F}_2\big(\tfrac{1}{\alpha}|w-\tau|\big)\frac{\big(\overline{h(w)}-\overline{h(\tau)}\big)\big(w-\tau\big)}{2\alpha|w-\tau|}d\tau.
\end{align*}
This can be written in the form
$$d_{f}I^{\textnormal{\tiny{E}}}(0)[h](w)+d_{f}I^{\textnormal{\tiny{SW}}}(0)[h](w)=\mathcal{T}_{\mathtt{K}_{1}}(h')(w)+\mathcal{T}_{\mathtt{K_2}}(1)(w)+\mathcal{T}_{\overline{\mathtt{K}_2}}(1)(w),$$
with
\begin{align*}
	\mathcal{T}_{\mathtt{K}}(u)(w)&\triangleq\fint_{\mathbb{T}}u(\tau)\mathtt{K}(w,\tau)d\tau,\\
	\mathtt{K}_{1}(w,\tau)&\triangleq\mathtt{F}_1\big(\tfrac{1}{\alpha}|w-\tau|\big),\\
	\mathtt{K}_{2}(w,\tau)&\triangleq\mathtt{F}_2\big(\tfrac{1}{\alpha}|w-\tau|\big)\frac{\big(h(w)-h(\tau)\big)\big(\overline{w}-\overline{\tau}\big)}{2\alpha|w-\tau|}d\tau.
\end{align*}
Using the boundedness of $\mathtt{F}_1$, $\mathtt{F}_1'$, $\mathtt{F}_2$ and $\mathtt{F}_2'$ close to zero, we get 
$$\begin{array}{ll}
	|\mathtt{K}_1(w,\tau)|\leqslant C,\qquad&\displaystyle|\partial_w\mathtt{K}_1(w,\tau)|\leqslant\frac{C}{|w-\tau|},\vspace{0.1cm}\\
	|\mathtt{K}_{2}(w,\tau)|\leqslant C\|h\|_{C^{1+\beta}(\mathbb{T})},\qquad&\displaystyle|\partial_w\mathtt{K}_{2}(w,\tau)|\leqslant \frac{C\|h\|_{C^{1+\beta}(\mathbb{T})}}{|w-\tau|}\cdot
\end{array}$$
Therefore, applying Lemma \ref{lem sing ker}, one obtains the continuity of $d_{f}I^{\textnormal{\tiny{E}}}(0)+d_{f}I^{\textnormal{\tiny{SW}}}(0):C^{1+\beta}(\mathbb{T})\rightarrow C^{\delta}(\mathbb{T})$ for any $\beta\leqslant\delta<1.$ Coming back to the definition of $\mathcal{K}$, we deduce that for any $\beta\leqslant\delta<1,$ the operator $\mathcal{K}:X^{1+\beta}\rightarrow Y^{\delta}$ is continuous (the symmetry property being easily obtained by straightforward calculations and changes of variables in the integrals). Hence, using the compact embedding of $C^{\delta}(\mathbb{T})$ into $C^{\beta}(\mathbb{T})$ for $\beta<\delta<1,$ one deduces that the operator $\mathcal{K}:X^{1+\beta}\rightarrow Y^{\beta}$ is compact. Consequently, $d_{f}F_{\alpha}(\Omega,0):X^{1+\beta}\rightarrow Y^{\beta}$ is a Fredholm operator with index $0$. Now, we shall compute the Fourier representation of this operator. Fix 
$$h(w)=\sum_{n=0}^{\infty}a_n\overline{w}^{n}\in X^{1+\beta}.$$
First observe that
\begin{equation}\label{DfcalI}
	\mathcal{I}[h](w)=-\big(\Omega-\tfrac{1}{2}+I_{1}\left(\tfrac{1}{\alpha}\right)K_1\left(\tfrac{1}{\alpha}\right)\big)\sum_{n=0}^{\infty}na_ne_{n+1}(w)
\end{equation}
and
\begin{equation}\label{DfOmh}
	\textnormal{Im}\left\lbrace\Omega h(w)\overline{w}\right\rbrace=-\Omega\sum_{n=0}^{\infty}a_{n}e_{n+1}(w).
\end{equation}
The next task is to compute $d_{f}I^{\textnormal{\tiny{E}}}(0).$ One has
\begin{align*}
	\fint_{\mathbb{T}}h'(\tau)\log\big(|w-\tau|\big)d\tau&=w\fint_{\mathbb{T}}h'(\omega\tau)\log\big(|1-\tau|\big)d\tau\\
	&=-\sum_{n=1}^{\infty}na_n\overline{\omega}^{n}\fint_{\mathbb{T}}\overline{\tau}^{n+1}\log\big(|1-\tau|\big)d\tau.
\end{align*}
The last integral can be computed as follows by using \eqref{int diego}
\begin{align*}
	\fint_{\mathbb{T}}\overline{\tau}^{n+1}\log\big(|1-\tau|\big)d\tau&=\tfrac{1}{2\pi}\int_{0}^{2\pi}e^{-\ii n\theta}\log\big(|1-e^{\ii\theta}|\big)d\theta\\
	&=\tfrac{1}{4\pi}\int_{0}^{2\pi}\log\big(\sin^2\left(\tfrac{\theta}{2}\right)\big)\cos(n\theta)d\theta\\
	&=-\tfrac{1}{2n}.
\end{align*}
Hence,
$$\fint_{\mathbb{T}}h'(\tau)\log\big(|w-\tau|\big)d\tau=\sum_{n=1}^{\infty}\tfrac{a_n}{2}\overline{w}^{n}.$$
Similarly, we obtain
\begin{align*}
	\fint_{\mathbb{T}}\frac{h(w)-h(\tau)}{2\big(w-\tau\big)}d\tau&=\sum_{n=0}^{\infty}\tfrac{a_n}{2}\fint_{\mathbb{T}}\frac{\overline{w}^n-\overline{\tau}^n}{w-\tau}=-\sum_{n=1}^{\infty}\sum_{k=0}^{n-1}\tfrac{a_n\overline{w}^{k+1}}{2}\fint_{\mathbb{T}}\overline{\tau}^{n-k}d\tau=-\sum_{n=1}^{\infty}\tfrac{a_n}{2}\overline{w}^n\\
	\fint_{\mathbb{T}}\frac{\overline{h(w)}-\overline{h(\tau)}}{2\big(\overline{w}-\overline{\tau}\big)}d\tau&=\sum_{n=0}^{\infty}\tfrac{a_n}{2}\fint_{\mathbb{T}}\frac{w^n-\tau^n}{\overline{w}-\overline{\tau}}=-\sum_{n=1}^{\infty}\sum_{k=0}^{n-1}\tfrac{a_nw^{k+1}}{2}\fint_{\mathbb{T}}\tau^{n-k}d\tau=0.
\end{align*}
Consequently,
\begin{equation}\label{DfIE=0}
	d_fI^{\textnormal{\tiny{E}}}(0)=0.
\end{equation}
Now, let us turn to the calculation of $\textnormal{Im}\left\lbrace d_{f}I^{\textnormal{\tiny{SW}}}(0)[h](w)\overline{w}\right\rbrace.$ The formula \eqref{int Bes SW} gives
\begin{align*}
	\fint_{\mathbb{T}}h'(\tau)K_0\big(\tfrac{1}{\alpha}|w-\tau|\big)d\tau&=-\sum_{n=1}^{\infty}na_n\overline{w}^n\fint_{\mathbb{T}}\overline{\tau}^{n+1}K_0\big(\tfrac{1}{\alpha}|1-\tau|\big)d\tau\\
	&=-\sum_{n=1}^{\infty}\tfrac{na_n\overline{w}^n}{2\pi}\int_{0}^{2\pi}K_0\big(\tfrac{2}{\alpha}\sin\big(\tfrac{\theta}{2}\big)\big)\cos(n\theta)d\theta\\
	&=-\sum_{n=1}^{\infty}nI_{n}\big(\tfrac{1}{\alpha}\big)K_{n}\big(\tfrac{1}{\alpha}\big)a_n\overline{w}^n.
\end{align*}
Thus
$$\textnormal{Im}\left\lbrace\overline{w}\fint_{\mathbb{T}}h'(\tau)K_0\big(\tfrac{1}{\alpha}|w-\tau|\big)d\tau\right\rbrace=\sum_{n=1}^{\infty}nI_{n}\big(\tfrac{1}{\alpha}\big)K_{n}\big(\tfrac{1}{\alpha}\big)a_ne_{n+1}(w).$$
On the other hand
\begin{align*}
	&\fint_{\mathbb{T}}K_0'\big(\tfrac{1}{\alpha}|w-\tau|\big)\frac{\big(h(w)-h(\tau)\big)\big(\overline{w}-\overline{\tau}\big)+\big(\overline{h(w)}-\overline{h(\tau)}\big)\big(w-\tau\big)}{2\alpha|w-\tau|}d\tau\\
	&=\sum_{n=0}^{\infty}a_n\left(\overline{w}^{n}\fint_{\mathbb{T}}K_0'\big(\tfrac{1}{\alpha}|1-\tau|\big)\frac{\big(\overline{\tau}^n-1\big)\big(\overline{\tau}-1\big)}{2\alpha|1-\tau|}d\tau+w^{n}\fint_{\mathbb{T}}K_0'\big(\tfrac{1}{\alpha}|1-\tau|\big)\frac{\big(\tau^n-1\big)\big(\tau-1\big)}{2\alpha|1-\tau|}d\tau\right).
\end{align*}
One can easily check that the above integrals are real and therefore
\begin{align*}
	&\textnormal{Im}\left\lbrace\overline{w}\fint_{\mathbb{T}}K_0'\big(\tfrac{1}{\alpha}|w-\tau|\big)\frac{\big(h(w)-h(\tau)\big)\big(\overline{w}-\overline{\tau}\big)+\big(\overline{h(w)}-\overline{h(\tau)}\big)\big(w-\tau\big)}{2\alpha|w-\tau|}d\tau\right\rbrace\\
	&=\sum_{n=0}^{\infty}a_n\left(\fint_{\mathbb{T}}K_0'\big(\tfrac{1}{\alpha}|1-\tau|\big)\frac{\big(\tau^n-1\big)\big(\tau-1\big)-\big(\overline{\tau}^n-1\big)\big(\overline{\tau}-1\big)}{2\alpha|1-\tau|}d\tau\right)e_{n+1}(w)\\
	&=\sum_{n=0}^{\infty}a_n\left(\fint_{\mathbb{T}}K_0'\big(\tfrac{1}{\alpha}|1-\tau|\big)\frac{\big(\tau^{n+1}-\overline{\tau}^{n+1}\big)-\big(\tau^{n}-\overline{\tau}^{n}\big)-\big(\tau-\overline{\tau}\big)}{2\alpha|1-\tau|}d\tau\right)e_{n+1}(w).
\end{align*}
Now, symmetry arguments together with an integration by parts and \eqref{int Bes SW} imply for any $k\in\mathbb{N}^*$
\begin{align*}
	\fint_{\mathbb{T}}K_0'\big(\tfrac{1}{\alpha}|1-\tau|\big)\frac{\tau^{k}-\overline{\tau}^{k}}{2\alpha|1-\tau|}d\tau&=\frac{-1}{4\pi\alpha}\int_{0}^{2\pi}K_0'\big(\tfrac{2}{\alpha}\big|\sin\left(\tfrac{\theta}{2}\right)\big|\big)\frac{\sin(k\theta)\sin(\theta)}{\big|\sin\left(\tfrac{\theta}{2}\right)\big|}d\theta\\
	&=\frac{-1}{2\pi\alpha}\int_{0}^{2\pi}K_0'\big(\tfrac{2}{\alpha}\sin\left(\tfrac{\theta}{2}\right)\big)\sin(k\theta)\cos\left(\tfrac{\theta}{2}\right)d\theta\\
	&=\frac{k}{2\pi}\int_{0}^{2\pi}K_0\big(\tfrac{2}{\alpha}\sin\left(\tfrac{\theta}{2}\right)\big)\cos(k\theta)d\theta\\
	&=kI_{k}\left(\tfrac{1}{\alpha}\right)K_{k}\left(\tfrac{1}{\alpha}\right).
\end{align*}
Combining the foregoing computations leads to
\begin{equation}\label{DfISW0}
	\textnormal{Im}\left\lbrace d_{f}I^{\textnormal{\tiny{SW}}}(0)[h](w)\overline{w}\right\rbrace=\sum_{n=0}^{\infty}a_n\Big((n+1)I_{n+1}\left(\tfrac{1}{\alpha}\right)K_{n+1}\left(\tfrac{1}{\alpha}\right)-I_{1}\left(\tfrac{1}{\alpha}\right)K_{1}\left(\tfrac{1}{\alpha}\right)\Big)e_{n+1}(w).
\end{equation}
Putting together \eqref{DfcalI}, \eqref{DfOmh}, \eqref{DfIE=0} and \eqref{DfISW0} gives the desired result.
\end{proof}
According to Proposition \ref{prop lin op Ea} the possible values for $\Omega$ from which we can hope to bifurcate are 
\begin{equation}\label{def disp set}
	\boldsymbol{\Omega}_{n}^{\textnormal{\tiny{E}}}(\alpha)\triangleq \tfrac{n-1}{2n}-\big[I_1\left(\tfrac{1}{\alpha}\right)K_1\left(\tfrac{1}{\alpha}\right)-I_n\left(\tfrac{1}{\alpha}\right)K_n\left(\tfrac{1}{\alpha}\right)\big]=\boldsymbol{\Omega}_{n}^{\textnormal{\tiny{E}}}-\boldsymbol{\Omega}_{n}^{\textnormal{\tiny{SW}}}\left(\tfrac{1}{\alpha}\right),\quad n\in\mathbb{N}^*.
\end{equation}
We shall now prove the monotonicity of the sequence $\big(\boldsymbol{\Omega}_{n}^{\textnormal{\tiny{E}}}(\alpha)\big)_{n\in\mathbb{N}^*}.$ This is given by the following result.
\begin{lem}\label{lem mono ev Ea}
	For any $\alpha>0,$ the sequence $\big(\boldsymbol{\Omega}_{n}^{\textnormal{\tiny{E}}}(\alpha)\big)_{n\in\mathbb{N}^*}$ is strictly increasing and tends to $\boldsymbol{\Omega}_{\infty}^{\textnormal{\tiny{E}}}(\alpha),$ with
	$$\boldsymbol{\Omega}_{\infty}^{\textnormal{\tiny{E}}}(\alpha)\triangleq \tfrac{1}{2}-I_1\left(\tfrac{1}{\alpha}\right)K_1\left(\tfrac{1}{\alpha}\right).$$
	Moreover,
$$\forall n\in\mathbb{N}^*,\quad\boldsymbol{\Omega}_{n}^{\textnormal{\tiny{E}}}(\alpha)\underset{\alpha\rightarrow0}{\longrightarrow}\tfrac{n-1}{2n}.$$
\end{lem}
\begin{proof}
	The convergences are immediate consequences of \eqref{asymp high order} and \eqref{asymp large z}. We shall now study the monotonicity. For $n\in\mathbb{N}^*,$ we can write
	$$\boldsymbol{\Omega}_{n+1}^{\textnormal{\tiny{E}}}(\alpha)-\boldsymbol{\Omega}_{n}^{\textnormal{\tiny{E}}}(\alpha)=\big(\tfrac{n}{2(n+1)}-\tfrac{n-1}{2n}\big)-\big(\boldsymbol{\Omega}_{n+1}^{\textnormal{\tiny{SW}}}\left(\tfrac{1}{\alpha}\right)-\boldsymbol{\Omega}_{n}^{\textnormal{\tiny{SW}}}\left(\tfrac{1}{\alpha}\right)\big).$$
	We look for the monotonicity of the function $\varphi_{n}$ defined for $n\in\mathbb{N}^*$ by
	$$\forall x>0,\quad\varphi_{n}(x)\triangleq \boldsymbol{\Omega}_{n+1}^{\textnormal{\tiny{SW}}}\left(x\right)-\boldsymbol{\Omega}_{n}^{\textnormal{\tiny{SW}}}\left(x\right)=I_{n}\left(x\right)K_{n}\left(x\right)-I_{n+1}\left(x\right)K_{n+1}\left(x\right).$$
	From the decay property of $x\mapsto I_{n}(x)K_{n}(x)$ on $(0,\infty)$ for every $n\in\mathbb{N}^{*}$ and the asymptotic expansion \eqref{asymp small arg} we deduce
	\begin{equation}\label{upper bound InKn}
		\forall n\in\mathbb{N}^{*},\quad\forall x>0,\quad I_{n}(x)K_{n}(x)<\frac{1}{2n}.
	\end{equation}
	Using \eqref{wronskian}, \eqref{ratio bounds with derivatives} and \eqref{upper bound InKn}, we find
	\begin{align*}
		\varphi_{n}'(x) & =  I_{n}'(x)K_{n}(x)+I_{n}(x)K_{n}'(x)-I_{n+1}'(x)K_{n+1}(x)-I_{n+1}(x)K_{n+1}'(x)\\
		& = 2\left(\frac{1}{x}+I_{n}(x)K_{n}'(x)-I_{n+1}'(x)K_{n+1}(x)\right)\\
		& = \frac{2}{x}\left(1+\frac{x K_{n}'(x)}{K_{n}(x)}I_{n}(x)K_{n}(x)-\frac{x I_{n+1}'(x)}{I_{n+1}(x)}I_{n+1}(x)K_{n+1}(x)\right)\\
		& <\frac{2}{x}\left(1-\sqrt{x^{2}+n^{2}}\big[I_{n}(x)K_{n}(x)+I_{n+1}(x)K_{n+1}(x)\big]\right)\\
		& < \frac{2}{x}\left(1-\sqrt{x^{2}+n^{2}}\right)\\
		&<  0.
	\end{align*}
We deduce that for all $n\in\mathbb{N}^{*}$, $\varphi_{n}$ is strictly decreasing on $(0,\infty).$ In addition, from \eqref{asymp small arg}, we infer
$$\forall n\in\mathbb{N}^*,\quad \lim_{x\to0}I_n(x)K_n(x)=\tfrac{1}{2n},$$
which implies in turn
$$\forall n\in\mathbb{N}^*,\quad \lim_{x\to0}\boldsymbol{\Omega}_n^{\textnormal{\tiny{SW}}}(x)=\tfrac{n-1}{2n}.$$
Therefore,
$$\forall n\in\mathbb{N}^*,\quad \boldsymbol{\Omega}_{n+1}^{\textnormal{\tiny{E}}}(\alpha)-\boldsymbol{\Omega}_{n}^{\textnormal{\tiny{E}}}(\alpha)>\big(\tfrac{n}{2(n+1)}-\tfrac{n-1}{2n}\big)-\big(\tfrac{n}{2(n+1)}-\tfrac{n-1}{2n}\big)=0.$$
This achieves the proof of Lemma \ref{lem mono ev Ea}. 
\end{proof}
\subsection{Construction of local bifurcation branches}
We check here the hypothesis of Crandall-Rabinowitz's Theorem \ref{Crandall-Rabinowitz theorem} which immediately imply Theorem \ref{thm Ea sc}. Notice that the line of trivial solutions has already been obtained in Lemma \ref{prop lin op Ea}-1-(iv).
\begin{prop}\label{prop hyp CR Ea}
	Let $\alpha>0,$ $\beta\in(0,1)$ and $\mathbf{m}\in\mathbb{N}^*.$ The following assertions hold true.
	\begin{enumerate}[label=(\roman*)]
		\item $F_{\alpha}:\mathbb{R}\times X_{\mathbf{m}}^{1+\beta}\rightarrow Y_{\mathbf{m}}^{\beta}$ is well-defined and of class $C^1.$
		\item The kernel $\ker\Big(d_fF_{\alpha}\big(\boldsymbol{\Omega}_{\mathbf{m}}^{\textnormal{\tiny{E}}}(\alpha),0\big)\Big)$ is one dimensional and generated by
		$$v_{0,\mathbf{m}}:\begin{array}[t]{rcl}
			\mathbb{T} & \rightarrow & \mathbb{C}\\
			w & \mapsto & \overline{w}^{\mathbf{m}-1}.		
		\end{array}$$
	\item The range $\mathcal{R}\Big(d_fF_{\alpha}\big(\boldsymbol{\Omega}_{\mathbf{m}}^{\textnormal{\tiny{E}}}(\alpha),0\big)\Big)$ is closed and of codimension one in $Y_{\mathbf{m}}^{\beta}.$
	\item Transversality condition :
	$$\partial_{\Omega}d_fF_{\alpha}\big(\boldsymbol{\Omega}_{\mathbf{m}}^{\textnormal{\tiny{E}}}(\alpha),0\big)[v_{0,\mathbf{m}}]\not\in\mathcal{R}\Big(d_fF_{\alpha}\big(\boldsymbol{\Omega}_{\mathbf{m}}^{\textnormal{\tiny{E}}}(\alpha),0\big)\Big).$$
\end{enumerate}
\end{prop}
\begin{proof}
	\textbf{(i)} Follows immediately from Proposition \ref{prop lin op Ea}-1.\\
	\textbf{(ii)} It is a direct consequence of Proposition \ref{prop lin op Ea}-2 and Lemma \ref{lem mono ev Ea} since for 
	$$h(w)=\sum_{n=0}^{\infty}a_n\overline{w}^{n\mathbf{m}-1}\in X_{\mathbf{m}}^{1+\beta},$$
	we have by \eqref{lin op Ea}
	\begin{equation}\label{lin op Ea sym}
		d_{f}F_{\alpha}\big(\Omega_{\mathbf{m}}^{\textnormal{\tiny{E}}}(\alpha),0\big)[h]=\sum_{n=1}^{\infty}n\mathbf{m}a_n\Big(\Omega_{n\mathbf{m}}^{\textnormal{\tiny{E}}}(\alpha)-\Omega_{\mathbf{m}}^{\textnormal{\tiny{E}}}(\alpha)\Big)e_{n\mathbf{m}}.
	\end{equation}
	\textbf{(iii)} From Proposition \ref{prop lin op Ea}-2, we know that $d_fF_{\alpha}\big(\Omega_{\mathbf{m}}^{\textnormal{\tiny{E}}}(\alpha),0\big)$ is a Fredholm operator with index zero. Together with the point (ii), we conclude that the range of $d_fF_{\alpha}\big(\Omega_{\mathbf{m}}^{\textnormal{\tiny{E}}}(\alpha),0\big)$ is closed with codimension one. We endow $Y_{\mathbf{m}}^{\beta}$ with the scalar product
	$$\big\langle f,g\big\rangle\triangleq\fint_{\mathbb{T}}f(w)\overline{g(w)}dw=\sum_{n=1}^{\infty}f_{n\mathbf
	m}g_{n\mathbf{m}},\qquad f=\sum_{n=1}^{\infty}f_{n\mathbf{m}}e_{n\mathbf{m}},\qquad g=\sum_{n=1}^{\infty}g_{n\mathbf{m}}e_{n\mathbf{m}}.$$
The continuity of $\langle\cdot,\cdot\rangle$ on $Y_{\mathbf{m}}^{\beta}\times Y_{\mathbf{m}}^{\beta}$ is a direct consequence of Cauchy-Schwarz inequality and the continuous embedding $C^{\beta}(\mathbb{T})\hookrightarrow L^{2}(\mathbb{T}).$ We claim that
\begin{equation}\label{orth range}
	\mathcal{R}\Big(d_f F_{\alpha}\big(\Omega_{\mathbf{m}}^{\textnormal{\tiny{E}}}(\alpha),0\big)\Big)=\langle e_{\mathbf{m}}\rangle^{\perp},
\end{equation}
where the orthogonal is understood in the sense of the scalar product $\langle\cdot,\cdot\rangle.$ Observe that \eqref{lin op Ea sym} and the point (i) immediately imply the first inclusion in \eqref{orth range}. The inverse inclusion is deduced from the fact that the range is of codimension one.\\ 
	\textbf{(iv)} For any $h\in X_{\mathbf{m}}^{1+\beta},$
	$$\forall w\in\mathbb{T},\quad\partial_{\Omega}d_fF_{\alpha}\big(\boldsymbol{\Omega}_{\mathbf{m}}^{\textnormal{\tiny{E}}}(\alpha),0\big)[h](w)=\textnormal{Im}\left\lbrace h(w)\overline{w}+\overline{h'(w)}\right\rbrace.$$
	Consequently, \eqref{orth range} implies
	$$\partial_{\Omega}d_fF_{\alpha}\big(\boldsymbol{\Omega}_{\mathbf{m}}^{\textnormal{\tiny{E}}}(\alpha),0\big)[v_{0,\mathbf{m}}]=-\mathbf{m}e_{\mathbf{m}}\not\in\mathcal{R}\Big(d_fF_{\alpha}\big(\boldsymbol{\Omega}_{\mathbf{m}}^{\textnormal{\tiny{E}}}(\alpha),0\big)\Big).$$
	This ends the proof of Proposition \ref{prop hyp CR Ea} and proves Theorem \ref{thm Ea sc}.
\end{proof}
\begin{remark}\label{rem triv sol}
	Observe that for $\mathbf{m}=1$, a similar calculation to Proposition \ref{prop lin op Ea}-1-(iv), shows that for $\Phi(z)=z+a_0,$ with $a_0\in\mathbb{R}$, one has $F_{\alpha}(0,\Phi)=0.$ By uniqueness of the constructed branch of bifurcation, this latter corresponds to a translation of the Rankine vortex.
\end{remark}
\section{Quasi-periodic Euler-$\alpha$ patches}
This section is devoted to the proof of Theorem \ref{thm QPS Ea}. It is based on KAM and Nash-Moser techniques in a similar way to the recent works \cite{BHM21,HHM21,HR21-1,HR21}. First let us introduce the notations and topologies used along this section.
\subsection{Topologies for functions and operators}\label{sec topo}
This subsection presents the notations and topologies used along this part on the construction of quasi-periodic vortex patches. 
First we fix 
\begin{equation}\label{alf0alf1}
	0<\alpha_0<\alpha_1.
\end{equation}
The parameter $\alpha$ will live in the interval $(\alpha_0,\alpha_1).$ More precisely, at the end it will belongs to a Cantor set contained in this interval. We shall denote
\begin{equation}\label{p-d}
	d\in\mathbb{N}^*
\end{equation}
the number of excited frequencies generating the quasi-periodic solutions. Consequently, the frequency vector $\omega$ belongs to $\mathbb{R}^d.$ More precisely, at the end $\omega$ should be close to the equilibrium frequency vector obtained at the linear level at the Rankine vortex. Our solutions will be searched in the Sobolev class constructed on $L^2$ in the variables $\varphi\in\mathbb{T}^d$ and $\theta\in\mathbb{T}.$ Hence, we decompose any $\rho\in L^{2}(\mathbb{T}^{d+1},\mathbb{C}),$ in Fourier series as follows
$$\rho=\sum_{(l,j)\in\mathbb{Z}^{d+1 }}\rho_{l,j}\,\mathbf{e}_{l,j},\qquad\rho_{l,j}\triangleq\big\langle\rho,\mathbf{e}_{l,j}\big\rangle_{L^{2}(\mathbb{T}^{d+1},\mathbb{C})},$$
where $(\mathbf{e}_{l,j})_{(l,j)\in\mathbb{Z}^{d}\times\mathbb{Z}}$ denotes the classical Hilbert basis of the complex Hilbert space  $L^{2}(\mathbb{T}^{d+1},\mathbb{C}).$ Explicitly, we have
$$\mathbf{e}_{l,j}(\varphi,\theta)\triangleq e^{\ii(l\cdot\varphi+j\theta)},\qquad\mathbf{e}_{j}\triangleq \mathbf{e}_{0,j}.$$
The associated Hermitian inner product is
$$\big\langle\rho_{1},\rho_{2}\big\rangle_{L^{2}(\mathbb{T}^{d+1},\mathbb{C})}\triangleq\bigintssss_{\mathbb{T}^{d+1}}\rho_{1}(\varphi,\theta)\overline{\rho_{2}(\varphi,\theta)}d\varphi d\theta,\qquad\int_{\mathbb{T}^{n}}f(x)dx\triangleq\frac{1}{(2\pi)^{n}}\bigintssss_{[0,2\pi]^{n}}f(x)dx.$$
For $s\in\mathbb{R}$ the complex Sobolev space $H^{s}(\mathbb{T}^{d +1},\mathbb{C})$ is given by 
$$H^{s}(\mathbb{T}^{d +1},\mathbb{C})\triangleq\Big\{\rho\in L^{2}(\mathbb{T}^{d +1},\mathbb{C})\quad\textnormal{s.t.}\quad\| \rho\|_{H^{s}}^{2}\triangleq \sum_{(l,j)\in\mathbb{Z}^{d+1 }}\langle l,j\rangle^{2s}|\rho_{l,j}|^{2}<\infty\Big\},\qquad \langle l,j\rangle\triangleq \max(1,|l|,|j|).$$
The closed sub-vector space of real valued functions is denoted
\begin{align*}
	\nonumber H^{s}\triangleq H^{s}(\mathbb{T}^{d+1},\mathbb{R})&\triangleq\left\lbrace\rho\in H^{s}(\mathbb{T}^{d +1},\mathbb{C})\quad\textnormal{s.t.}\quad\forall\, (\varphi,\theta)\in\mathbb{T}^{d+1 },\,\rho(\varphi,\theta)=\overline{\rho(\varphi,\theta)}\right\rbrace\\
	&=\Big\{\rho\in H^{s}(\mathbb{T}^{d +1},\mathbb{C})\quad\textnormal{s.t.}\quad\forall \,(l,j)\in\mathbb{Z}^{d+1},\,\rho_{-l,-j}=\overline{\rho_{l,j}}\Big\}.
\end{align*}
In order to ensure some suitable embeddings, we shall consider the following restrictions on the Sobolev indices. In particular $S$ is chosen large enough.
\begin{equation}\label{Sob index}
	S\geqslant s\geqslant s_{0}>\tfrac{d+1}{2}+q+2.
\end{equation} 
During the scheme, we shall also keep track of the regularity of our functions and operators with respect to the parameters $\mu\triangleq(\alpha,\omega).$ Thus, we introduce the parameters
\begin{equation}\label{p-cond0}
	\gamma\in(0,1),\qquad q\in\mathbb{N}^{*}
	\end{equation}
and consider the following weighted spaces
\begin{align*}
	W^{q,\infty,\gamma}(\mathcal{O},H^{s})&\triangleq\Big\lbrace \rho:\mathcal{O}\rightarrow H^{s}\quad\textnormal{s.t.}\quad\|\rho\|_{q,s}^{\gamma,\mathcal{O}}<\infty\Big\rbrace,\qquad	\|\rho\|_{q,s}^{\gamma,\mathcal{O}}\triangleq\sum_{\underset{|\alpha|\leqslant q}{\alpha\in\mathbb{N}^{d+1}}}\gamma^{|\alpha|}\sup_{\mu\in{\mathcal{O}}}\|\partial_{\mu}^{\alpha}\rho(\mu,\cdot)\|_{H^{s-|\alpha|}},\\
	W^{q,\infty,\gamma}(\mathcal{O},\mathbb{C})&\triangleq\Big\lbrace\rho:\mathcal{O}\rightarrow\mathbb{C}\quad\textnormal{s.t.}\quad\|\rho\|_{q}^{\gamma,\mathcal{O}}<\infty\Big\rbrace,\qquad\|\rho\|_{q}^{\gamma,\mathcal{O}}\triangleq\sum_{\underset{|\alpha|\leqslant q}{\alpha\in\mathbb{N}^{d+1}}}\gamma^{|\alpha|}\sup_{\mu\in{\mathcal{O}}}|\partial_{\mu}^{\alpha}\rho(\mu)|.
\end{align*}
The set $\mathcal{O}$ is an open bounded subset of $\mathbb{R}^{d+1}$ which will be fixed in \eqref{def calO}. The parameters $\gamma$ and $q$ will be fixed in \eqref{gma N0 NM} and \eqref{def q}, respectively. Observe that any $\rho\in W^{q,\infty,\gamma}(\mathcal{O},H^{s})$ decomposes as follows
$$\rho(\mu,\varphi,\theta)=\sum_{(l,j)\in\mathbb{Z}^{d+1}}\rho_{l,j}(\mu)\mathbf{e}_{l,j}(\varphi,\theta).$$
Some classical properties of the weighted Sobolev norm are gathered in the following lemma. We refer for instance to \cite{BM18,BFM21,BFM21-1} for the ideas of their proofs.
\begin{lem}\label{Lem-lawprod}
	Let $(d,\gamma,q,s_{0},s)$ satisfying \eqref{p-d}, \eqref{Sob index} and \eqref{p-cond0}. Take $\rho_1,\rho_2\in W^{q,\infty,\gamma}(\mathcal{O},H^s).$ Then the following assertions hold true.
	\begin{enumerate}[label=(\roman*)]
		\item Space translation invariance : for any $\eta\in\mathbb{T},$ we have $\tau_{\eta}\rho_1:(\mu,\varphi,\theta)\mapsto\rho_1(\mu,\varphi,\eta+\theta)\in W^{q,\infty,\gamma}(\mathcal{O},H^{s})$, and 
		$$\|\tau_{\eta}\rho_1\|_{q,s}^{\gamma,\mathcal{O}}=\|\rho_1\|_{q,s}^{\gamma,\mathcal{O}}.$$
		\item Projectors properties : For all $N\in\mathbb{N}^{*}$ and for all $t\in\mathbb{R}_{+}^{*},$
		$$\|\Pi_{N}\rho_1\|_{q,s+t}^{\gamma,\mathcal{O}}\leqslant N^{t}\|\rho_1\|_{q,s}^{\gamma,\mathcal{O}}\qquad\|\Pi_{N}^{\perp}\rho_1\|_{q,s}^{\gamma,\mathcal{O}}\leqslant N^{-t}\|\rho_1\|_{q,s+t}^{\gamma,\mathcal{O}},$$
		where the projectors are defined by
		$$\Pi_{N}\left(\sum_{j\in\mathbb{Z}}(\rho_1)_j\mathbf{e}_{j}\right)\triangleq\sum_{j\in\mathbb{Z}\atop|j|\leqslant N}(\rho_1)_j\mathbf{e}_{j},\qquad\Pi_{N}^{\perp}\triangleq\textnormal{Id}-\Pi_{N}.$$
		\item Law products : $\rho_{1}\rho_{2}\in W^{q,\infty,\gamma}(\mathcal{O},H^{s})$ and 
		$$\| \rho_{1}\rho_{2}\|_{q,s}^{\gamma,\mathcal{O}}\lesssim\| \rho_{1}\|_{q,s_{0}}^{\gamma,\mathcal{O}}\| \rho_{2}\|_{q,s}^{\gamma,\mathcal{O}}+\| \rho_{1}\|_{q,s}^{\gamma,\mathcal{O}}\| \rho_{2}\|_{q,s_{0}}^{\gamma,\mathcal{O}}.$$
		\item Composition law : For $f\in C^{\infty}(\mathcal{O}\times\mathbb{R},\mathbb{R})$, if there exists $\mathtt{M}>0$ such that
		$$\| \rho_{1}\|_{q,s}^{\gamma,\mathcal{O}},\|\rho_{2}\|_{q,s}^{\gamma,\mathcal{O}}\leqslant \mathtt{M},$$ 
		then $f(\rho_{1})-f(\rho_{2})\in W^{q,\infty,\gamma}(\mathcal{O},H^{s})$ with 
		$$\| f(\rho_{1})-f(\rho_{2})\|_{q,s}^{\gamma,\mathcal{O}}\leqslant C(s,d,q,f,\mathtt{M})\| \rho_{1}-\rho_{2}\|_{q,s}^{\gamma,\mathcal{O}},$$
		where we used the notation
		$$\forall (\mu,\varphi,\theta)\in \mathcal{O}\times\mathbb{T}^{d+1},\quad f(\rho)(\mu,\varphi,\theta)\triangleq  f(\mu,\rho(\mu,\varphi,\theta)).$$
	\end{enumerate}
\end{lem}
We shall now present the operator topology used along this section. In particular, we deal with the Toeplitz in time operator class. These notions are based on the one introduced in \cite{BBMH18,BFM21,BFM21-1,BM18}. We consider parameter dependent operators in the form 
$$T:\mu\in\mathcal{O}\mapsto T(\mu)\in\mathcal{L}(H^s(\T^{d+1},\mathbb{C}))$$ which can be identified with an infinite dimensional matrix $\left(T_{l_{0},j_{0}}^{l,j}(\mu)\right)_{\underset{(j,j_{0})\in\mathbb{Z}^{2}}{(l,l_{0})\in(\mathbb{Z}^{d })^{2}}}$ through their action on the Hilbert basis $\big(\mathbf{e}_{l,j}\big)_{(l,j)\in\mathbb{Z}^d\times\mathbb{Z}}$ as follows 
$$T(\mu)\mathbf{e}_{l_{0},j_{0}}=\sum_{(l,j)\in\mathbb{Z}^{d+1}}T_{l_{0},j_{0}}^{l,j}(\mu)\mathbf{e}_{l,j},\qquad T_{l_{0},j_{0}}^{l,j}(\mu)\triangleq\big\langle T(\mu)\mathbf{e}_{l_{0},j_{0}},\mathbf{e}_{l,j}\big\rangle_{L^{2}(\mathbb{T}^{d+1})}.$$
More precisely, we shall see such operators acting on $W^{q,\infty,\gamma}(\mathcal{O},H^{s}(\mathbb{T}^{d+1},\mathbb{C}))$ in the following sense, 
$$\rho\in W^{q,\infty,\gamma}\big(\mathcal{O},H^{s}(\mathbb{T}^{d+1},\mathbb{C})\big),\quad\,\quad (T\rho)(\mu,\varphi,\theta)\triangleq T(\mu)\rho(\mu,\varphi,\theta).$$
We shall now introduce the Toeplitz operators class. An operator $T(\mu)$ is said to be Toeplitz in time iff
	$$T_{l_{0},j_{0}}^{l,j}(\mu)=T_{j_{0}}^{j}(\mu,l-l_{0}),\qquad T_{j_{0}}^{j}(\mu,l)\triangleq T_{0,j_{0}}^{l,j}(\mu).$$
Its action on a function $\rho=\displaystyle\sum_{(l_{0},j_{0})\in\mathbb{Z}^{d+1}}\rho_{l_{0},j_{0}}\mathbf{e}_{l_{0},j_{0}}$ writes
\begin{equation}\label{action Toeplitz}
	T(\mu)\rho=\sum_{(l,l_{0})\in(\mathbb{Z}^{d})^{2}\\\atop (j,j_{0})\in\mathbb{Z}^{2}}T_{j_{0}}^{j}(\mu,l-l_{0})\rho_{l_{0},j_{0}}\mathbf{e}_{l,j}.
\end{equation}
The Toeplitz topology is given by the following off-diagonal norm,
$$\| T\|_{\textnormal{\tiny{O-d}},q,s}^{\gamma,\mathcal{O}}\triangleq\sum_{\underset{|\alpha|\leqslant q}{\alpha\in\mathbb{N}^{d+1}}}\gamma^{|\alpha|}\sup_{(b,\omega)\in{\mathcal{O}}}\|\partial_{\mu}^{\alpha}(T)(\mu)\|_{\textnormal{\tiny{O-d}},s-|\alpha|},\qquad\| T\|_{\textnormal{\tiny{O-d}},s}^{2}\triangleq\sum_{(l,m)\in\mathbb{Z}^{d+1}}\langle l,m\rangle^{2s}\sup_{j-k=m}|T_{j}^{k}(l)|^{2}.$$
We mention that we will encounter several operators acting only on the variable $\theta$ and that can be considered as  $\varphi$-dependent operators $T(\mu,\varphi)$ taking the form of integral operators
\begin{align*}
	T(\mu,\varphi)\rho(\varphi,\theta)=\int_{\T} K(\mu,\varphi,\theta,\eta)\rho(\varphi,\eta)d\eta.
\end{align*}
One can easily check that those operators are Toeplitz and therefore they  satisfy \eqref{action Toeplitz}. We shall now give some classical results on the Toeplitz norm. The proofs are very similar to those in \cite{BM18} concerning pseudo-differential operators.
\begin{lem}\label{lem prop Toe}
	Let $(d,s_{0},s,\gamma,q)$ satisfying \eqref{p-d}, \eqref{Sob index} and \eqref{p-cond0}.  Let $T,$ $T_{1}$ and $T_{2}$ be Toeplitz in time operators.
	\begin{enumerate}[label=(\roman*)]
		\item Composition law :
		$$\| T_{1}T_{2}\|_{\textnormal{\tiny{O-d}},q,s}^{\gamma,\mathcal{O}}\lesssim\| T_{1}\|_{\textnormal{\tiny{O-d}},q,s}^{\gamma,\mathcal{O}}\| T_{2}\|_{\textnormal{\tiny{O-d}},q,s_{0}}^{\gamma,\mathcal{O}}+\| T_{1}\|_{\textnormal{\tiny{O-d}},q,s_{0}}^{\gamma,\mathcal{O}}\| T_{2}\|_{\textnormal{\tiny{O-d}},q,s}^{\gamma,\mathcal{O}}.$$ 
		\item Link between operators and off-diagonal norms :
		$$\| T\rho\|_{q,s}^{\gamma,\mathcal{O}}\lesssim\| T\|_{\textnormal{\tiny{O-d}},q,s_{0}}^{\gamma,\mathcal{O}}\|\rho\|_{q,s}^{\gamma,\mathcal{O}}+\| T\|_{\textnormal{\tiny{O-d}},q,s}^{\gamma,\mathcal{O}}\|\rho\|_{q,s_{0}}^{\gamma,\mathcal{O}}.$$
		In particular
		$$\| T \rho\|_{q,s}^{\gamma,\mathcal{O}}\lesssim\| T\|_{\textnormal{\tiny{O-d}},q,s}^{\gamma,\mathcal{O}}\|\rho\|_{q,s}^{\gamma,\mathcal{O}}.$$
	\end{enumerate}
\end{lem}
Now, wee give the following definition inspired for instance from \cite[Def. 2.2]{BBM14}.
\begin{defin}\label{Def-Rev}
	We defined the following involutions
	$$(\mathscr{S}_{2}\rho)(\varphi,\theta)=\rho(-\varphi,-\theta),\qquad(\mathscr{S}_{c}\rho)(\varphi,\theta)=\overline{\rho(\varphi,\theta)}.$$
	An operator $T=T(\mu)$ is said to be
	\begin{enumerate}[label=\textbullet]
		\item real iff $\forall\rho\in L^{2}(\mathbb{T}^{d+1},\mathbb{C}),\,\mathscr{S}_{c}\rho=\rho\Rightarrow\mathscr{S}_{c}(T\rho)=T\rho.$
		\item reversible iff $T\circ\mathscr{S}_{2}=-\mathscr{S}_{2}\circ T.$
		\item reversibility preserving iff $T\circ\mathscr{S}_{2}=\mathscr{S}_{2}\circ T.$
	\end{enumerate}
\end{defin}
\noindent We end this subsection by recalling  an important lemma whose proof can be found in \cite[Lem. 4.4]{HR21}.
\begin{lem}\label{lem sym-Rev}
	Let $(d,s_{0},s,\gamma,q)$ satisfy \eqref{p-d}, \eqref{Sob index} and \eqref{p-cond0}, then for any integral opearator $T$ with a real-valued kernel $K$, namely
	$$(T\rho)(\mu,\varphi,\theta)=\int_{\mathbb{T}}\rho(\mu,\varphi,\eta)K(\mu,\varphi,\theta,\eta)d\eta,\qquad K:(\mu,\varphi,\theta,\eta)\mapsto K(\mu,\varphi,\theta,\eta),$$
	the following property holds true.
	\begin{enumerate}[label=\textbullet]
		\item If $K$ is even in $(\varphi,\theta,\eta)$, then $T$ is a real and reversibility preserving Toeplitz in time operator.
		\item If $K$ is odd in $(\varphi,\theta,\eta)$, then $T$ is a real and reversible Toeplitz in time operator.
	\end{enumerate}
	Moreover,
	$$\| T\|_{\textnormal{\tiny{O-d}},q,s}^{\gamma,\mathcal{O}}\lesssim \int_{\T}\|K(\ast,\cdot,\centerdot,\eta+\centerdot)\|_{q,s+s_{0}}^{\gamma,\mathcal{O}} d\eta$$
	and
	$$\| T\rho\|_{q,s}^{\gamma,\mathcal{O}}\lesssim \|\rho\|_{q,s_0}^{\gamma,\mathcal{O}} \int_{\T}\|K(\ast,\cdot,\centerdot,\eta+\centerdot)\|_{q,s}^{\gamma,\mathcal{O}} d\eta+\|\rho\|_{q,s}^{\gamma,\mathcal{O}} \int_{\T}\|K(\ast,\cdot,\centerdot,\eta+\centerdot)\|_{q,s_0}^{\gamma,\mathcal{O}} d\eta,$$
	where the notation $\ast,\cdot,\centerdot$ denote $\mu,\varphi,\theta$, respectively.
\end{lem}
\subsection{Hamiltonian contour dynamics equation and its linearization}
Here, we reformulate the contour dynamics equation \eqref{complex vp eq} as a Hamiltonian equation on the radial deformation of the patch motion near the Rankine vortex associated with the unit disc. We also compute the linearization of this Hamiltonian equation both at the equilibrium and at a general state close to it. Let us consider the following polar parametrization of the boundary of a patch $t\mapsto\mathbf{1}_{D_t}$ close to the Rankine patch $\mathbf{1}_{\mathbb{D}}.$
\begin{equation}\label{prmz}
	z(t,\theta)\triangleq R(t,\theta)e^{\ii(\theta-\Omega t)},\qquad R(t,\theta)\triangleq \sqrt{1+2r(t,\theta)}.
\end{equation} 
The radial deformation $r$ is assumed to be of small amplitudes and the angular velocity $\Omega>0$ is introduced to avoid the resonance of the first equilibrium frequency as in \cite{HHM21,HR21}. Therefore, $r$ satisfies an Hamiltonian PDE as explained in the following lemma.
\begin{lem}\label{lem eq Ea r + HAM}
	The following hold true.
	\begin{enumerate}[label=(\roman*)]
		\item At least for short time $T>0$, the radial deformation $r$ in \eqref{prmz} is solution of the following equation
		\begin{equation}\label{Ea eq r}
			\forall (t,\theta)\in[0,T]\times\mathbb{T},\quad\partial_{t}r(t,\theta)+\Omega\partial_{\theta}r(t,\theta)-F^{\textnormal{\tiny{E}}}[r](t,\theta)-F^{\textnormal{\tiny{SW}}}[r](\alpha,t,\theta)=0,
		\end{equation}
		where
		\begin{align}
			F^{\textnormal{\tiny{E}}}[r](t,\theta)&\triangleq \int_{\mathbb{T}}\log\big(A_r(t,\theta,\eta)\big)\partial_{\theta\eta}^2\Big(R(t,\theta)R(t,\eta)\sin(\eta-\theta)\Big)d\eta,\label{FE}\\
			F^{\textnormal{\tiny{SW}}}[r](\alpha,t,\theta)&\triangleq \int_{\mathbb{T}}K_0\left(\tfrac{1}{\alpha}A_r(t,\theta,\eta)\right)\partial_{\theta\eta}^2\Big(R(t,\theta)R(t,\eta)\sin(\eta-\theta)\Big)d\eta,\label{FSW}\\
			A_r(t,\theta,\eta)&\triangleq \left|R(t,\theta)e^{\ii\theta}-R(t,\eta)e^{\ii\eta}\right|.\label{Ar}
		\end{align}
		\item The nonlinear and nonlocal transport-type PDE \eqref{Ea eq r} can be written in the following Hamiltonian form
		\begin{equation}\label{Heq Ea}
			\partial_{t}r=\partial_{\theta}\nabla \mathscr{H}(r),
		\end{equation}
		where 
		$$\mathscr{H}(r)\triangleq\tfrac{1}{2}\big(\mathscr{E}(r)-\Omega \mathscr{J}(r)\big),\quad \mathscr{E}(r)(t)\triangleq-\tfrac{1}{2\pi}\int_{D_{t}}\mathbf{\Psi}(t,z)dA(z),\quad \mathscr{J}(r)(t)\triangleq\tfrac{1}{2\pi}\int_{D_t}|z|^2dA(z).$$
		The notation $\nabla$ stands for the $L_\theta^2(\mathbb{T})$-gradient associated with the $L_\theta^2(\mathbb{T})$ normalized inner product
		$$\big\langle \rho_{1}, \rho_{2}\big\rangle_{L^{2}(\mathbb{T})}\triangleq\int_{\mathbb{T}}\rho_{1}(\theta)\rho_{2}(\theta)d\theta.$$		
	\end{enumerate}
\end{lem}
\begin{proof}
	\textbf{(i)} First, from the polar parametrization \eqref{prmz}, it is easy to check that the left hand-side of \eqref{complex vp eq} writes
	$$\mbox{Im}\left(\partial_{t}z(t,\theta)\overline{\partial_{\theta}z(t,\theta)}\right)=-\partial_{t}r(t,\theta)-\Omega\partial_{\theta}r(t,\theta).$$
Now we study the right hand-side of \eqref{complex vp eq}. Combining \eqref{velocity Ealpha-1} and \eqref{sym Bessel}, we deduce
\begin{align*}
	\mbox{Im}\left(\mathbf{v}\big(t,z(t,\theta)\big)\overline{\partial_{\theta}z(t,\theta)}\right)&=-\displaystyle\int_{\mathbb{T}}\log\big(|z(t,\theta)-z(t,\eta)|\big)\mbox{Im}\left(\partial_{\eta}z(t,\eta)\overline{\partial_{\theta}z(t,\theta)}\right)d\eta\\
	&\quad-\int_{\mathbb{T}}K_{0}\left(\tfrac{1}{\alpha}|z(t,\theta)-z(t,\eta)|\right)\mbox{Im}\left(\partial_{\eta}z(t,\eta)\overline{\partial_{\theta}z(t,\theta)}\right)d\eta.
\end{align*}
We conclude by remarking that
\begin{align*}\mbox{Im}\left(\partial_{\eta}z(t,\eta)\overline{\partial_{\theta}z(t,\theta)}\right)=\partial^2_{\theta\eta}\mbox{Im}\left(z(t,\eta)\overline{z(t,\theta)}\right)=\partial_{\theta\eta}^{2}\Big(R(t,\eta)R(t,\theta)\sin(\eta-\theta)\Big).
\end{align*}
\textbf{(ii)} Using polar change of coordinates and \eqref{prmz}, we obtain
$$\mathscr{J}(r)(t)=\int_{\mathbb{T}}\int_{0}^{R(t,\theta)}\ell^3d\ell d\theta=\tfrac{1}{4}\int_{\mathbb{T}}\big(1+2r(t,\theta)\big)^2d\theta,$$
leading to
$$\nabla \mathscr{J}(r)=1+2r\quad\textnormal{and}\quad\tfrac{1}{2}\Omega\partial_{\theta}\nabla \mathscr{J}(r)=\Omega\partial_{\theta}r.$$
Now that we have treated the corresponding linear term in \eqref{Ea eq r}, we can assume $\Omega=0.$ From the complex notation, we have \begin{align*}
	\partial_{\theta}\boldsymbol{\Psi}(t,z(t,\theta))&=\nabla\boldsymbol{\Psi}(t,z(t,\theta))\cdot\partial_{\theta}z(t,\theta)\\
	&=\textnormal{Im}\left(\mathbf{v}(t,z(t,\theta))\overline{\partial_{\theta}z(t,\theta)}\right)\\
	&=-F^{\textnormal{\tiny{E}}}[r](t,\theta)-F^{\textnormal{\tiny{SW}}}[r](\alpha,t,\theta).
\end{align*}
From \eqref{Psi Ea intro}, we can write
$$\boldsymbol{\Psi}(t,z)=\int_{D_t}\widetilde{\mathbf{G}}(z,\xi)dA(\xi),\qquad\widetilde{\mathbf{G}}(z,\xi)\triangleq \mathbf{G}(|z-\xi|).$$
The kernel $\widetilde{\mathbf{G}}$ satisfies the following symmetry property
$$\widetilde{\mathbf{G}}(z,\xi)=\widetilde{\mathbf{G}}(\xi,z).$$
Consequently, we can apply the general \cite[Prop. 2.1]{HR21-1} giving 
$$\nabla \mathscr{E}(r)(t,\theta)=-2\boldsymbol{\Psi}\big(t,z(t,\theta)\big).$$
This achieves the proof of Lemma \ref{lem eq Ea r + HAM}.
\end{proof}
We shall now briefly discuss the symplectic framework behind the Hamiltonian equation \eqref{Heq Ea}. First notice that one deduces from \eqref{Heq Ea} that the space average is preserved along the motion so we may assume it zero and work in the following phase space
\begin{equation}\label{phase space}
	L_0^2(\mathbb{T})\triangleq\Big\{r=\sum_{j\in\mathbb{Z}^*}r_j\mathbf{e}_j\quad\textnormal{s.t.}\quad r_{-j}=\overline{r_j},\quad\sum_{j\in\mathbb{Z}^*}|r_j|^2<\infty\Big\},\qquad\mathbf{e}_j(\theta)\triangleq e^{\ii j\theta}.
\end{equation}
The symplectic form $\mathscr{W}$ on $L_0^2(\mathbb{T})$ generated by \eqref{Heq Ea} writes
\begin{equation}\label{scrW}
	\mathscr{W}(r,\widetilde{r})\triangleq\int_{\mathbb{T}}\partial_{\theta}^{-1}r(\theta)\widetilde{r}(\theta)d\theta,\qquad\partial_{\theta}^{-1}r\triangleq\sum_{j\in\mathbb{Z}^*}\tfrac{r_j}{\ii j}\mathbf{e}_j.
\end{equation}
The associated Hamiltonian vector-field $X_{\mathscr{H}}$ is defined by
$$d\mathscr{H}(r)[\rho]=\mathscr{W}(X_{\mathscr{H}}(r),\rho),\qquad X_{\mathscr{H}}(r)\triangleq\partial_{\theta}\nabla \mathscr{H}(r).$$
We shall also present the reversibility property of the Hamiltonian $\mathscr{H}$. For that purpose we introduce the following involution on $L_0^2(\mathbb{T})$
\begin{equation}\label{invscrS}
	(\mathscr{S}r)(\theta)\triangleq r(-\theta),\qquad\mathscr{S}^2=\textnormal{Id}.
\end{equation}
Then changes of variables give
$$\mathscr{S}\circ F^{\textnormal{\tiny{E}}}=-F^{\textnormal{\tiny{E}}}\circ\mathscr{S},\qquad \mathscr{S}\circ F^{\textnormal{\tiny{SW}}}=-F^{\textnormal{\tiny{SW}}}\circ\mathscr{S}$$
implying in turn
$$\mathscr{H}\circ\mathscr{S}=\mathscr{H},\qquad X_{\mathscr{H}}\circ\mathscr{S}=-\mathscr{S}\circ X_{\mathscr{H}}.$$
Now, in view of applying a Nash-Moser scheme, we shall compute the linearized operator both at the equilibrium and at a general state close to it. It is proved in \cite[Lem. 3.1 and 3.2]{HR21-1} that
\begin{equation}\label{dFE}
	-d_rF^{\textnormal{\tiny{E}}}(r)[\rho]=-\partial_{\theta}\big(V_{r}^{\textnormal{\tiny{E}}}\rho\big)+\partial_{\theta}\mathbf{L}_r^{\textnormal{\tiny{E}}}(\rho),
\end{equation}
with
\begin{align}
	V_r^{\textnormal{\tiny{E}}}(t,\theta)&\triangleq\tfrac{1}{R(t,\theta)}\int_{\mathbb{T}}\log\big(A_{r}(t,\theta,\eta)\big)\partial_{\eta}\big(R(t,\eta)\sin(\eta-\theta)\big)d\eta,\label{VrE}\\
	\mathbf{L}_{r}^{\textnormal{\tiny{E}}}(\rho)(t,\theta)&\triangleq\int_{\mathbb{T}}\rho(t,\eta)\log\left(A_{r}(t,\theta,\eta)\right)d\eta\label{mathbfLrE}
\end{align}
and 
$$-d_rF^{\textnormal{\tiny{E}}}(0)[\rho]=\tfrac{1}{2}\partial_{\theta}\rho+\partial_{\theta}\mathcal{K}\ast\rho=\ii\sum_{j\in\mathbb{Z}}j\boldsymbol{\Omega}_{|j|}^{\textnormal{\tiny{E}}}\rho_j\mathbf{e}_j,\qquad\mathcal{K}(\theta)\triangleq \tfrac{1}{2}\log\left(\sin^{2}\left(\tfrac{\theta}{2}\right)\right).$$
We recall that $A_{r}$ and $R$ are defined by \eqref{Ar} and \eqref{prmz}, respectively. In addition, it has been proved in \cite[Lem. 3.1 and 3.2]{HR21} that
\begin{equation}\label{dFSW}
	d_rF^{\textnormal{\tiny{SW}}}(r)[\rho]=\partial_{\theta}\big(V_{r}^{\textnormal{\tiny{SW}}}\rho\big)-\partial_{\theta}\mathbf{L}_r^{\textnormal{\tiny{SW}}}(\rho),
\end{equation}
with
\begin{align}
	V_r^{\textnormal{\tiny{SW}}}(\alpha,t,\theta)&\triangleq\tfrac{1}{R(t,\theta)}\int_{\mathbb{T}}K_0\big(\tfrac{1}{\alpha}A_{r}(t,\theta,\eta)\big)\partial_{\eta}\big(R(t,\eta)\sin(\eta-\theta)\big)d\eta,\label{VrSW}\\
	\mathbf{L}_{r}^{\textnormal{\tiny{SW}}}(\rho)(\alpha,t,\theta)&\triangleq\int_{\mathbb{T}}\rho(t,\eta)K_0\left(\tfrac{1}{\alpha}A_{r}(t,\theta,\eta)\right)d\eta\label{mathbfLrSW}
\end{align}
and
$$d_rF^{\textnormal{\tiny{SW}}}(0)[\rho]=I_{1}\left(\tfrac{1}{\alpha}\right)K_{1}\left(\tfrac{1}{\alpha}\right)\partial_{\theta}\rho-\partial_{\theta}\mathcal{Q}_{\alpha}\ast\rho=\ii\sum_{j\in\mathbb{Z}}j\boldsymbol{\Omega}_{|j|}^{\textnormal{\tiny{SW}}}\left(\tfrac{1}{\alpha}\right)\rho_{j}\mathbf{e}_j,\qquad\mathcal{Q}_{\alpha}(\theta)\triangleq K_0\left(\tfrac{2}{\alpha}\left|\sin\left(\tfrac{\theta}{2}\right)\right|\right).$$
Gathering the previous results leads to the following lemma giving the general expression of the linearized equation and stating that the equilibrium is given by a Fourier multiplier associated with an integrable Hamiltonian system.
\begin{lem}\label{lem lin eq QP Ea}
	The following assertions hold true.
	\begin{enumerate}[label=(\roman*)]
	\item The linearization of \eqref{Heq Ea} at a general state $r$ writes
	\begin{equation}\label{lin eq}
		\partial_{t}\rho=-\partial_{\theta}\Big(V_{r}\rho+\mathbf{L}_{r}(\rho)\Big),
	\end{equation}
	where $V_{r}$ and $\mathbf{L}_{r}$ are respectively defined by
	\begin{equation}
		V_{r}(\alpha,t,\theta)\triangleq\Omega+V_r^{\textnormal{\tiny{SW}}}(\alpha,t,\theta)-V_r^{\textnormal{\tiny{E}}}(t,\theta),\qquad
		\mathbf{L}_{r}\triangleq\mathbf{L}_{r}^{\textnormal{\tiny{E}}}+\mathbf{L}_{r}^{\textnormal{\tiny{SW}}}.\label{Vr bfLr}
	\end{equation}
	In addition, we have the following symmetry property
	\begin{equation}\label{symVr}
		r(-t,-\theta)=r(t,\theta)\quad\Rightarrow\quad V_{r}(\alpha,-t,-\theta)=V_{r}(\alpha,t,\theta).
	\end{equation}
\item 
\begin{enumerate}
	\item At $r=0$, the equation \eqref{lin eq} becomes
	\begin{equation}\label{Ham eq 0 Ea}
		\partial_{t}\rho=\partial_{\theta}\mathrm{L}(\alpha)\rho=\partial_{\theta}\nabla \mathscr{H}_{\mathrm{L}}(\rho),
	\end{equation}
	where $\mathrm{L}(\alpha)$ is the self-adjoint operator given by
	\begin{equation}\label{LaKa}
		\mathrm{L}(\alpha)\triangleq -V_0(\alpha)-\mathcal{K}_{\alpha}\ast\cdot,\qquad V_0(\alpha)\triangleq \Omega+\frac{1}{2}-I_1\left(\tfrac{1}{\alpha}\right)K_1\left(\tfrac{1}{\alpha}\right),\qquad
		\mathcal{K}_{\alpha}\triangleq \mathcal{K}-\mathcal{Q}_{\alpha}.
	\end{equation}
	The equation \eqref{Ham eq 0 Ea} is generated by the quadratic Hamiltonian
	\begin{equation}\label{HL Ea}
		\mathscr{H}_{\mathrm{L}}(\rho)\triangleq \tfrac{1}{2}\big\langle\mathrm{L}(\alpha)\rho,\rho\big\rangle_{L^{2}(\mathbb{T})}.
	\end{equation}
	\item In Fourier expansion, the solutions of \eqref{Ham eq 0 Ea} take the form
	$$\rho(t,\theta)=\sum_{j\in\mathbb{Z}^*}\rho_{j}(0)e^{\ii\left(j\theta-\Omega_{j}^{\textnormal{\tiny{E}}}(\alpha)t\right)},$$
	where for all $j\in\mathbb{Z}^*$,
	\begin{equation}\label{Omegajalpha}
		\Omega_{j}^{\textnormal{\tiny{E}}}(\alpha)\triangleq j\Big[\Omega+\tfrac{|j|-1}{2|j|}-\big[I_1\left(\tfrac{1}{\alpha}\right)K_1\left(\tfrac{1}{\alpha}\right)-I_{|j|}\left(\tfrac{1}{\alpha}\right)K_{|j|}\left(\tfrac{1}{\alpha}\right)\big]\Big]=j\Big[\Omega+\boldsymbol{\Omega}_{|j|}^{\textnormal{\tiny{E}}}(\alpha)\Big],
	\end{equation}
	with $\boldsymbol{\Omega}_j^{\textnormal{\tiny{E}}}(\alpha)$ are the frequencies obtained in the periodic case and defined in \eqref{def eigenvalues Easc}. The operator $\mathrm{L}(\alpha)$ and the Hamiltonian $\mathscr{H}_{\mathrm{L}}$ also write
	\begin{align}\label{Ham-Fourier}
		\mathrm{L}(\alpha)\rho=-\sum_{j\in\mathbb{Z}^{*}}\tfrac{\Omega_{j}^{\textnormal{\tiny{E}}}(\alpha)}{j}\rho_{j}\mathbf{e}_j\quad\mbox{ and }\quad \mathscr{H}_{\mathrm{L}}\rho=-\sum_{j\in\mathbb{Z}^{*}}\tfrac{\Omega_{j}^{\textnormal{\tiny{E}}}(\alpha)}{2j}|\rho_{j}|^{2}.
	\end{align}
\end{enumerate}
\end{enumerate}
\end{lem}
\subsection{Transversality and linear quasi-periodic solutions}
The aim of this section is to find quasi-periodic solutions for the linearized equation \eqref{Ham eq 0 Ea}, which is the basis for expecting to get them at the nonlinear level. The result reads as follows 
\begin{prop}\label{lemsoleqEa}
	Let $(\alpha_0,\alpha_1,d)$ as in \eqref{alf0alf1}-\eqref{p-d}. Take $\mathbb{S}\subset\mathbb{N}^*$ with $|\mathbb{S}|=d.$ Then, there exists a Cantor-like set $$\mathscr{C}_{\textnormal{Eq}}\subset[\alpha_0,\alpha_1],\qquad|\mathscr{C}_{\textnormal{Eq}}|=\alpha_1-\alpha_0$$
	such that for any $\alpha\in\mathscr{C}_{\textnormal{Eq}}$, every function in the form
	$$\rho(t,\theta)=\sum_{j\in\mathbb{S}}\rho_{j}\cos\left(j\theta-\Omega_{j}^{\textnormal{\tiny{E}}}(\alpha)t\right),\qquad\rho_{j}\in\mathbb{R}^*$$
	is a time quasi-periodic reversible solution to \eqref{Ham eq 0 Ea} with frequency vector
	\begin{equation}\label{omega Eq Ea}
		\omega_{\textnormal{Eq}}(\alpha)=\left(\Omega_{j}^{\textnormal{\tiny{E}}}(\alpha)\right)_{j\in\mathbb{S}}.
	\end{equation}
\end{prop}
The proof of this proposition is very similar to \cite[Prop. 3.1]{HR21} so we refer the reader to the corresponding paper. We mention that in the proof, to measure the Cantor set $\mathscr{C}_{\textnormal{Eq}}$, we make appeal to the following Rüssmann Lemma which can be found in \cite[Thm. 17.1]{R01}.
\begin{lem}\label{lem Russmann measure}
	Let $q_{0}\in\mathbb{N}^{*},$ $a,b\in\mathbb{R}$ with $a<b$ and $\mathtt{m},\mathtt{M}\in(0,\infty).$ Let $f\in C^{q_{0}}([a,b],\mathbb{R})$ such that
	\begin{equation}\label{trans cond Rlem}
		\inf_{x\in[a,b]}\max_{q\in\llbracket 0,q_{0}\rrbracket}\big|f^{(q)}(x)\big|\geqslant \mathtt{m}.
	\end{equation}
	Then, there exists $C=C(a,b,q_{0},\| f\|_{C^{q_{0}}([a,b],\mathbb{R})})>0$ such that 
	$$\Big|\left\lbrace x\in[a,b]\quad\textnormal{s.t.}\quad |f(x)|\leqslant \mathtt{M}\right\rbrace\Big|\leqslant C\frac{\mathtt{M}^{\frac{1}{q_{0}}}}{\mathtt{m}^{1+\frac{1}{q_{0}}}}\cdot$$
\end{lem}
To apply the previous lemma, we shall check the transversality condition \eqref{trans cond Rlem} for the equilibrium frequency vector $\omega_{\textnormal{Eq}}$ in \eqref{omega Eq Ea}. It is proved in Lemma \ref{lem trsvrslt Ea}-(i). Notice that the measure of the final Cantor set in Section \ref{sec non triv QP sol} generating quasi-periodic solution for the nonlinear model requires transversality conditions for the perturbed frequency vector. These latter are obtained by perturbative arguments from the one for the equilibrium frequency vector stated in Lemma \ref{lem trsvrslt Ea} and which are themselves deduced from the non-degeneracy of the unperturbed frequency vector proved in Lemma \ref{lem non-deg Ea}. First we start by giving some properties of the frequencies \eqref{Omegajalpha}.
\begin{lem}\label{properties omegajalpha}
	The following properties hold true.
	\begin{enumerate}[label=(\roman*)]
		\item $\forall\, \alpha>0,\quad \Omega_{j}^{\textnormal{\tiny{E}}}(\alpha)\underset{j\rightarrow+\infty}{\sim}V_{0}(\alpha)j,\,$ with $\, V_{0}(\alpha)$ as in \eqref{LaKa}.
		\item For all $\alpha>0,$ the sequence $(\Omega_{j}^{\textnormal{\tiny{E}}}(\alpha))_{j\in\mathbb{N}^{*}}$ is strictly increasing.
		\item For any $j\in\mathbb{Z}^{*},$ we have
		$$\forall \,\alpha>0,\quad \Big|\Omega_{j}^{\textnormal{\tiny{E}}}(\alpha)\Big|\geqslant\Omega|j|.$$
		\item For any $j,j'\in\mathbb{Z}^{*}$, we have 
		$$\forall\, \alpha>0,\quad \Big|\Omega_{j}^{\textnormal{\tiny{E}}}(\alpha)-\Omega_{j'}^{\textnormal{\tiny{E}}}(\alpha)\Big|\geqslant\Omega|j-j'|.$$
		\item Given $0<\alpha_0<\alpha_1$ and $q_0\in\mathbb{N},$ there exists $C_0>0$ such that
		$$\forall j,j_0\in\mathbb{Z}^{*},\quad \max_{q\in\llbracket 0,q_0\rrbracket}\sup_{\alpha\in[\alpha_0,\alpha_1]}\left|\partial_{b}^{q}\Big(\Omega_{j}^{\textnormal{\tiny{E}}}(\alpha)-\Omega^{\textnormal{\tiny{E}}}_{j_0}(\alpha)\Big)\right|\leqslant C_0|j-j_0|.$$
	\end{enumerate}
\end{lem}
\begin{proof}
	\textbf{(i)} and \textbf{(ii)} follow immediately from Lemma \ref{lem mono ev Ea} and \eqref{Omegajalpha}.\\
	\textbf{(iii)} Due to the symmetry \eqref{Omegajalpha}, it sufficies to study the case $j\in\mathbb{N}^*.$ From Lemma \ref{lem mono ev Ea}, we get
	$$\forall\alpha>0,\quad\boldsymbol{\Omega}_j^{\textnormal{\tiny{E}}}(\alpha)\geqslant0.$$
	Consequently, 
	$$\forall\alpha>0,\quad\Omega_j^{\textnormal{\tiny{E}}}(\alpha)\geqslant j\Omega.$$
	\textbf{(iv)} Due to the symmetry \eqref{Omegajalpha}, it suffices to prove that
	$$\forall\alpha>0,\quad\forall(j,j')\in(\mathbb{N}^*)^2,\quad\Big|\Omega_{j}^{\textnormal{\tiny{E}}}(\alpha)\pm\Omega_{j'}^{\textnormal{\tiny{E}}}(\alpha)\Big|\geqslant\Omega|j\pm j'|.$$
	The point (ii) allows us to restrict the discussion to the case $j\geqslant j'.$ We can write
	$$\Omega_{j}^{\textnormal{\tiny{E}}}(\alpha)\pm\Omega_{j'}^{\textnormal{\tiny{E}}}(\alpha)=\Omega(j\pm j')+j\boldsymbol{\Omega}_{j}^{\textnormal{\tiny{E}}}(\alpha)\pm j'\boldsymbol{\Omega}_{j'}^{\textnormal{\tiny{E}}}(\alpha).$$
	In view of Lemma \ref{lem mono ev Ea}, we obtain
	$$j\boldsymbol{\Omega}_{j}^{\textnormal{\tiny{E}}}(\alpha)\pm j'\boldsymbol{\Omega}_{j'}^{\textnormal{\tiny{E}}}(\alpha)\geqslant 0.$$
	Hence
	$$\Omega_{j}^{\textnormal{\tiny{E}}}(\alpha)\pm\Omega_{j'}^{\textnormal{\tiny{E}}}(\alpha)\geqslant\Omega(j\pm j').$$
	\textbf{(v)} As before, by symmetry, it sufficies to prove that
	$$\forall (j,j_0)\in(\mathbb{N}^*)^2,\quad\forall\alpha\in[\alpha_{0},\alpha_{1}],\quad\forall q\in\llbracket 0,q_0\rrbracket,\quad\Big|\partial_{\alpha}^{q}\Big(\Omega_{j}^{\textnormal{\tiny{E}}}(\alpha)\pm\Omega_{j_0}^{\textnormal{\tiny{E}}}(\alpha)\Big)\Big|\leqslant C_0|j\pm j_0|.$$
	Let us start with the difference. We can write form \eqref{Omegajalpha}
	$$\Omega_{j}^{\textnormal{\tiny{E}}}(\alpha)-\Omega_{j_0}^{\textnormal{\tiny{E}}}(\alpha)=\Omega(j- j_0)+\frac{j-j_0}{2}+\Omega_{j_0}^{\textnormal{\tiny{SW}}}\left(\tfrac{1}{\alpha}\right)-\Omega_{j}^{\textnormal{\tiny{SW}}}\left(\tfrac{1}{\alpha}\right),\qquad\Omega_{j}^{\textnormal{\tiny{SW}}}(\lambda)\triangleq j\boldsymbol{\Omega}_{|j|}^{\textnormal{\tiny{SW}}}(\lambda).$$
	Now it has been proved in \cite[Lem. 3.3-(vi)]{HR21} that for some $0<\lambda_0<\lambda_1,$ there exists $C>0$ such that
	\begin{equation}\label{e-diff freq QG}
		\forall (j,j_0)\in(\mathbb{N}^*)^2,\quad\forall\lambda\in[\lambda_{0},\lambda_{1}],\quad\forall q\in\llbracket 0,q_0\rrbracket,\quad\Big|\partial_{\lambda}^{q}\Big(\Omega_{j}^{\textnormal{\tiny{SW}}}(\lambda)\pm \Omega_{j_0}^{\textnormal{\tiny{SW}}}(\lambda)\Big)\Big|\leqslant C|j\pm j_0|.
	\end{equation}
	We warn the reader about the difference of definitions of the frequencies $\Omega_{j}^{\textnormal{\tiny{SW}}}$ between this article and \cite{HR21} (with the $\Omega$ missing). But this has no impact here. Hence, we conclude by the triangle inequality that for some $C_0>0$
	$$\forall (j,j_0)\in(\mathbb{N}^*)^2,\quad\forall\alpha\in[\alpha_{0},\alpha_{1}],\quad\forall q\in\llbracket 0,q_0\rrbracket,\quad\Big|\partial_{\alpha}^{q}\Big(\Omega_{j}^{\textnormal{\tiny{E}}}(\alpha)-\Omega_{j_0}^{\textnormal{\tiny{E}}}(\alpha)\Big)\Big|\leqslant C_0|j-j_0|.$$
	We now trun to the additional case. We can write
	$$0<\Omega_{j}^{\textnormal{\tiny{E}}}(\alpha)+\Omega_{j_0}^{\textnormal{\tiny{E}}}(\alpha)=\Omega(j+j_0)+\Big(j\tfrac{j-1}{2j}+j_0\tfrac{j_0-1}{2j_0}\Big)-\Omega_{j_0}^{\textnormal{\tiny{SW}}}\left(\tfrac{1}{\alpha}\right)-\Omega_{j}^{\textnormal{\tiny{SW}}}\left(\tfrac{1}{\alpha}\right).$$
	Notice that
	$$\forall j\in\mathbb{N}^*,\quad\tfrac{j-1}{2j}\leqslant\tfrac{1}{2}\quad\textnormal{and}\quad \Omega_{j}^{\textnormal{\tiny{SW}}}\left(\tfrac{1}{\alpha}\right)>0.$$
	Thus,
	$$0<\Omega_{j}^{\textnormal{\tiny{E}}}(\alpha)+\Omega_{j_0}^{\textnormal{\tiny{E}}}(\alpha)\leqslant\big(\Omega+\tfrac{1}{2}\big)(j+j_0).$$
	Combined with \eqref{e-diff freq QG}, this ends the proof of Lemma \ref{properties omegajalpha}.
\end{proof}
For a fixed finite set of Fourier modes
\begin{equation}\label{def tan modes}
	\mathbb{S}=\{j_1,\ldots,j_d\}\subset\mathbb{N}^{*},\qquad j_1<\ldots<j_d,\qquad d\in\mathbb{N}^*,
\end{equation}
we define the equilibrium frequency vector
\begin{equation}\label{def freq vec eqa}
	\omega_{\textnormal{Eq}}(\alpha)=\big(\Omega_{j}^{\textnormal{\tiny{E}}}(\alpha)\big)_{j\in\mathbb{S}}.
\end{equation}
Then, in view of the measure of the final Cantor set, we may check the Rüssmann conditions for the unperturped frequency vector \eqref{def freq vec eqa}. They are obtained by using the following non-degeneracy conditions.
\begin{lem}\label{lem non-deg Ea}
Let $(\alpha_0,\alpha_1)$ as in \eqref{alf0alf1}. The equilibrium frequency vector $\omega_{\textnormal{Eq}}$ and the vector-valued functions $(\omega_{\textnormal{Eq}},V_0)$ and $(\omega_{\textnormal{Eq}},V_0,1)$ are non-degenerate on $[\alpha_0,\alpha_1]$, namely the curves
$$\begin{array}{l}
	\alpha\in[\alpha_0,\alpha_1]\mapsto\omega_{\textnormal{Eq}}(\alpha),\\
	\alpha\in[\alpha_0,\alpha_1]\mapsto(\omega_{\textnormal{Eq}}(\alpha),V_0(\alpha)),\\
	\alpha\in[\alpha_0,\alpha_1]\mapsto(\omega_{\textnormal{Eq}}(\alpha),V_0(\alpha),1)
\end{array}$$
are not contained in an hyperplane of $\mathbb{R}^{d}$, $\mathbb{R}^{d+1}$ and $\mathbb{R}^{d+2}$, respectively.
\end{lem}
\begin{proof}
	$\blacktriangleright$ Assume that there exists $(c_1,\ldots,c_d)\in\mathbb{R}^d$ such that
	$$\forall\alpha\in[\alpha_0,\alpha_1],\quad\sum_{k=1}^{d}c_k\Omega_{j_k}^{\textnormal{\tiny{E}}}(\alpha)=0,$$
	which is equivalent to
	\begin{align}\label{eq deg Ea}
		\forall\alpha\in[\alpha_0,\alpha_1],\quad\sum_{k=1}^{d}c_kj_k\Big(\Omega+\tfrac{j_k-1}{2j_{k}}\Big)=\sum_{k=1}^{d}c_k\Omega_{j_k}^{\textnormal{\tiny{SW}}}\left(\tfrac{1}{\alpha}\right).
	\end{align}
By analyticity of the product $I_{n}K_{n}$ on $\{z\in\mathbb{C}\quad\textnormal{s.t.}\quad\textnormal{Re}(z)>0\}$ for any $n\in\mathbb{N}^*$, then by continuation principle, the previous identity is still true for $\alpha>0.$ Taking the limit $\alpha\to0$ in the previous relation implies from \eqref{asymp large z}
	$$\sum_{k=1}^{d}c_kj_k\Big(\Omega+\tfrac{j_k-1}{2j_{k}}\Big)=0.$$
	The equation \eqref{eq deg Ea} is reduced to 
	$$\sum_{k=1}^{d}c_k\Omega_{j_k}^{\textnormal{\tiny{SW}}}\left(\tfrac{1}{\alpha}\right)=0.$$
	Then, proceeding as in \cite[Lem. 3.4]{HR21}, the asymptotic expansion of large argument for $I_jK_j$ provides an invertible Vandermonde system leading to $\forall k\in\llbracket 1,d\rrbracket,\, c_k=0.$\\
	$\blacktriangleright$ Assume that there exists $(c_1,\ldots,c_d,c_{d+1})\in\mathbb{R}^{d+1}$ (resp. $(c_1,\ldots,c_d,c_{d+1},c_{d+2})\in\mathbb{R}^{d+2}$) such that
	$$\forall\alpha\in[\alpha_0,\alpha_1],\qquad (\textnormal{resp. }c_{d+2}+)\quad c_{d+1}V_0(\alpha)+\sum_{k=1}^{d}c_k\Omega_{j_k}^{\textnormal{\tiny{E}}}(\alpha)=0,$$
	which is equivalent to the fact that for any $\alpha\in[\alpha_0,\alpha_1],$
	\begin{align}\label{eq deg Ea-2}
		(\textnormal{resp. }c_{d+2}+)\quad c_{d+1}\big(\Omega+\tfrac{1}{2}\big)+\sum_{k=1}^{d}c_kj_k\Big(\Omega+\tfrac{j_k-1}{2j_{k}}\Big)=c_{d+1}I_1\left(\tfrac{1}{\alpha}\right)K_1\left(\tfrac{1}{\alpha}\right)+\sum_{k=1}^{d}c_k\Omega_{j_k}^{\textnormal{\tiny{SW}}}\left(\tfrac{1}{\alpha}\right).
	\end{align}
As in the previous point, this identity can be extended to $(0,\infty)$ and taking the limit $\alpha\to0$, one gets by \eqref{asymp large z}
$$(\textnormal{resp. }c_{d+2}+)\quad c_{d+1}\big(\Omega+\tfrac{1}{2}\big)+\sum_{k=1}^{d}c_kj_k\Big(\Omega+\tfrac{j_k-1}{2j_{k}}\Big)=0.$$
Inserting this information into \eqref{eq deg Ea-2} yields
$$\forall\alpha>0,\quad c_{d+1}I_1\left(\tfrac{1}{\alpha}\right)K_1\left(\tfrac{1}{\alpha}\right)+\sum_{k=1}^{d}c_k\Omega_{j_k}^{\textnormal{\tiny{SW}}}\left(\tfrac{1}{\alpha}\right)=0.$$
This equation has also been studied in \cite[Lem. 3.4]{HR21} leading to $c_1=\ldots=c_d=c_{d+1}=0$ (resp. supplemented by $c_{d+2}=0$). This achieves the proof of Lemma \ref{lem non-deg Ea}. 
\end{proof}
Now we shall prove the transversality conditions for the equilibrium frequency vector.
\begin{lem}{\textnormal{[Transversality]}}\label{lem trsvrslt Ea}
	Let $(\alpha_0,\alpha_1)$ as in \eqref{alf0alf1}. Then, there exist $q_0\in\mathbb{N}$ and $\rho_{0}>0$ such that the following results hold true. Recall that $\omega_{\textnormal{Eq}}$ and $\Omega_j^{\textnormal{\tiny{E}}}$ are defined in \eqref{def freq vec eqa} and \eqref{Omegajalpha} respectively.
		\begin{enumerate}[label=(\roman*)]
			\item For any $l\in\mathbb{Z}^{d }\setminus\{0\},$ we have
			$$
			\inf_{\alpha\in[\alpha_{0},\alpha_{1}]}\max_{q\in\llbracket 0, q_{0}\rrbracket}\Big|\partial_{\alpha}^{q}\omega_{\textnormal{Eq}}(\alpha)\cdot l\Big|\geqslant\rho_{0}\langle l\rangle.
			$$
			\item For any $ (l,j)\in\mathbb{Z}^{d }\times (\mathbb{N}^*\setminus\mathbb{S})$ 
				$$
				\quad\inf_{\alpha\in[\alpha_{0},\alpha_{1}]}\max_{q\in\llbracket 0, q_{0}\rrbracket}\Big|\partial_{\alpha}^{q}\Big(\omega_{\textnormal{Eq}}(\alpha)\cdot l\pm jV_0(\alpha)\Big)\Big|\geqslant\rho_{0}\langle l\rangle.
				$$
			\item  For any $ (l,j)\in\mathbb{Z}^{d }\times (\mathbb{N}^*\setminus\mathbb{S})$ 
			$$
			\quad\inf_{\alpha\in[\alpha_{0},\alpha_{1}]}\max_{q\in\llbracket 0, q_{0}\rrbracket}\Big|\partial_{\alpha}^{q}\Big(\omega_{\textnormal{Eq}}(\alpha)\cdot l\pm\Omega_{j}^{\textnormal{\tiny{E}}}(\alpha)\Big)\Big|\geqslant\rho_{0}\langle l\rangle.
			$$
			\item For any $ l\in\mathbb{Z}^{d }, j,j^\prime\in\mathbb{N}^*\setminus\mathbb{S}$  with $(l,j)\neq(0,j^\prime),$ we have
			$$\,\quad\inf_{\alpha\in[\alpha_{0},\alpha_{1}]}\max_{q\in\llbracket 0, q_{0}\rrbracket}\Big|\partial_{\alpha}^{q}\Big(\omega_{\textnormal{Eq}}(\alpha)\cdot l+\Omega_{j}^{\textnormal{\tiny{E}}}(\alpha)\pm\Omega_{j^\prime}^{\textnormal{\tiny{E}}}(\alpha)\Big)\Big|\geqslant\rho_{0}\langle l\rangle.$$	
	\end{enumerate}
\end{lem}
\begin{proof}
	We shall prove the point \textbf{(iv)} which is the most difficult one. The arguments are similar in the other cases using the corresponding non-degeneracy conditions provided by Lemma \ref{lem non-deg Ea}. Fix some $l\in\mathbb{Z}^{d }$ and $j,j^\prime\in\mathbb{N}^*\setminus\mathbb{S}$  with $(l,j)\neq(0,j^\prime).$ If for some $c_{0}>0$
	$$|j\pm j'|\geqslant c_{0}\langle l\rangle,$$
	then applying the triangle inequality together with Lemma \ref{properties omegajalpha}-(iv), we get
	$$\Big|\omega_{\textnormal{Eq}}(\alpha)\cdot l+\Omega_{j}^{\textnormal{\tiny{E}}}(\alpha)\pm\Omega_{j'}^{\textnormal{\tiny{E}}}(\alpha)\Big|\geqslant\Big|\Omega_{j}^{\textnormal{\tiny{E}}}(\alpha)\pm\Omega_{j'}^{\textnormal{\tiny{E}}}(\alpha)\Big|-|\omega_{\textnormal{Eq}}(\alpha)\cdot l|\geqslant \Omega|j\pm j'|-C|l|\geqslant \langle l\rangle.$$
	Therefore it remains to check the proof for  indices satisfying 
	\begin{equation}\label{restri param}
		|j\pm j'|<c_{0}\langle l\rangle,\qquad  l\in\mathbb{Z}^{d}\setminus\{0\},\qquad j,j^\prime\in\mathbb{N}^*\setminus\mathbb{S}.
	\end{equation}
	We assume in view of a contradiction that for all $m\in\mathbb{N}$, there exist real numbers $l_{m}\in\mathbb{Z}^{d}\setminus\{0\}$, $j_{m},j'_{m}\in\mathbb{N}^*\setminus\mathbb{S}$ satisfying \eqref{restri param} and $\alpha_{m}\in[\alpha_{0},\alpha_{1}]$ such that 
	$$\max_{q\in\llbracket 0, m\rrbracket}\left|\partial_{\alpha}^{q}\left(\omega_{\textnormal{Eq}}(\alpha)\cdot\tfrac{l_{m}}{| l_{m}|}+\tfrac{\Omega_{j_{m}}^{\textnormal{\tiny{E}}}(\alpha)\pm \Omega_{j'_{m}}^{\textnormal{\tiny{E}}}(\alpha)}{| l_{m}|}\right)_{|\alpha=\alpha_m}\right|<\tfrac{1}{m+1}.$$ 
	This implies that
	\begin{equation}\label{maj seq diff}
		\forall q\in\mathbb{N},\quad \forall m\geqslant q,\quad \left|\partial_{\alpha}^{q}\left(\omega_{\textnormal{Eq}}(\alpha)\cdot\tfrac{l_{m}}{| l_{m}|}+\tfrac{\Omega_{j_{m}}^{\textnormal{\tiny{E}}}(\alpha)\pm \Omega_{j'_{m}}^{\textnormal{\tiny{E}}}(\alpha)}{| l_{m}|}\right)_{|\alpha=\alpha_m}\right|<\tfrac{1}{m+1}\cdot
	\end{equation}
	By compactness and \eqref{restri param}, up to considering a subsequence, we can assume that
	\begin{equation}\label{cv jl}
		\lim_{m\to\infty}\tfrac{l_{m}}{| l_{m}|}=\bar{c}\neq 0,\qquad\lim_{m\to\infty}\tfrac{j_{m}\pm j'_{m}}{|l_m|}=\bar{d},\qquad\lim_{m\to\infty}\alpha_{m}=\bar{\alpha}.
	\end{equation}
	Now we shall study separetely the cases whether the sequence $(l_m)_m$ is bounded or not.\\
	$\blacktriangleright$ Here we assume that the sequence $(l_{m})_{m}$ is bounded. Then, by compactness, we can assume, up to an extraction, that we have the following convergence
	$$\displaystyle \lim_{m\to\infty}l_{m}=\bar{l}\neq 0.$$
	Now according to \eqref{restri param} we have two sub-cases to discuss depending whether the sequences $(j_{m})_{m}$ and $(j^\prime_{m})_{m}$ are simultaneously bounded or unbounded.\\
	$\bullet$ We first study the case where the sequences $(j_{m})_{m}$ and $(j^\prime_{m})_{m}$ are bounded. Observe that the is the only case to consider if we work with the sign $"+"$ in \eqref{restri param}. Since they are sequences of integers, then by compactness we may assume, up to considering an extraction, that they are constant, namely
	$$\exists\, \bar{j},\bar{j^\prime}\in\mathbb{N}^*\setminus\mathbb{S}\qquad\textnormal{s.t.}\qquad \forall m\in\mathbb{N},\quad j_{m}=\bar{j}\quad\textnormal{and}\quad j^\prime_{m}=\bar{j^\prime}.$$
	Hence taking the limit as $m\rightarrow\infty$ in \eqref{maj seq diff}, we obtain
	$$\forall q\in\mathbb{N},\quad \partial_{\alpha}^{q}\left(\omega_{\textnormal{Eq}}({\alpha})\cdot\bar{l}+\Omega_{\bar{j}}^{\textnormal{\tiny{E}}}(\alpha)\pm\Omega_{\bar{j^\prime}}^{\textnormal{\tiny{E}}}(\alpha)\right)_{|\alpha=\overline\alpha}=0.$$
	Thus, the analytic function $\alpha\mapsto\omega_{\textnormal{Eq}}(\alpha)\cdot\bar{l}+\Omega_{\bar{j}}^{\textnormal{\tiny{E}}}(\alpha)\pm\Omega_{\bar{j^\prime}}^{\textnormal{\tiny{E}}}(\alpha)$ is identically zero which enters in contradiction with Lemma \ref{lem non-deg Ea} up to replacing $\omega_{\textnormal{Eq}}$ by $(\omega_{\textnormal{Eq}},\Omega_{\bar{j}}^{\textnormal{\tiny{E}}})$ or $(\omega_{\textnormal{Eq}},\Omega_{\bar{j}}^{\textnormal{\tiny{E}}},\Omega_{\bar{j^\prime}}^{\textnormal{\tiny{E}}})$.\\
	$\bullet$ Now we study the case where $(j_{m})_{m}$ and $(j'_{m})_{m}$ are both unbounded and without loss of generality we can assume that 
	\begin{equation}\label{cv jm}
		\lim_{m\to\infty}j_{m}= \lim_{m\to\infty}j'_{m}=\infty.
	\end{equation}
	Notice that here this case only concerns the situation with a sign $"-"$ in \eqref{restri param}. Nevertheless, for later purposes, we may treat both sign situations. From the expression \eqref{Omegajalpha} we can write
	\begin{align*}
		\Omega_{j_m}^{\textnormal{\tiny{E}}}(\alpha)\pm\Omega_{j_{m}^\prime}^{\textnormal{\tiny{E}}}(\alpha)=&(j_m\pm j^\prime_{m})V_0(\alpha)-\big(\tfrac{1}{2}\pm\tfrac{1}{2}\big)\\
		&+(j_m\pm j^\prime_{m})(I_{j_m}K_{j_m})\big(\tfrac{1}{\alpha}\big)\pm j^\prime_{m}\Big((I_{j^\prime_{m}}K_{j^\prime_{m}})\big(\tfrac{1}{\alpha}\big)-(I_{j_m}K_{j_m})\big(\tfrac{1}{\alpha}\big)\Big),
	\end{align*}
	with $V_0(\alpha)$ as in \eqref{LaKa}. It has been proved in \cite[Lem. 3.5]{HR21} that for any $q\in\mathbb{N},$ and for any $0<\lambda_{0}<\lambda_{1},$
	\begin{align}\label{lim bes}
		\lim_{m\to\infty} \sup_{\lambda\in[\lambda_0,\lambda_1]}\Big|j_m^\prime\,\partial_\lambda^q(I_{j_m}K_{j_m}-I_{j_m^\prime}K_{j_m^\prime})(\lambda)\Big|=0,\qquad\lim_{m\to\infty}\sup_{\lambda\in[\lambda_0,\lambda_1]}\Big|\partial_{\lambda}^q(I_{j_m}K_{j_m})(\lambda)\Big|=0.
	\end{align}
	Therefore \eqref{cv jl} and \eqref{lim bes} imply for any $q\in\mathbb{N},$
	$$\lim_{m\to\infty}|l_m|^{-1}\partial_{\alpha}^{q}\left(\Omega_{j_{m}}^{\textnormal{\tiny{E}}}(\alpha)\pm\Omega_{j'_{m}}^{\textnormal{\tiny{E}}}(\alpha)\right)_{|\alpha=\alpha_m}=\partial_{\alpha}^{q}\,\Big(\bar{d}V_0(\alpha)+|\bar{l}|^{-1}\big(\tfrac{1}{2}\pm\tfrac{1}{2}\big)\Big)_{|\alpha=\overline\alpha}.$$ 
	By taking the limit as $m\rightarrow\infty$ in \eqref{maj seq diff}, we find				$$\forall q\in\mathbb{N},\quad \partial_{\alpha}^{q}\Big({\omega}_{\textnormal{Eq}}({\alpha})\cdot\bar{c}+\bar{d}V_0(\alpha)+|\bar{l}|^{-1}\big(\tfrac{1}{2}\pm\tfrac{1}{2}\big)\Big)_{|\alpha=\overline\alpha}=0.$$
	Thus, the analytic function $\alpha\mapsto{\omega}_{\textnormal{Eq}}(\alpha)\cdot\bar{c}+\bar{d}V_0(\alpha)+|\bar{l}|^{-1}\big(\tfrac{1}{2}\pm\tfrac{1}{2}\big)$ with $(\bar{c},\bar{d})\neq 0$ and $\big(\bar{c},\bar{d},|\bar{l}|^{-1}\big)\neq 0$ is vanishing which contradicts Lemma \ref{lem non-deg Ea}.\\
	$\blacktriangleright$ Now we treat the case where the sequence $(l_{m})_{m}$ is unbounded. Up to an extraction we can assume that 
	$$\lim_{m\to\infty}|l_{m}|=\infty.$$
	We shall distinguish three sub-cases.\\
	$\bullet$ We first assume that the sequences  $(j_{m})_{m}$ and $(j'_{m})_{m}$ are bounded. Hence we have convergences of the type \eqref{cv jm}. Then taking the limit in \eqref{maj seq diff} yields,
	$$\forall q\in\mathbb{N},\quad\partial_{\alpha}^{q}{\omega}_{\textnormal{Eq}}(\bar{\alpha})\cdot\bar{c}=0.$$
	which leads to a contradiction as before.\\
	$\bullet$ The case where the sequences  $(j_{m})_{m}$ and $(j^\prime_{m})_{m}$ are both unbounded is similar to what has been done previously.\\
	$\bullet$ Now we assume that the sequence $(j_{m})_{m}$ is unbounded and $(j^\prime_{m})_{m}$ is bounded (the symmetric case is similar).  Without loss of generality we can assume that 
	$$\lim_{m\to\infty}j_m=\infty,\qquad  j_{m}^\prime=\overline{j}.$$
	One obtains from \eqref{lim bes} and \eqref{cv jl}
	\begin{align*}
		\forall q\in\mathbb{N},\quad\lim_{m\to\infty}|l_m|^{-1}
		\partial_\alpha^q\Big(\Omega_{j_m}^{\textnormal{\tiny{E}}}(\alpha)\pm\Omega_{j_{m}^\prime}^{\textnormal{\tiny{E}}}(\alpha)\Big)_{|\alpha=\alpha_m}=\bar{d}\partial_{\alpha}^qV_{0}(\bar{\alpha}).
	\end{align*}
	Consequently, taking the limit $m\to\infty$ in \eqref{maj seq diff} gives 
	$$\forall q\in\mathbb{N},\quad\partial_{\alpha}^{q}\Big({\omega}_{\textnormal{Eq}}({\alpha})\cdot\bar{c}+\bar{d}V_0(\alpha)\Big)_{\alpha=\overline\alpha}=0.$$
	Thus, the analytic function $\alpha\mapsto{\omega}_{\textnormal{Eq}}(\alpha)\cdot\bar{c}+\bar{d}V_0(\alpha)$ is identically zero with $(\bar{c},\bar{d})\neq0$ which contradicts Lemma \ref{lem non-deg Ea}. This completes the proof of Lemma \ref{lem trsvrslt Ea}.
\end{proof}

\subsection{The functional of interest and associated tame estimates}
In this subsection, we shall reformulate the problem in terms of embedded tori through the introduction of action-angle variables. This leads to look for the zeros of a nonlinear functional. Observe that the equation \eqref{Heq Ea} can be seen as a quasilinear perturbation of its linearization at the equilibrium state, namely
\begin{equation}\label{ql per}
	\partial_{t}r=\partial_{\theta}\mathrm{L}(\alpha)(r)+X_{\mathscr{P}}(r),\qquad X_{\mathscr{P}}(r)\triangleq\tfrac{1}{2}\partial_{\theta}r+\partial_{\theta}\mathcal{K}_{\alpha}\ast r+F^{\textnormal{\tiny{E}}}[r]+F^{\textnormal{\tiny{SW}}}[r],
\end{equation}
with $F^{\textnormal{\tiny{E}}},$ $F^{\textnormal{\tiny{SW}}}$ and $\mathrm{L}(\alpha)$ as in \eqref{FE}, \eqref{FSW} and \eqref{LaKa}. The smallness property is encoded by the introduction of a small parameter $\varepsilon.$ Then we consider the rescalling $r\mapsto\varepsilon r$ with $r$ bounded. Therefore \eqref{Heq Ea} becomes
\begin{equation}\label{pertham}
	\partial_{t}r=\partial_{\theta}\mathrm{L}(\alpha)(r)+\varepsilon X_{\mathscr{P}_{\varepsilon}}(r),\qquad X_{\mathscr{P}_{\varepsilon}}(r)\triangleq \varepsilon^{-2}X_{\mathscr{P}}(\varepsilon r).
\end{equation}
Remark that \eqref{pertham} can be written in the Hamiltonian form
$$\partial_{t}r=\partial_{\theta}\nabla \mathcal{H}_{\varepsilon}(r),\qquad\mathcal{H}_{\varepsilon}(r)\triangleq\varepsilon^{-2}\mathscr{H}(\varepsilon r)\triangleq \mathscr{H}_{\mathrm{L}}(r)+\varepsilon \mathscr{P}_{\varepsilon}(r),$$
with $\mathscr{H}_{\mathrm{L}}$ as in \eqref{HL Ea} and $\varepsilon \mathscr{P}_{\varepsilon}(r)$ containing the higher order terms more than cubic. We shall now reformulate the problem in terms of embedded tori. For that purpose, we introduce the action-angle variables. This is done in the following way. Introducing the symmetrized tangential sets $\overline{\mathbb{S}}$ and $\mathbb{S}_{0}$ associated to $\mathbb{S}$ in \eqref{def tan modes}
$$\overline{\mathbb{S}}\triangleq\mathbb{S}\cup (-\mathbb{S})=\{\pm j,\,\,j\in\mathbb{S}\},\qquad\mathbb{S}_0\triangleq\overline{\mathbb{S}}\cup\{0\},$$
we can split the phase space $L^2_0(\mathbb{T})$ into tangential and normal subspaces
\begin{equation}\label{ph sp splt}
	L^2_0(\mathbb{T})=L_{\overline{\mathbb{S}}}\overset{\perp}{\oplus}L^{2}_{\bot},\qquad L_{\overline{\mathbb{S}}}\triangleq  \Big\{v=\sum_{j\in\overline{\mathbb{S}}}v_j\mathbf{e}_j,\,\, v_{-j}=\overline{v_j}\Big\},\qquad L^{2}_{\bot}\triangleq\Big\{z=\sum_{j\in\mathbb{Z}\setminus\mathbb{S}_0}z_j\mathbf{e}_j,\,\,z_{-j}=\overline{z_j}\Big\}\,.
\end{equation}
Therefore, we can decompose any $r\in L_0^2(\mathbb{T})$ as follows
$$r=v+z,\qquad v=\Pi_{\mathbb{S}_0}r\triangleq\sum_{j\in\overline{\mathbb{S}}}r_j\mathbf{e}_j\in L_{\overline{\mathbb{S}}},\qquad z=\Pi_{\mathbb{S}_0}^{\perp}r\triangleq\sum_{j\in\mathbb{Z}\setminus\mathbb{S}_0}r_j\mathbf{e}_j\in L_{\perp}^2.$$
We consider small amplitudes 
$$(\mathtt{a}_{j})_{j\in\overline{\mathbb{S}}}\in(\mathbb{R}_{+}^{*})^{d},\qquad\mathtt{a}_{-j}=\mathtt{a}_{j}$$
and introduce the action-angle variables 
$$(I,\vartheta)=\Big(\big(I_j\big)_{j\in\overline{\mathbb{S}}},\big(\vartheta_j\big)_{j\in\overline{\mathbb{S}}}\Big),\qquad I_{-j}=I_j\in\mathbb{R},\qquad\vartheta_{-j}=-\vartheta_j\in \mathbb{T}$$
such that on the tangential set $L_{\overline{\mathbb{S}}}$ we have
\begin{equation}\label{action-angle var}
	\forall r\in L_{\overline{\mathbb{S}}},\quad r=\sum_{j\in\overline{\mathbb{S}}}\sqrt{\mathtt{a}_{j}^{2}+|j|I_j}\,e^{\ii \vartheta_j} \mathbf{e}_j\triangleq v(\vartheta,I).
\end{equation}
This defines an application 
$$\mathtt{A}:\begin{array}[t]{rcl}
	\mathbb{T}^d\times\mathbb{R}^d\times L_\perp^2(\mathbb{T}) & \rightarrow & L_0^2(\mathbb{T})\\
	(\vartheta,I,z) & \mapsto & r=v(\vartheta,I)+z
\end{array}$$
which is symplectic with respect to the symplectic structure $\mathscr{W}$ defined in \eqref{scrW}. We refer for instance to \cite{HR21-1} for a proof of this result.
In these coordinates the new Hamiltonian system is generated by the Hamiltonian
\begin{equation}\label{scrHeps}
	\mathscr{H}_{\varepsilon}\triangleq \mathcal{H}_{\varepsilon}\circ\mathtt{A}=\mathscr{N}+\varepsilon\mathcal{P}_{\varepsilon},\qquad\mathscr{N}\triangleq-{\omega}_{\textnormal{Eq}}(\alpha)\cdot I+\frac{1}{2}\big\langle  \mathrm{L}(\alpha)\,z,z \big\rangle_{L^2(\mathbb{T})},	\qquad\mathcal{P}_{\varepsilon}\triangleq    \mathscr{P}_\varepsilon\circ\mathtt{A}.  
\end{equation}
We look for an embedded invariant torus
$$i:\begin{array}[t]{rcl}
		\mathbb{T}^d & \rightarrow & \mathbb{T}^d\times\mathbb{R}^d\times L^{2}_{\bot}\\
		\varphi & \mapsto & i(\varphi)\triangleq(\vartheta(\varphi),I(\varphi),z(\varphi))
	\end{array}$$
of the Hamiltonian vector field 
\begin{equation}\label{X scrHeps}
	X_{\mathscr{H}_{\varepsilon}}\triangleq  
	(\partial_{I}\mathscr{H}_{\varepsilon},-\partial_{\vartheta}\mathscr{H}_{\varepsilon},\Pi_{\mathbb{S}_0}^\bot\partial_{\theta}\nabla_{z} \mathscr{H}_{\varepsilon}) 
\end{equation}
filled by quasi-periodic solutions with Diophantine frequency 
vector $\omega$. 
Note that for value $\varepsilon=0$, the Hamiltonian system 
$$\omega\cdot\partial_\varphi i (\varphi) = X_{\mathscr{H}_0} (i(\varphi))$$ possesses, for any value of the parameter $\alpha\in (\alpha_0,\alpha_1)$, the flat invariant torus $(\varphi,0,0).$ Similarly to \cite{BB15,HHM21,HR21-1,HR21}, we shall introduce a free parameter $\kappa \in\mathbb{R}^d$ to deal with zero $\varphi$-average conditions in the construction of an almost approximate right inverse for the linearized operator and therefore consider the following family of modified Hamiltonians,
\begin{equation}\label{H-ttc}
	\begin{aligned}
		\mathscr{H}_\varepsilon^\kappa  \triangleq\mathscr{N}_\kappa  +\varepsilon\mathcal{P}_{\varepsilon},\qquad \mathscr{N}_\kappa \triangleq\kappa \cdot I+\frac{1}{2}\big\langle \mathrm{L}(\alpha)\, z, z\big\rangle_{L^2(\mathbb{T})}.
	\end{aligned}
\end{equation}
We mention that the original problem is recovered by taking $\kappa =-{\omega}_{\textnormal{Eq}}(\alpha).$ Now we are interested in finding non-trivial zeros of the nonlinear functional
\begin{equation}\label{def scrF}
	\begin{array}{l}
		\mathscr{F}(i,\kappa ,(\alpha,\omega),\varepsilon)\triangleq \omega\cdot\partial_{\varphi}i(\varphi)-X_{\mathscr{H}_{\varepsilon}^{\kappa }}(i(\varphi))=\left(\begin{array}{c}
			\omega\cdot\partial_{\varphi}\vartheta(\varphi)-\kappa -\varepsilon\partial_{I}\mathcal{P}_{\varepsilon}(i(\varphi))\\
			\omega\cdot\partial_{\varphi}I(\varphi)+\varepsilon\partial_{\vartheta}\mathcal{P}_{\varepsilon}(i(\varphi))\\
			\omega\cdot\partial_{\varphi}z(\varphi)-\partial_{\theta}\big[\mathrm{L}(\alpha)z(\varphi)+\varepsilon\nabla_{z}\mathcal{P}_{\varepsilon}\big(i(\varphi)\big)\big]
		\end{array}\right).
	\end{array}
\end{equation}
We point out that we can easily check that 
the Hamiltonian $\mathscr{H}_{\varepsilon}^{\kappa }$ is reversible in the sense of the Definition \ref{Def-Rev}, that is, 
$$\mathscr{H}_{\varepsilon}^{\kappa }\circ\mathbf{S}=\mathscr{H}_{\varepsilon}^{\kappa },\qquad\mathbf{S}(\vartheta,I,z)\triangleq(-\vartheta,I,\mathscr{S}z),$$
with $\mathscr{S}$ as in \eqref{invscrS}. Thus, we look for reversible tori solutions of $\mathscr{F}(i,\kappa ,(\alpha,\omega),\varepsilon)=0,$ that is satisfying
$$\mathbf{S}i(\varphi)=i(-\varphi),\qquad\textnormal{i.e.}\qquad\begin{array}{ccc}
		\vartheta(-\varphi)=-\vartheta(\varphi), & I(-\varphi)=I(\varphi), & z(-\varphi)=(\mathscr{S}z)(\varphi).
	\end{array}$$
Now we define the periodic component $\mathfrak{I}$ of the torus $i$ together with its Sobolev norm by
$$\mathfrak{I}(\varphi)\triangleq i(\varphi)-(\varphi,0,0)=(\Theta(\varphi),I(\varphi),z(\varphi)),\qquad \Theta(\varphi)\triangleq\vartheta(\varphi)-\varphi,\qquad\|\mathfrak{I}\|_{q,s}^{\gamma,\mathcal{O}}\triangleq \|\Theta\|_{q,s}^{\gamma,\mathcal{O}}+\| I\|_{q,s}^{\gamma,\mathcal{O}}+\| z\|_{q,s}^{\gamma,\mathcal{O}}.$$
The norm $\|\cdot\|_{q,s}^{\gamma,\mathcal{O}}$ is defined in Section \ref{sec topo}. We shall fix $q$ as follows
\begin{equation}\label{def q}
	q\triangleq q_0+1,
\end{equation}
with the $q_0$ appearing in Lemma \ref{lem trsvrslt Ea}. This particular choice is relevant in the final subsection when checking the perturbed Rüssmann conditions. We also define the set of parameters $\mathcal{O}$ as follows 
\begin{equation}\label{def calO}
	\mathcal{O}\triangleq(\alpha_{0},\alpha_{1})\times B(0,R_0)\subset\mathbb{R}^{d+1},
\end{equation}
where the open ball $B(0,R_0)\subset\mathbb{R}^d$ with $R_0>0$ is chosen such that
$$\omega_{\textnormal{Eq}}\big((\alpha_{0},\alpha_{1})\big)\subset B\big(0,\tfrac{R_{0}}{2}\big).$$
This is well-defined by continuity of the application $\omega_{\textnormal{Eq}}$ in \eqref{def freq vec eqa}. This particular structure is chosen so that the perturbed frequency vector in Section \ref{sec non triv QP sol} is included in the corresponding component of $\mathcal{O}$ in order to check trivial inclusions. We shall now prove tame estimates for $X_{\mathscr{P}}$.
\begin{lem}\label{lem e-XP}
	Let $(\gamma,s_{0},s,q)$ satisfying \eqref{p-cond0}, \eqref{Sob index} and \eqref{def q}. There exists $\varepsilon_{0}\in(0,1]$ such that if
		$$\| r\|_{q,s_{0}+2}^{\gamma,\mathcal{O}}\leqslant\varepsilon_{0},$$
		then we have the following estimates for the vector field $X_{\mathscr{P}}$ in \eqref{ql per}
		\begin{enumerate}[label=(\roman*)]
			\item $\| X_{\mathscr{P}}(r)\|_{q,s}^{\gamma,\mathcal{O}}\lesssim \| r\|_{q,s+2}^{\gamma,\mathcal{O}}\| r\|_{q,s_0+1}^{\gamma,\mathcal{O}}.$
			\item $\| d_{r}X_{\mathscr{P}}(r)[\rho]\|_{q,s}^{\gamma,\mathcal{O}}\lesssim\|\rho\|_{q,s+2}^{\gamma,\mathcal{O}}\|r\|_{q,s_0+1}^{\gamma,\mathcal{O}}+\| r\|_{q,s+2}^{\gamma,\mathcal{O}}\|\rho\|_{q,s_{0}+1}^{\gamma,\mathcal{O}}.$
			\item 
			$\| d_r^{2}X_{\mathscr{P}}(r)[\rho_{1},\rho_{2}]\|_{q,s}^{\gamma,\mathcal{O}}\lesssim\|\rho_{1}\|_{q,s_{0}+1}^{\gamma,\mathcal{O}}\|\rho_{2}\|_{q,s+2}^{\gamma,\mathcal{O}}+\|\rho_{1}\|_{q,s+2}^{\gamma,\mathcal{O}}\|\rho_{2}\|_{q,s_{0}+1}^{\gamma,\mathcal{O}}+\| r\|_{q,s+2}^{\gamma,\mathcal{O}}\|\rho_{1}\|_{q,s_{0}+1}^{\gamma,\mathcal{O}}\|\rho_{2}\|_{q,s_{0}+1}^{\gamma,\mathcal{O}}$.
	\end{enumerate}
\end{lem}
\begin{proof}
	It suffices to prove the estimate (iii). Indeed, the estimates (i) and (ii) are consequences of (iii) by a direct application of Taylor formula since $X_{\mathscr{P}}(0)=0$ and $d_rX_{\mathscr{P}}(0)=0.$ Recall from \eqref{dFE} and \eqref{dFSW} that
	\begin{align*}
		d_rX_{\mathscr{H}}(r)[\rho]&=d_rF^{\textnormal{\tiny{E}}}(r)[\rho]+d_rF^{\textnormal{\tiny{SW}}}(r)[\rho]\\
		&=\partial_{\theta}\big(V_r^{\textnormal{\tiny{E}}}\rho\big)+\partial_{\theta}\mathbf{L}_r^{\textnormal{\tiny{E}}}(\rho)-\partial_{\theta}\big(V_r^{\textnormal{\tiny{SW}}}\rho\big)-\partial_{\theta}\mathbf{L}_r^{\textnormal{\tiny{SW}}}(\rho).
	\end{align*}
	Differentiating the last expression with respect to $r$ yields
	$$d_r^2X_{\mathscr{P}}(r)[\rho_1,\rho_2]=\partial_{\theta}\big((d_rV_r^{\textnormal{\tiny{E}}}[\rho_2])\rho_1\big)-\partial_{\theta}(d_r\mathbf{L}_r^{\textnormal{\tiny{E}}}[\rho_2]\rho_1)+\partial_{\theta}\big((d_rV_r^{\textnormal{\tiny{SW}}}[\rho_2])\rho_1\big)-\partial_{\theta}(d_r\mathbf{L}_r^{\textnormal{\tiny{SW}}}[\rho_2]\rho_1).$$
	But it has been proved in \cite[Lem. 5.2]{HR21-1} and \cite[Lem. 5.2]{HR21} that each term in the right hand-side of the previous equality satisfy an estimate as in (iii) in the statement of this lemma which concludes the proof of this latter.
\end{proof}

Consequently proceeding as for \cite[Lem. 5.3]{HR21}, the previous lemma implies the following one stating tame estimates for the perturbed Hamiltonian vector field in the action-angle-normal variables.
\begin{lem}\label{lem e-XPeps}
	Let $(\gamma,s_{0},s,q)$ satisfy \eqref{p-cond0}, \eqref{Sob index} and \eqref{def q}. There exists $\varepsilon_0\in(0,1)$ such that if 
	$$\varepsilon\leqslant\varepsilon_0,\qquad\|\mathfrak{I}\|_{q,s_{0}+2}^{\gamma,\mathcal{O}}\leqslant 1,$$ 
	then the perturbed Hamiltonian vector field in the new variables  
	$$X_{\mathcal{P}_{\varepsilon}}=(\partial_{I}\mathcal{P}_{\varepsilon},-\partial_{\vartheta}\mathcal{P}_{\varepsilon},\Pi_{\mathbb{S}}^{\perp}\partial_{\theta}\nabla_{z}\mathcal{P}_{\varepsilon})$$
	defined through \eqref{scrHeps} and \eqref{X scrHeps} satisfies the following tame estimates
	\begin{enumerate}[label=(\roman*)]
		\item $\| X_{\mathcal{P}_{\varepsilon}}(i)\|_{q,s}^{\gamma,\mathcal{O}}\lesssim 1+\|\mathfrak{I}\|_{q,s+2}^{\gamma,\mathcal{O}}.$
		\item $\big\| d_{i}X_{\mathcal{P}_{\varepsilon}}(i)[\,\widehat{i}\,]\big\|_{q,s}^{\gamma,\mathcal{O}}\lesssim \|\,\widehat{i}\,\|_{q,s+2}^{\gamma,\mathcal{O}}+\|\mathfrak{I}\|_{q,s+2}^{\gamma,\mathcal{O}}\|\,\widehat{i}\,\|_{q,s_{0}+1}^{\gamma,\mathcal{O}}.$
		\item $\big\| d_{i}^{2}X_{\mathcal{P}_{\varepsilon}}(i)[\,\widehat{i},\widehat{i}\,]\big\|_{q,s}^{\gamma,\mathcal{O}}\lesssim \|\,\widehat{i}\,\|_{q,s+2}^{\gamma,\mathcal{O}}\|\,\widehat{i}\,\|_{q,s_{0}+1}^{\gamma,\mathcal{O}}+\|\mathfrak{I}\|_{q,s+2}^{\gamma,\mathcal{O}}\left(\|\,\widehat{i}\,\|_{q,s_{0}+1}^{\gamma,\mathcal{O}}\right)^{2}.$
	\end{enumerate}
\end{lem}
\subsection{Almost approximate right inverse}\label{sec cons aai}
The aim of this section is to construct for any vector-valued function $\kappa _0:\mathcal{O}\to \mathbb{R}^d$ and any reversible torus $i_0=(\vartheta_0,I_0,z_0)$ close to the flat one an almost approximate right inverse for the linearized operator
\begin{equation}\label{Lin F}
	d_{(i,\kappa )}\mathscr{F}(i_0,\kappa _0)[\widehat{\imath}\,,\widehat{\kappa }]=\omega \cdot\partial_{\varphi}\widehat\imath-d_{i}X_{\mathscr{H}_{\varepsilon}^{\kappa _0}}(i_0(\varphi))[\widehat{\imath}]-(\widehat{\kappa },0,0)
\end{equation}
associated with the functional $\mathscr{F}$ defined in \eqref{def scrF}. For this purpose, we use the Berti-Bolle theory developed in \cite{BB15} and \cite[Sec. 6]{HHM21}. Namely, we can find a linear diffeomorphism of the toroidal phase space $\mathbb{T}^d\times \mathbb{R}^d\times L_{\perp}^2$ such that the conjugation of \eqref{Lin F} by this application is a triangular system in the action-angles-normal variables up to error terms either vanishing at an exact solution or small fast decaying. The key point to solve the triangular system is that it is sufficient to almost invert the linearized operator in the normal directions. According to the computations done in \cite{BB15,HHM21}, this latter admits the following form
\begin{equation}\label{hLom}
	\mathscr{L}_{\perp}\triangleq \mathscr{L}_{\perp}(i_0)\triangleq \Pi_{\mathbb{S}_0}^{\perp}\Big(\omega\cdot\partial_{\varphi}-\partial_{\theta}\,\partial_z\nabla_z \mathscr{H}_\varepsilon^{\kappa _0}(i_{0}(\varphi))-\varepsilon\partial_{\theta}\mathcal{R}(\varphi)\Big)\Pi_{\mathbb{S}_0}^{\perp},
\end{equation}
where $\mathscr{H}_{\varepsilon}^{\kappa _0}$ is as in \eqref{H-ttc} and $\mathcal{R}(\varphi)$ is a remainder operator coming from a coupling with the tangential part and given by
$$\mathcal{R}(\varphi)\triangleq L_{2}^{\top}(\varphi)\partial_{I}\nabla_I\mathcal{P}_{\varepsilon}(i_{0}(\varphi))L_{2}(\varphi)+L_{2}^{\top}(\varphi)\partial_{z}\nabla_{I}\mathcal{P}_{\varepsilon}(i_{0}(\varphi))+\partial_{I}\nabla_{z}\mathcal{P}_{\varepsilon}(i_{0}(\varphi))L_{2}(\varphi),$$
with $\mathcal{P}_{\varepsilon}$ as in \eqref{scrHeps} and  
$$L_2:\mathbb{R}^d\rightarrow L^2_\perp,\qquad L_2(\phi)\triangleq-[(\partial_\vartheta \widetilde{z}_0)(\vartheta_0(\phi))]^\top \partial_\theta^{-1},\qquad\widetilde{z}_0(\vartheta)\triangleq z_0(\vartheta_0^{-1}(\vartheta)).$$
We used the following definition by duality for the transposed operator $L_2^{\top}:L^2_{\perp}\to\mathbb{R}^d$ 
$$\forall\, u\in L^2_{\perp}\, ,\quad\forall\,v\in\mathbb{R}^d,\quad\big\langle L_2^{\top}(\varphi)u,v\big\rangle_{\mathbb{R}^d} =\big\langle u,L_2(\varphi) v\big\rangle_{L^2(\mathbb{T}^d)}.$$
Furthermore, we can have a more explicit decomposition of the operator $\mathscr{L}_{\perp}$. The result is described in the following proposition whose proof is similar to \cite[Prop. 6.1]{HR21} or \cite[Prop. 6.1]{HR21-1}.
\begin{prop}\label{prop hat L omega} Let $(d,\gamma,s_{0},q)$ satisfy \eqref{p-d}, \eqref{p-cond0}, \eqref{Sob index} and \eqref{def q}.
	Then, the operator $\mathscr{L}_{\perp}$ in \eqref{hLom} writes
	\begin{equation}\label{hLom-2}
		\mathscr{L}_{\perp}=\Pi_{\mathbb{S}_0}^{\perp}\Big(\mathcal{L}_{\varepsilon r}-\varepsilon\partial_{\theta}\mathcal{R}\Big)\Pi_{\mathbb{S}_0}^{\perp},
	\end{equation}
	where the operator $\mathcal{L}_{\varepsilon r}$ is defined as follows with $V_{\varepsilon r}$ and $\mathbf{L}_{\varepsilon r}$ obtained from \eqref{Vr bfLr},
		\begin{equation}\label{L eps r Ea}
			\mathcal{L}_{\varepsilon r}\triangleq \omega\cdot\partial_{\varphi}+\partial_{\theta}\big(V_{\varepsilon r}\cdot\big)+\partial_{\theta}\mathbf{L}_{\varepsilon r}.
		\end{equation}
		The function $r$ is linked to the reversible torus $i_0$ in the following way
		\begin{equation}\label{sym r}
			r(\varphi,\cdot)\triangleq\mathtt{A}\big(i_{0}(\varphi)\big),\qquad r(-\varphi,-\theta)=r(\varphi,\theta)
		\end{equation}
		and satisfies the following estimates
		\begin{equation}\label{e-r-I0}
			\| r\|_{q,s}^{\gamma,\mathcal{O}}\lesssim 1+\|\mathfrak{I}_{0}\|_{q,s}^{\gamma,\mathcal{O}},\qquad
			\|\Delta_{12}r\|_{q,s}^{\gamma,\mathcal{O}}\lesssim\|\Delta_{12}i\|_{q,s}^{\gamma,\mathcal{O}}+\|\Delta_{12}i\|_{q,s_0}^{\gamma,\mathcal{O}}\max_{j\in\{1,2\}}\|\mathfrak{I}_j\|_{q,s}^{\gamma,\mathcal{O}}.
		\end{equation}
		Finally, the operator $\mathcal{R}$ is an integral operator with kernel $\mathcal{J}$ satisfying the symmetry property
		\begin{equation}\label{symJ}
			\mathcal{J}(-\varphi,-\theta,-\eta)=\mathcal{J}(\varphi,\theta,\eta)
		\end{equation}
		and the following estimates for all $\ell\in\mathbb{N}$,
		\begin{equation}\label{e1-J}
			\sup_{\eta\in\mathbb{T}}\|(\partial_{\theta}^{\ell}\mathcal{J})(\ast,\cdot,\centerdot,\eta+\centerdot)\|_{q,s}^{\gamma,\mathcal{O}}\lesssim 1+\|\mathfrak{I}_{0}\|_{q,s+3+\ell}^{\gamma,\mathcal{O}}
		\end{equation}
		and
		\begin{equation}\label{e2-J} 
			\sup_{\eta\in\mathbb{T}}\|\Delta_{12}(\partial_{\theta}^{\ell}\mathcal{J})(\ast,\cdot,\centerdot,\eta+\centerdot)\|_{q,s}^{\gamma,\mathcal{O}}\lesssim\|\Delta_{12}i\|_{q,s+3+\ell}^{\gamma,\mathcal{O}}+\|\Delta_{12}i\|_{q,s_0+3}^{\gamma,\mathcal{O}}\max_{j\in\{1,2\}}\|\mathfrak{I}_j\|_{q,s+3+\ell}^{\gamma,\mathcal{O}}.
		\end{equation}
	where $*,\cdot,\centerdot,$ denote the variables $\alpha,\varphi,\theta$ and $\displaystyle\mathfrak{I}_{\ell}(\varphi)\triangleq i_{\ell}(\varphi)-(\varphi,0,0).$
\end{prop}
Now we shall start the reduction procedure of the operator $\mathscr{L}_{\perp}.$ The first step is the reduction of the transport part of $\mathcal{L}_{\varepsilon r}$ in \eqref{L eps r Ea}, which is done by conjugation with quasi-periodic symplectic change of variables close to the identity similarly to \cite{BBM16,BKM21,FGMP19,HR21}. Actually, we apply here \cite[Prop. 6.2]{HR21}. The result is the following.
\begin{prop}\label{prop reduc trans}
	Let $(d,s_{0},S,\gamma,q,q_0)$ satisfy \eqref{p-d}, \eqref{Sob index}, \eqref{p-cond0} and \eqref{def q}. We consider the following parameters
\begin{equation}\label{ptrans}
		\begin{array}{ll} 
			\upsilon\triangleq\frac{1}{q_0+3}, & \tau_1\triangleq dq_0+1,\\
			s_{l}\triangleq s_0+\tau_1q+\tau_1+2,& \overline{\mu}_2\triangleq 4\tau_1q+6\tau_1+3,\\
			\overline{s}_{l}\triangleq s_l+\tau_2q+\tau_2, & \overline{s}_{h}\triangleq \frac{3}{2}\overline{\mu}_{2}+s_{l}+1,\\
			\sigma_{1}\triangleq s_0+\tau_{1}q+2\tau_{1}+4, & \sigma_2\triangleq s_0+\sigma_{1}+3.
		\end{array}
	\end{equation}
	For every choice of additional parameters $(\mu_2,\mathtt{p},s_h)$ with the constraints
	\begin{equation}\label{ptrans-2}
		\mu_{2}\geqslant \overline{\mu}_{2}, \qquad\mathtt{p}\geqslant 0,\qquad s_{h}\geqslant\max\left(\frac{3}{2}\mu_{2}+s_{l}+1,\overline{s}_{h}+\mathtt{p}\right),
	\end{equation}
	there exists $\varepsilon_{0}>0$ such that if the following smallness condition holds
	\begin{equation}\label{sml-trans}
		\varepsilon\gamma^{-1}N_{0}^{\mu_{2}}\leqslant\varepsilon_{0},\qquad \|\mathfrak{I}_{0}\|_{q,s_{h}+\sigma_{2}}^{\gamma,\mathcal{O}}\leqslant 1,
	\end{equation}
	then there exist $\mathtt{m}_{i_0}^{\infty}\in W^{q,\infty,\gamma }(\mathcal{O},\mathbb{C})$ $\varepsilon$-close to the equilibrium one $V_0$ in \eqref{LaKa}, namely
		\begin{equation}\label{e-r1}
			\| \mathtt{m}_{i_0}^{\infty}-V_0\|_{q}^{\gamma ,\mathcal{O}}\lesssim \varepsilon
		\end{equation}
	and an invertible quasi-periodic symplectic change of variables $\mathscr{B}$ in the form
	$$\mathscr{B}\triangleq(1+\partial_{\theta}\beta)\mathcal{B},\qquad\mathcal{B}\rho(\mu,\varphi,\theta)\triangleq\rho\big(\mu,\varphi,\theta+\beta(\mu,\varphi,\theta)\big),\qquad\beta\in W^{q,\infty,\gamma }(\mathcal{O},H^{S})$$
with inverse $\mathscr{B}^{-1}$ in the form
	$$\mathscr{B}^{-1}=(1+\partial_y\widehat{\beta})\mathcal{B}^{-1},\qquad \mathcal{B}^{-1} \rho(\mu,\varphi,y)=\rho\big(\mu,\varphi,y+\widehat{\beta}(\mu,\varphi,y)\big)$$
with
$$y=\theta+\beta(\mu,\varphi,\theta)\quad\Leftrightarrow\quad\theta=y+\widehat{\beta}(\mu,\varphi,y)$$
enjoying the symmetry properties
\begin{equation}\label{sym-beta}
	\beta(\alpha,\omega,-\varphi,-\theta)=-\beta(\alpha,\omega,\varphi,\theta),\qquad\widehat{\beta}(\alpha,\omega,-\varphi,-\theta)=-\widehat{\beta}(\alpha,\omega,\varphi,\theta)
\end{equation}
and the following estimates for all $s\in[s_{0},S]$ 
			\begin{equation}\label{e-scrB-scrB1}
				\|\mathscr{B}^{\pm 1}\rho\|_{q,s}^{\gamma ,\mathcal{O}}+\|\mathcal{B}^{\pm 1}\rho\|_{q,s}^{\gamma ,\mathcal{O}}\lesssim\|\rho\|_{q,s}^{\gamma ,\mathcal{O}}+\varepsilon\gamma ^{-1}\| \mathfrak{I}_{0}\|_{q,s+\sigma_{1}}^{\gamma ,\mathcal{O}}\|\rho\|_{q,s_{0}}^{\gamma ,\mathcal{O}}
			\end{equation}
			and 
			\begin{equation}\label{e-btr}
				\|\widehat{\beta}\|_{q,s}^{\gamma,\mathcal{O}}\lesssim\|\beta\|_{q,s}^{\gamma ,\mathcal{O}}\lesssim \varepsilon\gamma ^{-1}\left(1+\| \mathfrak{I}_{0}\|_{q,s+\sigma_{1}}^{\gamma ,\mathcal{O}}\right)
			\end{equation}
such that for any $n\in\mathbb{N}$, then in restriction to the Cantor-like set
			\begin{equation}\label{Cantrans}
				\mathscr{O}_{\textnormal{\tiny{transport}}}^{n,\gamma,\tau_1}(i_{0})\triangleq\bigcap_{(l,j)\in\mathbb{Z}^{d}\times\mathbb{Z}\setminus\{(0,0)\}\atop|l|\leqslant N_{n}}\left\lbrace(\alpha,\omega)\in \mathcal{O}\quad\textnormal{s.t.}\quad\big|\omega\cdot l+j\mathtt{m}_{i_0}^{\infty}(\alpha,\omega)\big|>\tfrac{4\gamma^{\upsilon}\langle j\rangle}{\langle l\rangle^{\tau_{1}}}\right\rbrace,
			\end{equation}
			the following reduction holds
			\begin{equation}\label{reduc trans eq}
				\mathfrak{L}_{\varepsilon r}\triangleq \mathscr{B}^{-1}\mathcal{L}_{\varepsilon r}\mathscr{B}=\omega\cdot\partial_{\varphi}+\mathtt{m}_{i_0}^{\infty}\partial_{\theta}+\partial_{\theta}\mathcal{K}_{\alpha}\ast\cdot+\partial_{\theta}\mathbf{R}_{\varepsilon r}+\mathbf{E}_{n}^{0},
			\end{equation}
			with $\mathcal{K}_{\alpha}$ as in \eqref{LaKa} and $\mathbf{E}_{n}^{0}=\mathbf{E}_{n}^{0}(\alpha,\omega,i_{0})$ a linear operator satisfying
			\begin{equation}\label{e-En0}
				\|\mathbf{E}_{n}^{0}\rho\|_{q,s_{0}}^{\gamma,\mathcal{O}}\lesssim\varepsilon N_{0}^{\mu_{2}}N_{n+1}^{-\mu_{2}}\|\rho\|_{q,s_{0}+2}^{\gamma,\mathcal{O}}.
			\end{equation}
			The operator $\mathbf{R}_{\varepsilon r}$ is a real and reversibility preserving integral operator satisfying
			\begin{align}\label{e-frkR}
				\forall s\in[s_0,S],\quad \max_{k\in\{0,1,2\}}\|\partial_{\theta}^{k}\mathbf{R}_{\varepsilon r}\|_{\textnormal{\tiny{O-d}},q,s}^{\gamma,\mathcal{O}}&\lesssim\varepsilon\gamma^{-1}\left(1+\|\mathfrak{I}_{0}\|_{q,s+\sigma_{2}}^{\gamma,\mathcal{O}}\right).
			\end{align}
		In addition, for any tori $i_{1}$ and $i_{2}$ both satisfying \eqref{sml-trans}, we have 
		\begin{equation}\label{e-12-Vbt}
			\|\Delta_{12}\mathtt{m}_{i}^{\infty}\|_{q}^{\gamma,\mathcal{O}}\lesssim\varepsilon\| \Delta_{12}i\|_{q,\overline{s}_{h}+2}^{\gamma,\mathcal{O}},\qquad\|\Delta_{12}\beta\|_{q,\overline{s}_{h}+\mathtt{p}}^{\gamma,\mathcal{O}}+\|\Delta_{12}\widehat\beta\|_{q,\overline{s}_{h}+\mathtt{p}}^{\gamma,\mathcal{O}}\lesssim\varepsilon\gamma^{-1}\|\Delta_{12}i\|_{q,\overline{s}_{h}+\mathtt{p}+\sigma_{1}}^{\gamma,\mathcal{O}}
		\end{equation}
		and
		\begin{align}\label{e-dfrkR}
			\max_{k\in\{0,1\}}\|\Delta_{12}(\partial_{\theta}^k\mathbf{R}_{\varepsilon r})\|_{\textnormal{\tiny{O-d}},q,\overline{s}_{h}+\mathtt{p}}^{\gamma,\mathcal{O}}&\lesssim\varepsilon\gamma^{-1}\|\Delta_{12}i\|_{q,\overline{s}_{h}+\mathtt{p}+\sigma_{2}}^{\gamma,\mathcal{O}}.
		\end{align}
\end{prop}
\begin{proof}
 	Using the computations in \cite[Prop. 6.2]{HR21} and \cite[Prop. 6.2]{HR21-1} for the $V_{\varepsilon r}^{\textnormal{\tiny{E}}}$ and $V_{\varepsilon r}^{\textnormal{\tiny{SW}}}$ in \eqref{VrE}-\eqref{VrSW}, we get the following estimate for $V_{\varepsilon r}$ in \eqref{Vr bfLr},
	$$\|V_{\varepsilon r}-V_{0}\|_{q,s}^{\gamma,\mathcal{O}}\leqslant\|V_{\varepsilon r}^{\textnormal{\tiny{E}}}+\tfrac{1}{2}\|_{q,s}^{\gamma,\mathcal{O}}+\|V_{\varepsilon r}^{\textnormal{\tiny{SW}}}-I_{1}K_1\|_{q,s}^{\gamma,\mathcal{O}}\lesssim\varepsilon\left(1+\|\mathfrak{I}_{0}\|_{q,s+1}^{\gamma,\mathcal{O}}\right).$$
	This allows us to start a KAM reduction procedure as detailed in \cite[prop 6.2]{HR21}, which provides $\mathscr{O}_{\textnormal{\tiny{transport}}}^{n,\gamma,\tau_1}(i_{0}),$ $\mathtt{m}_{i_0}^{\infty}$, $\beta$ and $\mathbf{E}_{n}^{0}$ as in the statement, with the corresponding estimates \eqref{e-scrB-scrB1}, \eqref{e-btr}, \eqref{e-En0}, \eqref{e-12-Vbt} and symmetry \eqref{sym-beta} such that in restriction to $\mathscr{O}_{\textnormal{\tiny{transport}}}^{n,\gamma,\tau_1}(i_{0})$ we have
	$$\mathscr{B}^{-1}\Big(\omega\cdot\partial_{\varphi}+\partial_{\theta}\big(V_{\varepsilon r}\cdot\big)\Big)\mathscr{B}=\omega\cdot\partial_{\varphi}+\mathtt{m}_{i_0}^{\infty}\partial_{\theta}+\mathbf{E}_{n}^{0}.$$
Now, we shall look at the conjugation effect on the nonlocal term. Refering to \cite[Lem. 5.1]{HR21} and \cite[Lem. 5.1]{HR21-1}, we have the following decompositions for the operators in \eqref{mathbfLrE}-\eqref{mathbfLrSW},
\begin{align}
	\partial_{\theta}\mathbf{L}_{\varepsilon r}^{\textnormal{\tiny{E}}}&=\partial_{\theta}\mathcal{K}\ast\cdot+\partial_{\theta}\mathbf{L}_{\varepsilon r,1}^{\textnormal{\tiny{E}}},\qquad\mathbf{L}_{\varepsilon r,1}^{\textnormal{\tiny{E}}}(\rho)(\varphi,\theta)\triangleq\int_{\mathbb{T}}\rho(\varphi,\eta)\mathbb{K}_{\varepsilon r}^{\textnormal{\tiny{E}}}(\varphi,\theta,\eta)d\eta,\label{decomp LerE}\\
	\partial_{\theta}\mathbf{L}_{\varepsilon r}^{\textnormal{\tiny{SW}}}&=\partial_{\theta}\mathcal{Q}_{\alpha}\ast\cdot+\partial_{\theta}\mathbf{L}_{\varepsilon r,1}^{\textnormal{\tiny{SW}}},\qquad\mathbf{L}_{\varepsilon r,1}^{\textnormal{\tiny{SW}}}(\rho)(\alpha,\varphi,\theta)\triangleq\int_{\mathbb{T}}\rho(\varphi,\eta)\mathbb{K}_{\varepsilon r}^{\textnormal{\tiny{SW}}}(\alpha,\varphi,\theta,\eta)d\eta,\label{decomp LerSW}
\end{align}
where the kernels $\mathbb{K}_{\varepsilon r}^{\textnormal{\tiny{E}}}$ and $\mathbb{K}_{\varepsilon r}^{\textnormal{\tiny{SW}}}$ are respectively defined by
\begin{align*}
	\mathbb{K}_{\varepsilon r}^{\textnormal{\tiny{E}}}(\varphi,\theta,\eta)&\triangleq\log\big(v_{\varepsilon r}(\varphi,\theta,\eta)\big),\qquad v_{\varepsilon r}(\varphi,\theta,\eta)\triangleq\left(\left(\tfrac{R(\varphi,\theta)-R(\varphi,\eta)}{2\sin\left(\frac{\theta-\eta}{2}\right)}\right)^2+R(\varphi,\theta)R(\varphi,\eta)\right)^{\frac{1}{2}}\\
	\mathbb{K}_{\varepsilon r}^{\textnormal{\tiny{SW}}}(\varphi,\theta,\eta)&\triangleq\mathscr{K}(\theta-\eta)\mathscr{K}_{\varepsilon r,1}(\lambda,\varphi,\theta,\eta)+\mathscr{K}_{\varepsilon r,2}(\lambda,\varphi,\theta,\eta),\qquad\mathscr{K}(\theta)\triangleq\sin^2\left(\tfrac{\theta}{2}\right)\log\left(\left|\sin\left(\tfrac{\theta}{2}\right)\right|\right),\\
	\mathscr{K}_{\varepsilon r,1}(\lambda,\varphi,\theta,\eta)&\triangleq \sum_{m=1}^{\infty}\tfrac{(2\lambda)^{2m}}{(m!)^2}\sin^{2m-2}\left(\tfrac{\theta-\eta}{2}\right)\big(1-v_{\varepsilon r}(\varphi,\theta,\eta)\big),\qquad\lambda\triangleq\frac{1}{\alpha}\\
	\mathscr{K}_{\varepsilon r,2}(\lambda,\varphi,\theta,\eta)&\triangleq \log(\lambda)\sin^2\left(\frac{\eta-\theta}{2}\right)\mathscr{K}_{\varepsilon r,1}(\lambda,\varphi,\theta,\eta)-\mathbb{K}_{\varepsilon r}^{\textnormal{\tiny{E}}}(\varphi,\theta,\eta)I_0\big(\lambda A_{\varepsilon r}(\varphi,\theta,\eta)\big)\\
	&\quad+f\big(\lambda A_{\varepsilon r}(\varphi,\theta,\eta)\big)-f\left(2\lambda\left|\sin\left(\tfrac{\eta-\theta}{2}\right)\right|\right),\qquad f\textnormal{ analytic.}
\end{align*}
One can easily check that \eqref{sym r} implies
\begin{equation}\label{sym ker E-SW}
	\mathbb{K}_{\varepsilon r}^{\textnormal{\tiny{E}}}(-\varphi,-\theta,-\eta)=\mathbb{K}_{\varepsilon r}^{\textnormal{\tiny{E}}}(\varphi,\theta,\eta),\qquad\mathbb{K}_{\varepsilon r}^{\textnormal{\tiny{SW}}}(\alpha,-\varphi,-\theta,-\eta)=\mathbb{K}_{\varepsilon r}^{\textnormal{\tiny{SW}}}(\alpha,\varphi,\theta,\eta)
\end{equation}
Hence, in restriction to $\mathscr{O}_{\textnormal{\tiny{transport}}}^{n,\gamma,\tau_1}(i_{0})$ we have by \eqref{decomp LerE}-\eqref{decomp LerSW},
\begin{align*}
	\mathscr{B}^{-1}\mathcal{L}_{\varepsilon r}\mathscr{B}&=\omega\cdot\partial_{\varphi}+\mathtt{m}_{i_0}^{\infty}\partial_{\theta}+\partial_{\theta}\mathcal{B}^{-1}\mathbf{L}_{\varepsilon r}^{\textnormal{\tiny{E}}}\mathscr{B}+\partial_{\theta}\mathcal{B}^{-1}\mathbf{L}_{\varepsilon r}^{\textnormal{\tiny{SW}}}\mathscr{B}+\mathbf{E}_{n}^{0}\\
	&=\omega\cdot\partial_{\varphi}+\mathtt{m}_{i_0}^{\infty}\partial_{\theta}+\partial_{\theta}\mathcal{K}_{\alpha}\ast\cdot+\partial_{\theta}\mathbf{R}_{\varepsilon r}+\mathbf{E}_{n}^{0},
\end{align*}
with
$$\mathbf{R}_{\varepsilon r}\triangleq\mathcal{B}^{-1}\mathbf{L}_{\varepsilon r,1}^{\textnormal{\tiny{E}}}\mathscr{B}+\mathcal{B}^{-1}\mathbf{L}_{\varepsilon r,1}^{\textnormal{\tiny{SW}}}\mathscr{B}+\Big(\mathcal{B}^{-1}\big(\mathcal{K}\ast\cdot\big)\mathscr{B}-\mathcal{K}\ast\cdot\Big)+\Big(\mathcal{B}^{-1}\big(\mathcal{Q}_{\alpha}\ast\cdot\big)\mathscr{B}-\mathcal{Q}_{\alpha}\ast\cdot\Big).$$
The estimates \eqref{e-frkR} and \eqref{e-dfrkR} are obtained similarly to \cite[Prop. 6.2]{HR21} and \cite[Prop. 6.2]{HR21-1} by using Lemma \ref{lem sym-Rev}, Lemma \ref{Lem-lawprod} and \eqref{e-btr}-\eqref{e-12-Vbt}. In particular, the reversibility property follows from \eqref{sym ker E-SW}-\eqref{sym-beta}. This concludes the proof of Proposition \ref{prop reduc trans}.
\end{proof}

Then we study the action of the localization in the normal directions. For that, we introduce the operator
$$\mathscr{B}_{\perp}\triangleq\Pi_{\mathbb{S}_0}^{\perp}\mathscr{B}\Pi_{\mathbb{S}_{0}}^{\perp},$$
which satisfies by virtue of \eqref{e-scrB-scrB1} the following estimate
\begin{equation}\label{e-Bop perp}
	\|\mathscr{B}_{\perp}^{\pm 1}\rho\|_{q,s}^{\gamma ,\mathcal{O}}\lesssim\|\rho\|_{q,s}^{\gamma ,\mathcal{O}}+\varepsilon\gamma ^{-1}\| \mathfrak{I}_{0}\|_{q,s+\sigma_{1}}^{\gamma ,\mathcal{O}}\|\rho\|_{q,s_{0}}^{\gamma ,\mathcal{O}}.
\end{equation}
Hence, the result reads as follows.
\begin{prop}\label{prop projnor}
	Let $(d,s_{0},S,\gamma,q,q_0,\tau_{1},s_{h},\overline{s}_{h},\sigma_{2},\mathtt{p})$ satisfy \eqref{p-d}, \eqref{Sob index}, \eqref{p-cond0}, \eqref{def q}, \eqref{ptrans} and \eqref{ptrans-2}. There exist $\varepsilon_0>0$ and $\sigma_{3}=\sigma_{3}(\tau_{1},q,d,s_{0})\geqslant\sigma_{2}$ such that if the following smallness condition holds
	$$\varepsilon\gamma^{-1}N_0^{\mu_2}\leqslant\varepsilon_0,\qquad\|\mathfrak{I}_0\|_{q,s_h+\sigma_3}^{\gamma,\mathcal{O}}\leqslant 1,$$
	then for any $n\in\mathbb{N}^{*},$ in restriction to the Cantor-like set $\mathscr{O}_{\textnormal{\tiny{transport}}}^{n,\gamma,\tau_1}(i_{0})$ defined in \eqref{Cantrans}, the following decomposition holds 
		\begin{align}
			\mathscr{B}_{\perp}^{-1}\mathscr{L}_{\perp}\mathscr{B}_{\perp}&=\big(\omega\cdot\partial_{\varphi}+\mathtt{m}_{i_0}^{\infty}\partial_{\theta}+\partial_{\theta}\mathcal{K}_{\alpha}\ast\cdot\big)\Pi_{\mathbb{S}_0}^{\perp}+\mathscr{R}_{0}+\mathbf{E}_{n}^{1}\nonumber\\
			&\triangleq \omega\cdot\partial_{\varphi}\Pi_{\mathbb{S}_0}^{\perp}+\mathscr{D}_{0}+\mathscr{R}_{0}+\mathbf{E}_{n}^{1}\nonumber\\
			&\triangleq \mathscr{L}_{0}+\mathbf{E}_{n}^{1}.\label{def scr L0}
		\end{align}
		The operator $\mathscr{D}_{0}=\Pi_{\mathbb{S}_0}^\perp \mathscr{D}_{0}\Pi_{\mathbb{S}_0}^\perp$ is a reversible diagonal operator described by 
		\begin{equation}\label{mu0-r1}
			\mathscr{D}_{0}\mathbf{e}_{l,j}=\ii \mathrm{d}_{j}^{0}\,\mathbf{e}_{l,j},\qquad\mathrm{d}_{j}^{0}(\alpha,\omega,i_{0})\triangleq\Omega_{j}^{\textnormal{\tiny{E}}}(\alpha)+jr^{0}(\alpha,\omega,i_{0}),\qquad r^{0}(\alpha,\omega,i_{0})\triangleq \mathtt{m}_{i_0}^{\infty}(\alpha,\omega)-V_0(\alpha)
		\end{equation}
		with
		\begin{equation}\label{e-r1-bis}
			\|r^{0}\|_{q}^{\gamma,\mathcal{O}}\lesssim \varepsilon,\qquad \|\Delta_{12}r^{0}\|_{q}^{\gamma,\mathcal{O}}\lesssim\varepsilon \| \Delta_{12}i\|_{q,\overline{s}_{h}+2}^{\gamma,\mathcal{O}}.
		\end{equation}
		The operator $\mathscr{R}_{0}=\Pi_{\mathbb{S}_0}^\perp \mathscr{R}_{0}\Pi_{\mathbb{S}_0}^\perp$ is a real and reversible Toeplitz in time integral operator satisfying 
		\begin{equation}\label{escr R Ods}
			\forall s\in [s_{0},S],\quad \max_{k\in\{0,1\}}\|\partial_{\theta}^{k}\mathscr{R}_{0}\|_{\textnormal{\tiny{O-d}},q,s}^{\gamma,\mathcal{O}}\lesssim\varepsilon\gamma^{-1}\left(1+\| \mathfrak{I}_{0}\|_{q,s+\sigma_{3}}^{\gamma,\mathcal{O}}\right)
		\end{equation}
		and
		\begin{equation}\label{escr R Ods diff}
			\|\Delta_{12}\mathscr{R}_{0}\|_{\textnormal{\tiny{O-d}},q,\overline{s}_{h}+\mathtt{p}}^{\gamma,\mathcal{O}}\lesssim\varepsilon\gamma^{-1}\| \Delta_{12}i\|_{q,\overline{s}_{h}+\mathtt{p}+\sigma_{3}}^{\gamma,\mathcal{O}}.
		\end{equation}
	The operator $\mathbf{E}_{n}^{1}$ satisfies the following estimate
	\begin{equation}\label{e-En1}
		\|\mathbf{E}_{n}^{1}\rho\|_{q,s_{0}}^{\gamma,\mathcal{O}}\lesssim \varepsilon N_{0}^{\mu_{2}}N_{n+1}^{-\mu_{2}}\|\rho\|_{q,s_{0}+2}^{\gamma,\mathcal{O}}.
	\end{equation}
		In addition, the operator $\mathscr{L}_{0}$ satisfies
		\begin{equation}\label{e-scrL0}
			\forall s\in[s_{0},S],\quad\|\mathscr{L}_{0}\rho\|_{q,s}^{\gamma ,\mathcal{O}}\lesssim\|\rho\|_{q,s+1}^{\gamma ,\mathcal{O}}+\varepsilon\gamma ^{-1}\| \mathfrak{I}_{0}\|_{q,s+\sigma_{3}}^{\gamma ,\mathcal{O}}\|\rho\|_{q,s_{0}}^{\gamma ,\mathcal{O}}.
		\end{equation}
\end{prop}
\begin{proof}
	The identities \eqref{hLom-2} and $\textnormal{Id}=\Pi_{\mathbb{S}_{0}}+\Pi_{\mathbb{S}_{0}}^{\perp}$ imply
	\begin{align*}
		\mathscr{B}_{\perp}^{-1}\mathscr{L}_{\perp}\mathscr{B}_{\perp}&=\mathscr{B}_{\perp}^{-1}\Pi_{\mathbb{S}_{0}}^{\perp}(\mathcal{L}_{\varepsilon r}-\varepsilon\partial_{\theta}\mathcal{R})\mathscr{B}_{\perp}
		\\
		&=\mathscr{B}_{\perp}^{-1}\Pi_{\mathbb{S}_{0}}^{\perp}\mathcal{L}_{\varepsilon r}\mathscr{B}\Pi_{\mathbb{S}_{0}}^{\perp}-\mathscr{B}_{\perp}^{-1}\Pi_{\mathbb{S}_{0}}^{\perp}\mathcal{L}_{\varepsilon r}\Pi_{\mathbb{S}_{0}}\mathscr{B}\Pi_{\mathbb{S}_{0}}^{\perp}-\varepsilon\mathscr{B}_{\perp}^{-1}\Pi_{\mathbb{S}_{0}}^{\perp}\partial_{\theta}\mathcal{R}\mathscr{B}_{\perp}.
	\end{align*}
	Now using \eqref{reduc trans eq} and the fact that
	$$\mathscr{B}_{\perp}^{-1}\Pi_{\mathbb{S}_{0}}^{\perp}=\mathscr{B}_{\perp}^{-1},\qquad [\Pi_{\mathbb{S}_{0}}^{\perp},F]=0=[\Pi_{\mathbb{S}_{0}},F],\quad F\textnormal{ Fourier multiplier},$$
	we obtain in restriction to the Cantor set $\mathscr{O}_{\textnormal{\tiny{transport}}}^{n,\gamma,\tau_1}(i_{0})$ the following decomposition	
\begin{align*}
	\mathscr{B}_{\perp}^{-1}\Pi_{\mathbb{S}_{0}}^{\perp}\mathcal{L}_{\varepsilon r}\mathscr{B}\Pi_{\mathbb{S}_{0}}^{\perp}&=\mathscr{B}_{\perp}^{-1}\Pi_{\mathbb{S}_{0}}^{\perp}\mathscr{B}\mathfrak{L}_{\varepsilon r}\Pi_{\mathbb{S}_{0}}^{\perp}\\
	&=\big(\omega\cdot\partial_{\varphi}+\mathtt{m}_{i_0}^{\infty}\partial_{\theta}+\partial_{\theta}\mathcal{K}_{\alpha}\ast\cdot\big)\Pi_{\mathbb{S}_{0}}^{\perp}+\Pi_{\mathbb{S}_{0}}^{\perp}\partial_{\theta}\mathbf{R}_{\varepsilon r}\Pi_{\mathbb{S}_{0}}^{\perp}+\mathscr{B}_{\perp}^{-1}\mathscr{B}\Pi_{\mathbb{S}_{0}}\partial_{\theta}\mathbf{R}_{\varepsilon r}\Pi_{\mathbb{S}_{0}}^{\perp}\\
	&\quad+\mathscr{B}_{\perp}^{-1}\Pi_{\mathbb{S}_{0}}^{\perp}\mathscr{B}\mathbf{E}_{n}^{0}\Pi_{\mathbb{S}_{0}}^{\perp}.
\end{align*}
Hence, using also \eqref{decomp LerE}, \eqref{decomp LerSW}, one gets that in the Cantor set $\mathscr{O}_{\textnormal{\tiny{transport}}}^{n,\gamma,\tau_1}(i_{0})$, the following identity holds
	\begin{align*}
		 \mathscr{B}_{\perp}^{-1}\widehat{\mathcal{L}}_{\omega}\mathscr{B}_{\perp}&=\big(\omega\cdot\partial_{\varphi}+\mathtt{m}_{i_0}^{\infty}\partial_{\theta}+\partial_{\theta}\mathcal{K}_{\alpha}\ast\cdot\big)\Pi_{\mathbb{S}_{0}}^{\perp}+\Pi_{\mathbb{S}_{0}}^{\perp}\partial_{\theta}\mathbf{R}_{\varepsilon r}\Pi_{\mathbb{S}_{0}}^{\perp}+\mathscr{B}_{\perp}^{-1}\mathscr{B}\Pi_{\mathbb{S}_{0}}\partial_{\theta}\mathbf{R}_{\varepsilon r}\Pi_{\mathbb{S}_{0}}^{\perp}\\
		 &\quad-\mathscr{B}_{\perp}^{-1}\Pi_{\mathbb{S}_{0}}^{\perp}\left(\partial_{\theta}\left(V_{\varepsilon r}\cdot\right)+\partial_{\theta}\mathbf{L}_{\varepsilon r,1}^{\textnormal{\tiny{E}}}+\partial_{\theta}\mathbf{L}_{\varepsilon r,1}^{\textnormal{\tiny{SW}}}\right)\Pi_{\mathbb{S}_{0}}\mathscr{B}\Pi_{\mathbb{S}_{0}}^{\perp}-\varepsilon\mathscr{B}_{\perp}^{-1}\partial_{\theta}\mathcal{R}\mathscr{B}_{\perp}+\mathscr{B}_{\perp}^{-1}\Pi_{\mathbb{S}_{0}}^{\perp}\mathscr{B}\mathbf{E}_{n}^{0}\Pi_{\mathbb{S}_{0}}^{\perp}\\
		&\triangleq \omega\cdot\partial_{\varphi}\Pi_{\mathbb{S}_{0}}^{\perp}+\mathscr{D}_0+\mathscr{R}_0+\mathbf{E}_{n}^{1},
	\end{align*}
	with
	\begin{equation}\label{En1}
		\mathscr{D}_0\triangleq \big(\mathtt{m}_{i_0}^{\infty}\partial_{\theta}+\partial_{\theta}\mathcal{K}_{\alpha}\ast\cdot\big)\Pi_{\mathbb{S}_{0}}^{\perp},\qquad\mathbf{E}_{n}^{1}\triangleq \mathscr{B}_{\perp}^{-1}\Pi_{\mathbb{S}_{0}}^{\perp}\mathscr{B}\mathbf{E}_{n}^{0}\Pi_{\mathbb{S}_{0}}^{\perp}.
	\end{equation}
	The estimate \eqref{e-En1} is obtained gathering \eqref{En1}, \eqref{e-scrB-scrB1}, \eqref{e-Bop perp}, \eqref{e-En0} and Lemma \ref{Lem-lawprod}-(iii). The expression \eqref{mu0-r1} follows from the Fourier representation \eqref{Ham-Fourier}. The estimates \eqref{e-r1-bis} correspond to \eqref{e-r1}-\eqref{e-12-Vbt}. The estimates \eqref{escr R Ods} and \eqref{escr R Ods diff} are obtained by the same method as Lemma \cite[Prop 6.3 and Lem. 6.3]{HR21} by using a nice duality representations of $\mathscr{B}_{\perp}^{\pm 1}$ together with \eqref{e1-J}, \eqref{e2-J}, \eqref{e-frkR}, \eqref{e-12-Vbt} and \eqref{e-dfrkR} In particular, the reversibility property of $\mathscr{R}_0$ is a consequence of \eqref{symVr}, \eqref{symJ}, \eqref{sym-beta}, \eqref{sym ker E-SW} and the reversibility property of $\mathbf{R}_{\varepsilon r}$. To prove \eqref{e-scrL0}, we use Lemma \ref{lem prop Toe}-(ii) together with \eqref{e-r1}, \eqref{escr R Ods} and 
	$$\forall\,0<\lambda_0<\lambda_1,\quad\sup_{j\in\mathbb{Z}}\Big(|j|\max_{k\in\llbracket 0,q\rrbracket}\big\|(I_jK_j)^{(k)}\big\|_{L^{\infty}([\lambda_0,\lambda_1])}\Big)\leqslant C(\lambda_0,\lambda_1).$$
	The last estimate has been proved in \cite[Lem. 5.1]{HR21}.
\end{proof}
The next step is to get rid of the remainder term $\mathscr{R}_0.$ The properties \eqref{escr R Ods}, \eqref{escr R Ods diff}, \eqref{e-scrL0} and Lemma \ref{properties omegajalpha}-(v) allows to start a KAM iterative procedure similar to \cite[Prop. 6.5]{HR21}. We also refer to \cite[Prop. 6.4]{HR21-1} for a brief explaination of this method. The result reads as follows.
\begin{prop}\label{prop redrem}
	Let $(d,s_{0},S,\gamma,q,q_0,\tau_{1},s_l,\overline{s}_l,\overline{s}_h, \overline{\mu}_2,\sigma_{3})$ as in \eqref{p-d}, \eqref{Sob index}, \eqref{p-cond0}, \eqref{def q}, \eqref{ptrans} and Proposition \ref{prop projnor}. Set
	\begin{equation}\label{to2}
		\tau_2\triangleq\tau_1+dq_0+1.
	\end{equation}
	For every choice of additional parameters $(\mu_2,s_h)$ enjoying the following constraints compatibles with \eqref{ptrans-2} for the choice $\mathtt{p}=4\tau_2 q+4\tau_2$,
	\begin{align}\label{prem}
		\mu_2\geqslant \overline{\mu}_2+2\tau_2q+2\tau_2,\qquad s_h\geqslant \frac{3}{2}\mu_{2}+\overline{s}_{l}+1,
	\end{align}
	there exist $\varepsilon_{0}\in(0,1)$ and $\sigma_{4}=\sigma_{4}(\tau_1,\tau_2,q,d)\geqslant\sigma_{3}$ such that if the following smallness condition holds
	\begin{equation}\label{sml redrem}
		\varepsilon\gamma^{-2-q}N_{0}^{\mu_{2}}\leqslant \varepsilon_{0},\qquad\|\mathfrak{I}_{0}\|_{q,s_{h}+\sigma_{4}}^{\gamma,\mathcal{O}}\leqslant 1,
	\end{equation}
	then there exists an operator $\Phi_{\infty}:\mathcal{O}\to \mathcal{L}\big(H^{s}\cap L_{\perp}^2\big)$ satisfying
		\begin{equation}\label{e-Phiinfty}
			\forall s\in[s_{0},S],\quad \mbox{ }\|\Phi_{\infty}^{\pm 1}\rho\|_{q,s}^{\gamma ,\mathcal{O}}\lesssim \|\rho\|_{q,s}^{\gamma,\mathcal{O}}+\varepsilon\gamma^{-2}\| \mathfrak{I}_{0}\|_{q,s+\sigma_{4}}^{\gamma,\mathcal{O}}\|\rho\|_{q,s_{0}}^{\gamma,\mathcal{O}}
		\end{equation}
		and a diagonal operator $\mathscr{L}_\infty=\mathscr{L}_{\infty}(\alpha,\omega,i_{0})$ in the form
		\begin{equation}\label{op scrLinfty}
			\mathscr{L}_{\infty}=\omega\cdot\partial_{\varphi}\Pi_{\mathbb{S}_0}^{\perp}+\mathscr{D}_{\infty}
		\end{equation}
		with $\mathscr{D}_{\infty}=\Pi_{\mathbb{S}_0}^{\perp}\mathscr{D}_{\infty}\Pi_{\mathbb{S}_0}^{\perp}$ a reversible Fourier multiplier operator given by,
		\begin{equation}\label{specDinfty}
			\mathscr{D}_{\infty}\mathbf{e}_{l,j}=\ii\,\mathrm{d}_{j}^{\infty}\,\mathbf{e}_{l,j},\qquad\mathrm{d}_{j}^{\infty}(\alpha,\omega,i_{0})\triangleq\mathrm{d}_{j}^{0}(\alpha,\omega,i_{0})+r_{j}^{\infty}(\alpha,\omega,i_{0}),\qquad\sup_{j\in\mathbb{S}_{0}^{c}}|j|\| r_{j}^{\infty}\|_{q}^{\gamma ,\mathcal{O}}\lesssim\varepsilon\gamma^{-1}
		\end{equation}
		such that in restriction to the Cantor-like set
		\begin{align}\label{Cantrem}
			\mathscr{O}_{\textnormal{\tiny{remainder}}}^{n,\gamma,\tau_1,\tau_{2}}(i_{0})\triangleq \bigcap_{\underset{|l|\leqslant N_{n}}{(l,j,j_{0})\in\mathbb{Z}^{d}\times(\mathbb{S}_0^{c})^{2}}\atop(l,j)\neq(0,j_{0})}\Big\{&(\alpha,\omega)\in\mathscr{O}_{\textnormal{\tiny{transport}}}^{n,\gamma,\tau_1}(i_{0}),\big|\omega\cdot l+\mathrm{d}_{j}^{\infty}(\alpha,\omega,i_{0})-\mathrm{d}_{j_{0}}^{\infty}(\alpha,\omega,i_{0})\big|>\tfrac{2\gamma \langle j-j_{0}\rangle}{\langle l\rangle^{\tau_{2}}}\Big\}
		\end{align}
		we have 
		\begin{equation}\label{e-3rd En}
			\Phi_{\infty}^{-1}\mathscr{L}_{0}\Phi_{\infty}=\mathscr{L}_{\infty}+\mathbf{E}^2_n,\qquad\|\mathbf{E}^2_n\rho\|_{q,s_0}^{\gamma,\mathcal{O}}\lesssim \varepsilon\gamma^{-2}N_{0}^{{\mu}_{2}}N_{n+1}^{-\mu_{2}} \|\rho\|_{q,s_0+1}^{\gamma,\mathcal{O}}.
		\end{equation}
		Recall that the Cantor set $\mathscr{O}_{\textnormal{\tiny{transport}}}^{n,\gamma,\tau_1}(i_{0})$, the operator $\mathscr{L}_{0}$ and the frequencies $\big(\mathrm{d}_{j}^{0}(\alpha,\omega,i_{0})\big)_{j\in\mathbb{S}_0^c}$ are respectively given by \eqref{Cantrans}, \eqref{def scr L0} and \eqref{mu0-r1}. Moreover, for two tori $i_{1}$ and $i_{2}$ both satisfying \eqref{sml redrem}, we have
		\begin{equation}\label{e-diff mujinfty}
			\forall j\in\mathbb{S}_{0}^{c},\quad\|\Delta_{12}r_{j}^{\infty}\|_{q}^{\gamma,\mathcal{O}}\lesssim\varepsilon\gamma^{-1}\|\Delta_{12}i\|_{q,\overline{s}_{h}+\sigma_{4}}^{\gamma,\mathcal{O}},\qquad\|\Delta_{12}\mathrm{d}_{j}^{\infty}\|_{q}^{\gamma,\mathcal{O}}\lesssim\varepsilon\gamma^{-1}|j|\| \Delta_{12}i\|_{q,\overline{s}_{h}+\sigma_{4}}^{\gamma,\mathcal{O}}.
		\end{equation}	
\end{prop}
Now, that we have completely diagonalized the operator $\mathscr{L}_{\perp}$ in \eqref{hLom-2} up to error terms, we can find an almost approximate right inverse for it. This is done by almost inverting the operator $\mathscr{L}_{\infty}$ in \eqref{op scrLinfty}. The result is stated below and its proof follows word by word \cite[Prop. 6.6]{HR21} or \cite[Prop. 6.5]{HR21-1}. Consequently, we omit the proof and only mention that the estimates mainly follows from \eqref{e-Bop perp}-\eqref{e-Phiinfty}-\eqref{e-En0}-\eqref{e-En1}-\eqref{e-3rd En} and Lemma \ref{Lem-lawprod}-(ii).
\begin{prop}\label{prop inv nor}
	Let $(d,s_{0},S,\gamma,q,\tau_{1},\tau_2,s_h,\mu_2,\sigma_{4})$ as in \eqref{p-d}, \eqref{Sob index}, \eqref{p-cond0}, \eqref{def q}, \eqref{ptrans}, \eqref{to2}, \eqref{prem} and Proposition \ref{prop redrem}.
	There exists $\sigma\triangleq \sigma(\tau_1,\tau_2,q,d)\geqslant\sigma_{4}$ such that if the following smallness condition holds
	\begin{equation}\label{sml inv nor}
		\varepsilon\gamma^{-2-q}N_0^{\mu_2}\leqslant\varepsilon_0,\qquad\|\mathfrak{I}_0\|_{q,s_h+\sigma}^{\gamma,\mathcal{O}}\leqslant 1,
	\end{equation}
	\begin{enumerate}[label=(\roman*)]
		\item There exists a family of operators $\big(\mathtt{T}_n\big)_{n\in\mathbb{N}}$ defined on $\mathcal{O}$ satisfying
		$$\forall s\in[s_0,S],\quad \sup_{n\in\mathbb{N}}\|\mathtt{T}_{n}\rho\|_{q,s}^{\gamma ,\mathcal{O}}\lesssim \gamma ^{-1}\|\rho\|_{q,s+\tau_{1}q+\tau_{1}}^{\gamma ,\mathcal{O}}$$
		and such that for any $n\in\mathbb{N}$, in restriction to the Cantor set
		\begin{equation}\label{Lbd set}
			\mathscr{O}_{\textnormal{\tiny{inversion}}}^{n,\gamma,\tau_1}(i_{0})\triangleq\bigcap_{(l,j)\in\mathbb{Z}^{d }\times\mathbb{S}_{0}^{c}\atop |l|\leqslant N_{n}}\left\lbrace(\alpha,\omega)\in\mathcal{O}\quad\textnormal{s.t.}\quad\left|\omega\cdot l+\mathrm{d}_{j}^{\infty}(\alpha,\omega,i_{0})\right|>\tfrac{\gamma \langle j\rangle }{\langle l\rangle^{\tau_{1}}}\right\rbrace,
		\end{equation}
		the following identity holds
		$$\mathscr{L}_{\infty}\mathtt{T}_n=\textnormal{Id}+\mathbf{E}^3_n,\qquad\forall s_{0}\leqslant s\leqslant\overline{s}\leqslant S, \quad \|\mathbf{E}^3_{n}\rho\|_{q,s}^{\gamma ,\mathcal{O}} \lesssim 
		N_n^{s-\overline{s}}\gamma^{-1}\|\rho\|_{q,\overline{s}+1+\tau_{1}q+\tau_{1}}^{\gamma ,\mathcal{O}}.$$
		This means that we have an almost approximate right inverse for the operator $\mathscr{L}_{\infty}$ in \eqref{op scrLinfty}.
		\item 
		There exists a family of operators $\big(\mathtt{T}_{\perp,n}\big)_{n\in\mathbb{N}}$ defined on $\mathcal{O}$ satisfying
		$$\forall \, s\in\,[ s_0, S],\quad\sup_{n\in\mathbb{N}}\|\mathtt{T}_{\perp,n}\rho\|_{q,s}^{\gamma ,\mathcal{O}}\lesssim\gamma^{-1}\left(\|\rho\|_{q,s+\sigma}^{\gamma ,\mathcal{O}}+\| \mathfrak{I}_{0}\|_{q,s+\sigma}^{\gamma ,\mathcal{O}}\|\rho\|_{q,s_{0}+\sigma}^{\gamma,\mathcal{O}}\right)$$
		and such that in restriction to the Cantor set
		\begin{equation}\label{calGn}
			\mathscr{G}_n(\gamma,\tau_{1},\tau_{2},i_{0})\triangleq \mathscr{O}_{\textnormal{\tiny{transport}}}^{n,\gamma,\tau_1}(i_{0})\cap\mathscr{O}_{\textnormal{\tiny{remainder}}}^{n,\gamma,\tau_1,\tau_{2}}(i_{0})\cap\mathscr{O}_{\textnormal{\tiny{inversion}}}^{n,\gamma,\tau_1}(i_{0}),
		\end{equation}
		the following identity holds
		$$\mathscr{L}_{\perp}\mathtt{T}_{\perp,n}=\textnormal{Id}+\mathbf{E}_n,$$
		with $\mathbf{E}_n$ satisfying
		\begin{align*}
			\forall\, s\in [s_0,S],\quad  &\|\mathbf{E}_n\rho\|_{q,s_0}^{\gamma ,\mathcal{O}}\lesssim N_n^{s_0-s}\gamma^{-1}\Big( \|\rho\|_{q,s+\sigma}^{\gamma,\mathcal{O}}+\varepsilon\gamma^{-2}\| \mathfrak{I}_{0}\|_{q,s+\sigma}^{\gamma,\mathcal{O}}\|\rho\|_{q,s_{0}}^{\gamma,\mathcal{O}} \Big)\nonumber\\
			&\qquad\qquad\quad+ \varepsilon\gamma^{-3}N_{0}^{{\mu}_{2}}N_{n+1}^{-\mu_{2}} \|\rho\|_{q,s_0+\sigma}^{\gamma,\mathcal{O}}.
		\end{align*}
		Recall that  $\mathscr{L}_{\perp},$ $\mathscr{O}_{\textnormal{\tiny{transport}}}^{n,\gamma,\tau_1}(i_{0})$ and $\mathscr{O}_{\textnormal{\tiny{remainder}}}^{n,\gamma,\tau_1,\tau_{2}}(i_{0})$ are respectively defined in \eqref{hLom-2}, \eqref{Cantrans} and \eqref{Cantrem}.
		\item When restricted to the Cantor set $\mathscr{G}_{n}(\gamma,\tau_{1},\tau_{2},i_{0})$, we have also have the following splitting required to apply the Berti-Bolle theory
		$$\mathscr{L}_{\perp}=\mathtt{L}_{\perp,n}+\mathtt{R}_{\perp,n},\qquad\mathtt{L}_{\perp,n}\mathtt{T}_{\perp,n}=\textnormal{Id},\qquad\mathtt{R}_{\perp,n}=\mathbf{E}_{n}\mathtt{L}_{\perp,n},$$
		with the operators $\mathtt{L}_{\perp,n}$ and $\mathtt{R}_{\perp,n}$ defined in  $\mathcal{O}$ and satisfying 
		\begin{align*}
			\forall s\in[s_{0},S],\quad& \sup_{n\in\mathbb{N}}\|\mathtt{L}_{\perp,n}\rho\|_{q,s}^{\gamma,\mathcal{O}}\lesssim\|\rho\|_{q,s+1}^{\gamma,\mathcal{O}}+\varepsilon\gamma^{-2}\|\mathfrak{I}_{0}\|_{q,s+\sigma}^{\gamma,\mathcal{O}}\|\rho\|_{q,s_{0}+1}^{\gamma,\mathcal{O}},\\
			\forall s\in[s_{0},S],\quad &\|\mathtt{R}_{\perp,n}\rho\|_{q,s_{0}}^{\gamma,\mathcal{O}}\lesssim N_{n}^{s_{0}-s}\gamma^{-1}\left(\|\rho\|_{q,s+\sigma}^{\gamma,\mathcal{O}}+\varepsilon\gamma^{-2}\|\mathfrak{I}_{0}\|_{q,s+\sigma}^{\gamma,\mathcal{O}}\|\rho\|_{q,s_{0}+\sigma}^{\gamma,\mathcal{O}}\right)\\
			&\qquad\qquad\quad+\varepsilon\gamma^{-3}N_{0}^{\mu_{2}}N_{n+1}^{-\mu_{2}}\|\rho\|_{q,s_{0}+\sigma}^{\gamma,\mathcal{O}}.
		\end{align*}
	\end{enumerate}
\end{prop}
Finally, as mentioned at the beginning of this subsection, we can apply the Berti-Bolle theory \cite{BB15} and \cite[Sec. 6]{HHM21} to construct an approximate right inverse for the full linearized operator $d_{i,\kappa }\mathscr{F}(i_0,\kappa _0).$ For a complete proof of the result, the reader is refered to \cite[Thm. 6.1]{HHM21}.
\begin{theo}\label{thm ai}
	\textbf{(Almost approximate right inverse)}\\
	Let $(d,s_{0},S,\gamma,q,\tau_{1},\tau_2,s_h,\mu_{2})$ satisfy  \eqref{p-d}, \eqref{Sob index}, \eqref{p-cond0}, \eqref{def q}, \eqref{ptrans}, \eqref{to2} and \eqref{prem}.
	Then there exists $ \overline{\sigma}= \overline{\sigma}(\tau_1,\tau_2,d,q)>0$ and a family of reversible operators $\big(\mathrm{T}_{0,n}\big)_{n\in\mathbb{N}}$ such that if the smallness condition \eqref{sml inv nor} holds, then
	for all $g=(g_1,g_2,g_3)$, satisfying 
	$$g_1(\varphi)=g_1(\varphi),\qquad g_2(-\varphi)=-g_2(\varphi)\qquad g_3(-\varphi)=(\mathscr{S}g_3)(\varphi),$$  
	the function $\mathrm{T}_{0,n}g$ satisfies
	$$\forall s\in [s_0,S],\quad \|\mathrm{T}_{0,n} g\|_{q,s}^{\gamma,\mathcal{O}}\lesssim\gamma^{-1}\left(\|g\|_{q,s+\overline{\sigma}}^{\gamma,\mathcal{O}}+\|\mathfrak{I}_{0}\|_{q,s+\overline{\sigma}}^{\gamma,\mathcal{O}}\|g\|_{q,s_{0}+\overline{\sigma}}^{\gamma,\mathcal{O}}\right).$$
	Moreover $\mathrm{T}_{0,n}$ is an almost-approximate  
	right inverse of $d_{i,\kappa }\mathscr{F}(i_0,\kappa _0)$ in the Cantor set $\mathscr{G}_{n}(\gamma,\tau_1,\tau_2,i_0)$. More precisely, for all $ (\alpha,\omega)\in \mathscr{G}_{n}(\gamma,\tau_1,\tau_2,i_0)$ we can write
	$$d_{i,\kappa }\mathscr{F}(i_0,\kappa _0)\circ\mathrm{T}_{0,n}-\textnormal{Id}=\mathscr{E}^{(n)}_1+\mathscr{E}^{(n)}_2+\mathscr{E}^{(n)}_3,$$
	where the operators $\mathscr{E}^{(n)}_1$, $\mathscr{E}^{(n)}_2$ and $\mathscr{E}^{(n)}_3$ are defined in the whole set $\mathcal{O}$ with the estimates
	\begin{align*}
		\|\mathscr{E}^{(n)}_1 g\|_{q,s_0}^{\gamma,\mathcal{O}}&\lesssim\gamma^{-1}\|\mathscr{F}(i_0,\kappa _0)\|_{q,s_0+\overline{\sigma}}^{\gamma,\mathcal{O}}\|g\|_{q,s_0+\overline{\sigma}}^{\gamma,\mathcal{O}},\\
		\forall\,\mathtt{b}\geqslant 0,\quad\|\mathscr{E}^{(n)}_2 g\|_{q,s_0}^{\gamma,\mathcal{O}}&\lesssim\gamma^{-1}N_n^{-\mathtt{b}} \Big(\|g\|_{q,s_0+\mathtt{b}+\overline{\sigma}}^{\gamma,\mathcal{O}}+\varepsilon\|\mathfrak{I}_{0}\|_{q,s_0+\mathtt{b}+\overline{\sigma}}^{\gamma,\mathcal{O}}\big\|g\|_{q,s_0+\overline{\sigma}}^{\gamma,\mathcal{O}}\Big)\,,\\
		\forall\,\mathtt{b}\in [0,S-s_0],\quad\|\mathscr{E}^{(n)}_3 g\|_{q,s_0}^{\gamma,\mathcal{O}}&\lesssim N_n^{-\mathtt{b}}\gamma^{-2}\Big(\|g\|_{q,s_0+\mathtt{b}+\overline{\sigma}}^{\gamma,\mathcal{O}}+{\varepsilon\gamma^{-2}}\|\mathfrak{I}_{0}\|_{q,s_0+\mathtt{b}+\overline{\sigma}}^{\gamma,\mathcal{O}}\|g\|_{q,s_0+\overline{\sigma}}^{\gamma,\mathcal{O}}\Big)\\
		&\quad+ \varepsilon\gamma^{-4}N_{0}^{{\mu}_{2}}{N_{n}^{-\mu_{2}}} \|g\|_{q,s_0+\overline{\sigma}}^{\gamma,\mathcal{O}}.
	\end{align*}
\end{theo}
\subsection{Construction of a non-trivial quasi-periodic solution}\label{sec non triv QP sol}
In this final section, we construct a non-trivial zero for the functional $\mathscr{F}$ defined by \eqref{def scrF}. It is obtained by a Nash-Moser iteration procedure. At each step, we can find an approximate right inverse of the linearized operator following the construction explained in the previous subsection. The proof is technical and already detailed in the previous works \cite{BM18,HHM21,HR21}. Here, we may use the version exposed in \cite[Prop. 7.1 and Cor. 7.1]{HR21}.
\begin{prop}\label{Nash-Moser}
	\textbf{(Nash-Moser)}
	\begin{enumerate}[label=(\roman*)]
		\item Let $(d,s_{0},S,q,\tau_{1},\tau_{2},\overline{\sigma})$ as in \eqref{p-d}, \eqref{Sob index}, \eqref{def q}, \eqref{ptrans}, \eqref{to2} and Theorem \ref{thm ai}. Consider the parameters fixed by 
	\begin{equation}\label{p-NM}
		\begin{array}{c}
		\mu_1\triangleq 3q(\tau_{2}+2)+6\overline{\sigma}+6, \qquad a_{1}\triangleq6q(\tau_{2}+2)+12\overline{\sigma}+15,\\
		\mu_{2}\triangleq2q(\tau_{2}+2)+5\overline{\sigma}+7,\qquad a_{2}\triangleq3q(\tau_{2}+2)+6\overline{\sigma}+9,\\
		s_{h}\triangleq s_{0}+4q(\tau_{2}+2)+9\overline{\sigma}+11,\qquad b_{1}\triangleq2s_{h}-s_{0},\qquad \overline{a}\triangleq\tau_{2}+2
		\end{array}
	\end{equation}
 and 
 \begin{equation}\label{gma N0 NM}
 	0<a<\tfrac{1}{\mu_{2}+q+2},\qquad \gamma\triangleq\varepsilon^{a},\qquad N_{0}\triangleq\gamma^{-1}.
 \end{equation}
We consider the finite dimensional subspaces
\begin{align*}
	\mathtt{E}_{n}\triangleq &\Big\{\mathfrak{I}=(\Theta,I,z)\quad
	\textnormal{s.t.}\quad\Theta=P_{N_n}\Theta,\quad I=P_{N_n}I\quad\textnormal{and}\quad z=P_{N_n}z\Big\},
\end{align*}
where $P_{N}$ is the projector defined by
$$f(\varphi,\theta)=\sum_{(l,j)\in\mathbb{Z}^{d}\times\mathbb{Z}}f_{l,j}e^{\ii(l\cdot\varphi+j\theta)}\quad\Rightarrow\quad P_{N}f(\varphi,\theta)=\sum_{\langle l,j\rangle\leqslant N}f_{l,j}e^{\ii(l\cdot\varphi+j\theta)}.$$
	There exist $C_{\ast}>0$ and $\varepsilon_{0}>0$ such that for any $\varepsilon\in[0,\varepsilon_{0}]$ we get for all $n\in\mathbb{N}$ the following properties,
	\begin{enumerate}
		\item There exists a $q$-times differentiable function 
		$$\mathtt{W}_{n}:\begin{array}[t]{rcl}
			\mathcal{O} & \rightarrow &  \mathtt{E}_{n-1}\times\mathbb{R}^{d}\times\mathbb{R}^{d+1}\\
			(\alpha,\omega) & \mapsto & \big(\mathfrak{I}_{n}(\alpha,\omega),\kappa _{n}(\alpha,\omega)-\omega,0\big)
		\end{array}$$
		satisfying $\mathtt{W}_{0}=0$ and if $n\in\mathbb{N}^*,$
		$$\|\mathtt{W}_{n}\|_{q,s_{0}+\overline{\sigma}}^{\gamma,\mathcal{O}}\leqslant C_{\ast}\varepsilon\gamma^{-1}N_{0}^{q\overline{a}},\qquad\|\mathtt{W}_{n}\|_{q,b_{1}+\overline{\sigma}}^{\gamma,\mathcal{O}}\leqslant C_{\ast}\varepsilon\gamma^{-1}N_{n-1}^{\mu_1}.$$
		We set
		$$\mathtt{U}_{0}\triangleq\Big((\varphi,0,0),\omega,(\alpha,\omega)\Big)\quad\textnormal{and}\quad  \hbox{if}\quad n\in\mathbb{N}^*,\quad \mathtt{U}_{n}\triangleq \mathtt{U}_{0}+\mathtt{W}_{n},\qquad \mathtt{H}_{n} \triangleq \mathtt{U}_{n}-\mathtt{U}_{n-1}.$$
		Then 
		$$\forall s\in[s_{0},S],\,\| \mathtt{H}_{1}\|_{q,s}^{\gamma,\mathcal{O}}\leqslant \frac{1}{2}C_{\ast}\varepsilon\gamma^{-1}N_{0}^{q\overline{a}},\qquad\forall\, 2\leqslant k\leqslant n,\,\| \mathtt{H}_{k}\|_{q,s_{0}+\overline{\sigma}}^{\gamma,\mathcal{O}}\leqslant C_{\ast}\varepsilon\gamma^{-1}N_{k-1}^{-a_{2}}.$$
		We also have for $n\geqslant 2,$
		\begin{equation}\label{e-Hn-diff} \|\mathtt{H}_{n}\|_{q,\overline{s}_h+\sigma_{4}}^{\gamma,\mathcal{O}}\leqslant C_{\ast}\varepsilon\gamma^{-1}N_{n-1}^{-a_{2}}.
		\end{equation}
		\item Define 
		\begin{equation}\label{def gamman}
			i_{n}\triangleq(\varphi,0,0)+\mathfrak{I}_{n},\qquad \gamma_{n}\triangleq\gamma(1+2^{-n})\in[\gamma,2\gamma].
		\end{equation}
		The torus $i_n$ is reversible. Define also
		\begin{equation}\label{Angm}
			\mathscr{A}_{0}^{\gamma}\triangleq\mathcal{O},\qquad \mathscr{A}_{n+1}^{\gamma}\triangleq\mathscr{A}_{n}^{\gamma}\cap\mathscr{G}_{n}(\gamma_{n+1},\tau_{1},\tau_{2},i_{n}),
		\end{equation}
		with $\mathscr{G}_{n}(\gamma_{n+1},\tau_{1},\tau_{2},i_{n})$ as in \eqref{calGn}. Consider the open sets 
		$$\mathrm{O}_{n}^\gamma\triangleq \Big\{(\alpha,\omega)\in\mathcal{O}\quad\textnormal{s.t.}\quad {\mathtt{dist}}\big((\alpha,\omega),\mathscr{A}_{n}^{\gamma}\big)<\gamma N_{n}^{-\overline{a}}\Big\},\qquad\mathtt{dist}(x,A)\triangleq\inf_{y\in A}\|x-y\|.$$
		Then we have
		$$\|\mathscr{F}(\mathtt{U}_{n})\|_{q,s_{0}}^{\gamma,\mathrm{O}_{n}^{\gamma}}\leqslant C_{\ast}\varepsilon N_{n-1}^{-a_{1}}.$$
	\end{enumerate}
	\item There exists $\varepsilon_0>0$ such that, for any $\varepsilon\in(0,\varepsilon_0),$ the following hold true. We consider the Cantor set $\mathscr{G}_{\infty}^{\gamma}$ (related to $\varepsilon$ through $\gamma$) and given by
	\begin{equation}\label{Ginfty}
		\mathscr{G}_{\infty}^{\gamma}\triangleq \bigcap_{n\in\mathbb{N}}\mathscr{A}_{n}^{\gamma}.
	\end{equation}
	There exists a $q$-times differentiable function
	$$\mathtt{U}_{\infty}:\begin{array}[t]{rcl}
		\mathcal{O} & \rightarrow & \left(\mathbb{T}^{d}\times\mathbb{R}^{d}\times H^{s_0}\cap L_{\perp}^{2}\right)\times\mathbb{R}^{d}\times\mathbb{R}^{d+1}\\
		(\alpha,\omega) & \mapsto & \big(i_{\infty}(\alpha,\omega),\kappa _{\infty}(\alpha,\omega),(\alpha,\omega)\big)
	\end{array}$$
	such that 
	$$\forall(\alpha,\omega)\in\mathscr{G}_{\infty}^{\gamma},\quad\mathscr{F}(\mathtt{U}_{\infty}(\alpha,\omega))=0.$$
	The torus $i_{\infty}$ is reversible and $\kappa _{\infty}\in W^{q,\infty,\gamma}(\mathcal{O},\mathbb{R}^d).$
	Furthermore, there exists a $q$-times differentiable function $\alpha\in(\alpha_{0},\alpha_{1})\mapsto\omega(\alpha,\varepsilon)$ such that
	\begin{equation}\label{alpha infty-1}
		\kappa _{\infty}\big(\alpha,\omega(\alpha,\varepsilon)\big)=-\omega_{\textnormal{Eq}}(\alpha),\qquad\omega(\alpha,\varepsilon)=-\omega_{\textnormal{Eq}}(\alpha)+\bar{r}_{\varepsilon}(\alpha),\qquad \|\bar{r}_{\varepsilon}\|_{q}^{\gamma,\mathcal{O}}\lesssim\varepsilon\gamma^{-1}N_{0}^{q\overline{a}},
	\end{equation}
	and 
	$$\forall \alpha\in \mathscr{C}_{\infty}^{\varepsilon},\quad \mathscr{F}\big(\mathtt{U}_{\infty}(\alpha,\omega(\alpha,\varepsilon))\big)=0,$$
	where the  Cantor set $\mathscr{C}_{\infty}^{\varepsilon}$ is defined by 
	\begin{equation}\label{Cinfty}
		\mathscr{C}_{\infty}^{\varepsilon}\triangleq\Big\{\alpha\in(\alpha_{0},\alpha_{1})\quad\textnormal{s.t.}\quad\big(\alpha,\omega(\alpha,\varepsilon)\big)\in\mathscr{G}_{\infty}^{\gamma}\Big\}.
	\end{equation}
\end{enumerate}
\end{prop}
Now to conclude the proof of Theorem \ref{thm QPS Ea}, it remains to prove that the Cantor set $\mathscr{C}_{\infty}^{\varepsilon}$ in \eqref{Cinfty} is not empty. Actually, we can prove a lower bound measure for $\mathscr{C}_{\infty}^{\varepsilon}$ which shows that when the magnitude of the perturbation tends to zero, then the set of admissible parameters tends to be of full Lebesgue measure in $(\alpha_0,\alpha_1).$  
\begin{prop}\label{prop meas Cant}
	Let $(q_{0},\upsilon,a)$ as in Lemma \ref{lem trsvrslt Ea}, \eqref{ptrans} and \eqref{gma N0 NM}. Then there exists $C>0$ such that 
	$$\big|\mathscr{C}_{\infty}^{\varepsilon}\big|\geqslant (\alpha_1-\alpha_0)-C \varepsilon^{\frac{a\upsilon}{q_{0}}},\qquad\textnormal{implying in turn}\qquad \lim_{\varepsilon\to0}\big|\mathscr{C}_{\infty}^{\varepsilon}\big|=\alpha_1-\alpha_0.$$
\end{prop}
\begin{proof}
	In view of \eqref{Cinfty} and \eqref{Ginfty}, we can write
	\begin{equation}\label{fin Cant set}
		\mathscr{C}_{\infty}^{\varepsilon}= \bigcap_{n\in\mathbb{N}}\mathscr{C}_{n}^{\varepsilon},\qquad \mathscr{C}_{n}^{\varepsilon}\triangleq \Big\{\alpha\in(\alpha_{0},\alpha_{1})\quad\hbox{s.t}\quad \big(\alpha,{\omega}(\alpha,\varepsilon)\big)\in\mathscr{A}_n^{\gamma}\Big\}.
	\end{equation}
Hence,
\begin{equation}\label{dec-Cinfty}
	(\alpha_{0},\alpha_{1})\setminus\mathscr{C}_{\infty}^{\varepsilon}=\big((\alpha_{0},\alpha_{1})\setminus\mathscr{C}_{0}^{\varepsilon}\big)\sqcup\bigsqcup_{n=0}^{\infty}\big(\mathscr{C}_{n}^{\varepsilon}\setminus\mathscr{C}_{n+1}^{\varepsilon}\big).
\end{equation}
Observe that \eqref{alpha infty-1}, \eqref{p-NM} and \eqref{gma N0 NM} imply
$$\sup_{\alpha\in(\alpha_0,\alpha_1)}\left|\omega(\alpha,\varepsilon)+\omega_{\textnormal{Eq}}(\alpha)\right|\leqslant\|\bar{\mathrm{r}}_{\varepsilon}\|_{q}^{\gamma,\mathcal{O}}\leqslant C\varepsilon\gamma^{-1}N_{0}^{q\overline{a}}=C\varepsilon^{1-a(1+q\overline{a})},\qquad 0<a<\frac{1}{1+q\overline{a}}.$$
Now by construction \eqref{def calO}, the previous estimate implies for $\varepsilon$ small enough
$$\forall \alpha\in(\alpha_0,\alpha_1),\quad \omega(\alpha,\varepsilon)\in B(0,R_{0}).$$
We immediately deduce that
\begin{equation}\label{e-Cant comp}
	\mathscr{C}_{0}^{\varepsilon}=(\alpha_0,\alpha_1),\qquad\Big|(\alpha_{0},\alpha_{1})\setminus\mathscr{C}_{\infty}^{\varepsilon}\Big|\leqslant\sum_{n=0}^{\infty}\Big|\mathscr{C}_{n}^{\varepsilon}\setminus\mathscr{C}_{n+1}^{\varepsilon}\Big|.
\end{equation}
	According to the notations \eqref{specDinfty} and \eqref{mu0-r1}, we denote the perturbed frequencies associated with the reduced linearized operator at state $i_n$ in the following way  
	\begin{align}\label{mujinftyn}
		\mathrm{d}_{j}^{\infty,n}(\alpha,\varepsilon)&\triangleq \mathrm{d}_{j}^{\infty}\big(\alpha,{\omega}(\alpha,\varepsilon),i_{n}\big)=\Omega_{j}^{\textnormal{\tiny{E}}}(\alpha)+jr^{0,n}(\alpha,\varepsilon)+r_{j}^{\infty,n}(\alpha,\varepsilon),
	\end{align}
	with
	\begin{equation}\label{r1-Vn-rn}
		r^{0,n}(\alpha,\varepsilon)\triangleq \mathtt{m}_{n}^{\infty}(\alpha,\varepsilon)-V_0(\alpha),\qquad \mathtt{m}_{n}^{\infty}(\alpha,\varepsilon)\triangleq \mathtt{m}_{i_{n}}^{\infty}(\alpha,{\omega}(\alpha,\varepsilon)),\qquad r_{j}^{\infty,n}(\alpha,\varepsilon)\triangleq r_{j}^{\infty}\big(\alpha,{\omega}(\alpha,\varepsilon),i_{n}\big).
	\end{equation}
	Now, by construction, one can write in view of \eqref{fin Cant set}, \eqref{Angm}, \eqref{calGn}, \eqref{Cantrans}, \eqref{Cantrem} and \eqref{Lbd set} for any $n\in\mathbb{N}$, 
	\begin{equation}\label{split Cn-Cn+1}
		\mathscr{C}_{n}^{\varepsilon}\setminus\mathscr{C}_{n+1}^{\varepsilon}=\bigcup_{(l,j)\in\mathbb{Z}^{d}\times\mathbb{Z}\setminus\{(0,0)\}\atop |l|\leqslant N_{n}}\mathscr{R}_{l,j}^{(0)}(i_{n})\bigcup_{(l,j,j_{0})\in\mathbb{Z}^{d}\times(\mathbb{S}_{0}^{c})^{2}\atop |l|\leqslant N_{n}}\mathscr{R}_{l,j,j_{0}}^{(2)}(i_{n})\bigcup_{(l,j)\in\mathbb{Z}^{d}\times\mathbb{S}_{0}^{c}\atop |l|\leqslant N_{n}}\mathscr{R}_{l,j}^{(1)}(i_{n}),
	\end{equation}
	where
	\begin{align*}
		\mathscr{R}_{l,j}^{(0)}(i_{n})&\triangleq \left\lbrace \alpha\in\mathscr{C}_{n}^{\varepsilon}\quad\textnormal{s.t.}\quad\Big|{\omega}(\alpha,\varepsilon)\cdot l+j\mathtt{m}_{n}^{\infty}(\alpha,\varepsilon)\Big|\leqslant\tfrac{4\gamma_{n+1}^{\upsilon}\langle j\rangle}{\langle l\rangle^{\tau_{1}}}\right\rbrace,\\
		\mathscr{R}_{l,j,j_{0}}^{(2)}(i_{n})&\triangleq \left\lbrace \alpha\in\mathscr{C}_{n}^{\varepsilon}\quad\textnormal{s.t.}\quad\Big|{\omega}(\alpha,\varepsilon)\cdot l+\mathrm{d}_{j}^{\infty,n}(\alpha,\varepsilon)-\mathrm{d}_{j_{0}}^{\infty,n}(\alpha,\varepsilon)\Big|\leqslant\tfrac{2\gamma_{n+1}\langle j-j_{0}\rangle}{\langle l\rangle^{\tau_{2}}}\right\rbrace,\\
		\mathscr{R}_{l,j}^{(1)}(i_{n})&\triangleq \left\lbrace \alpha\in\mathscr{C}_{n}^{\varepsilon}\quad\textnormal{s.t.}\quad\Big|{\omega}(\alpha,\varepsilon)\cdot l+\mathrm{d}_{j}^{\infty,n}(\alpha,\varepsilon)\Big|\leqslant\tfrac{\gamma_{n+1}\langle j\rangle}{\langle l\rangle^{\tau_{1}}}\right\rbrace.
	\end{align*}
	Copying the proof of \cite[Lem. 7.1]{HR21} we can show, using in particular \eqref{e-Hn-diff}-\eqref{e-r1-bis}-\eqref{e-diff mujinfty}, that for any $n\in\mathbb{N}\setminus\{0,1\}$ and any $l\in\mathbb{Z}^{d}$ such that $|l|\leqslant N_{n-1},$ the following properties hold.
	\begin{enumerate}[label=\textbullet]
		\item For $ j\in\mathbb{Z}$ with $(l,j)\neq(0,0)$, we get  $\,\,\mathscr{R}_{l,j}^{(0)}(i_{n})=\varnothing.$
		\item For  $ (j,j_{0})\in(\mathbb{S}_{0}^{c})^{2}$ with $(l,j)\neq(0,j_0),$ we get $\,\,\mathscr{R}_{l,j,j_{0}}^{(2)}(i_{n})=\varnothing.$
		\item For  $j\in\mathbb{S}_{0}^{c}$, we get $\,\,\mathscr{R}_{l,j}^{(1)}(i_{n})=\varnothing.$
	\end{enumerate}
This leads to write for any $n\in\mathbb{N}\setminus\{0,1\},$
		\begin{equation}\label{decomp Cnbis}
			\mathscr{C}_{n}^{\varepsilon}\setminus\mathscr{C}_{n+1}^{\varepsilon}=\bigcup_{\underset{N_{n-1}<|l|\leqslant N_{n}}{(l,j)\in\mathbb{Z}^{d}\times\mathbb{Z}\setminus\{(0,0)\}}}\mathscr{R}_{l,j}^{(0)}(i_{n})\cup\bigcup_{\underset{N_{n-1}<|l|\leqslant N_{n}}{(l,j,j_{0})\in\mathbb{Z}^{d}\times(\mathbb{S}_{0}^{c})^{2}}}\mathscr{R}_{l,j,j_{0}}^{(2)}(i_{n})\cup\bigcup_{\underset{N_{n-1}<|l|\leqslant N_{n}}{(l,j)\in\mathbb{Z}^{d}\times\mathbb{S}_{0}^{c}}}\mathscr{R}_{l,j}^{(1)}(i_{n}).
		\end{equation}
To estimate, we shall make use of Rüssmann's Lemma \ref{lem Russmann measure}. First notice that
	$$W^{q,\infty,\gamma}(\mathcal{O},\mathbb{C})\hookrightarrow C^{q-1}(\mathcal{O},\mathbb{C}),\qquad q=q_0+1,$$
	imply that for any $n\in\mathbb{N}$, the $C^{q_0}$ regularity for the curves
		$$\begin{array}{ll}
			\displaystyle \alpha\mapsto\omega(\alpha,\varepsilon)\cdot l+j\mathtt{m}_{n}^{\infty}(\alpha,\varepsilon),&\quad (l,j)\in\mathbb{Z}^{d}\times\mathbb{Z}\backslash\{(0,0)\},\vspace{0.1cm}\\
			\displaystyle \alpha\mapsto\omega(\alpha,\varepsilon)\cdot l+\mathrm{d}_{j}^{\infty,n}(\alpha,\varepsilon)-\mathrm{d}_{j_{0}}^{\infty,n}(\alpha,\varepsilon),&\quad (l,j,j_0)\in\mathbb{Z}^{d}\times(\mathbb{S}_0^c)^{2},\vspace{0.1cm}\\
			\displaystyle \alpha\mapsto\omega(\alpha,\varepsilon)\cdot l+\mathrm{d}_{j}^{\infty,n}(\alpha,\varepsilon),&\quad (l,j)\in\mathbb{Z}^{d}\times\mathbb{S}_0^c.
		\end{array}$$
	In addition, the following perturbed Rüssmann conditions hold. They are obtained by a perturbative argument from the equilibrium transversality conditions in Lemma \ref{lem trsvrslt Ea} similarly to \cite[Lem. 7.3]{HR21} : for $q_{0}$, $C_{0}$ and $\rho_{0}$ as in Lemma \ref{lem trsvrslt Ea}, there exist $\varepsilon_{0}>0$ small enough such that for any   $\varepsilon\in[0,\varepsilon_{0}]$ we have
	\begin{enumerate}[label=\textbullet]
		\item For all $(l,j)\in\mathbb{Z}^{d+1}\setminus\{(0,0)\}$ such that $|j|\leqslant C_{0}\langle l\rangle,$ we have
		$$\forall n\in\mathbb{N},\quad\inf_{\alpha\in[\alpha_{0},\alpha_{1}]}\max_{k\in\llbracket 0,q_{0}\rrbracket}|\partial_{\alpha}^{k}\big({\omega}(\alpha,\varepsilon)\cdot l+j\mathtt{m}_{n}^{\infty}(\alpha,\varepsilon)\big)|\geqslant\tfrac{\rho_{0}\langle l\rangle}{2}.$$
		\item For all $(l,j)\in\mathbb{Z}^{d}\times\mathbb{S}_{0}^{c}$ such that $|j|\leqslant C_{0}\langle l\rangle,$ we have
		$$\forall n\in\mathbb{N},\quad\inf_{\alpha\in[\alpha_{0},\alpha_{1}]}\max_{k\in\llbracket 0,q_{0}\rrbracket}\big|\partial_{\alpha}^{k}\big({\omega}(\alpha,\varepsilon)\cdot l+\mathrm{d}_{j}^{\infty,n}(\alpha,\varepsilon)\big)\big|\geqslant\tfrac{\rho_{0}\langle l\rangle}{2}.$$
		\item For all $(l,j,j_{0})\in\mathbb{Z}^{d}\times(\mathbb{S}_{0}^{c})^{2}$ such that $|j-j_{0}|\leqslant C_{0}\langle l\rangle,$ we have
		$$\forall n\in\mathbb{N},\quad\inf_{\alpha\in[\alpha_{0},\alpha_{1}]}\max_{k\in\llbracket 0,q_{0}\rrbracket}\big|\partial_{\alpha}^{k}\big({\omega}(\alpha,\varepsilon)\cdot l+\mathrm{d}_{j}^{\infty,n}(\alpha,\varepsilon)-\mathrm{d}_{j_{0}}^{\infty,n}(\alpha,\varepsilon)\big)\big|\geqslant\tfrac{\rho_{0}\langle l\rangle}{2}.$$
\end{enumerate}
\noindent Notice that the proof of the previous transversality conditions requires in particular the estimates \eqref{unif e-r1} and \eqref{unif e-rjfty}. Thus, applying Lemma \ref{lem Russmann measure}, we get for all $n\in\mathbb{N}$, 
	\begin{align}\label{e-scrR0112}
		\Big|\mathscr{R}_{l,j}^{(0)}(i_{n})\Big|&\lesssim\gamma^{\frac{\upsilon}{q_{0}}}\langle j\rangle^{\frac{1}{q_{0}}}\langle l\rangle^{-1-\frac{\tau_{1}+1}{q_{0}}},\nonumber\\
		\Big|\mathscr{R}_{l,j}^{(1)}(i_{n})\Big|&\lesssim\gamma^{\frac{1}{q_{0}}}\langle j\rangle^{\frac{1}{q_{0}}}\langle l\rangle^{-1-\frac{\tau_{1}+1}{q_{0}}},\\
		\Big|\mathscr{R}_{l,j,j_{0}}^{(2)}(i_{n})\Big|&\lesssim\gamma^{\frac{1}{q_{0}}}\langle j-j_{0}\rangle^{\frac{1}{q_{0}}}\langle l\rangle^{-1-\frac{\tau_{2}+1}{q_{0}}}.\nonumber
	\end{align}
Since \eqref{decomp Cnbis} is valid for $n\in\mathbb{N}\setminus\{0,1\}$, we first need to estimate the first two terms in the right hand-side of \eqref{e-Cant comp}. Fix $k\in\{0,1\},$ then we have by Lemma \ref{lem empty Cant}

	\begin{align}\label{e-Ck-Ck+1}
		\Big|\mathscr{C}_{k}\setminus\mathscr{C}_{k+1}\Big|\lesssim  \sum_{\underset{|j|\leqslant C_{0}\langle l\rangle, |l|\leqslant N_{k}}{(l,j)\in\mathbb{Z}^{d}\times\mathbb{Z}\setminus\{(0,0)\}}}\Big|\mathscr{R}_{l,j}^{(0)}(i_{k})\Big|+\sum_{\underset{\underset{\min(|j|,|j_{0}|)\leqslant c_{2}\gamma_{k+1}^{-\upsilon}\langle l\rangle^{\tau_{1}}}{|j-j_{0}|\leqslant C_{0}\langle l\rangle, |l|\leqslant N_{k}}}{(l,j,j_{0})\in\mathbb{Z}^{d}\times(\mathbb{S}_{0}^{c})^{2}}}\Big|\mathscr{R}_{l,j,j_{0}}^{(2)}(i_{k})\Big|+\sum_{\underset{|j|\leqslant C_{0}\langle l\rangle, |l|\leqslant N_{k}}{(l,j)\in\mathbb{Z}^{d}\times\mathbb{S}_{0}^{c}}}\Big|\mathscr{R}_{l,j}^{(1)}(i_{k})\Big|.
	\end{align}
	Observe that for $|j-j_{0}|\leqslant C_{0}\langle l\rangle$ and $\min(|j|,|j_{0}|)\leqslant c_2\gamma_{k+1}^{-\upsilon}\langle l\rangle^{\tau_{1}}$, then 
	\begin{equation}\label{bound maxjj0}
		\max(|j|,|j_{0}|)=\min(|j|,|j_{0}|)+|j-j_{0}|\leqslant c_2\gamma_{k+1}^{-\upsilon}\langle l\rangle^{\tau_{1}}+C_{0}\langle l\rangle\lesssim\gamma^{-\upsilon}\langle l\rangle^{\tau_{1}}.
	\end{equation}
	Combining \eqref{e-scrR0112}, \eqref{e-Ck-Ck+1} and \eqref{bound maxjj0}, we get for $k\in\{0,1\}$,
	\begin{align*}
		\Big|\mathscr{C}_{k}\setminus\mathscr{C}_{k+1}\Big|&\lesssim  \gamma^{\frac{1}{q_{0}}}\sum_{(l,j)\in\mathbb{Z}^{d}\times(\mathbb{S}_0)^c\atop |j|\leqslant C_{0}\langle l\rangle}\langle j\rangle^{\frac{1}{q_{0}}}\langle l\rangle^{-1-\frac{\tau_{1}+1}{q_{0}}}+\gamma^{\frac{1}{q_{0}}}\sum_{\underset{\underset{\min(|j|,|j_{0}|)\leqslant c_{2}\gamma_{k+1}^{-\upsilon}\langle l\rangle^{\tau_{1}}}{|j-j_{0}|\leqslant C_{0}\langle l\rangle, |l|\leqslant N_{k}}}{(l,j,j_{0})\in\mathbb{Z}^{d}\times(\mathbb{S}_{0}^{c})^{2}}}\langle j-j_{0}\rangle^{\frac{1}{q_{0}}}\langle l\rangle^{-1-\frac{\tau_{2}+1}{q_{0}}}\\
		&\quad+\gamma^{\frac{\upsilon}{q_0}}\sum_{(l,j)\in\mathbb{Z}^{d}\times\mathbb{Z}\setminus\{(0,0)\}\atop|j|\leqslant C_{0}\langle l\rangle}\langle j\rangle^{\frac{1}{q_{0}}}\langle l\rangle^{-1-\frac{\tau_{1}+1}{q_{0}}}.
	\end{align*}
	Thus, by using \eqref{ptrans} and \eqref{to2}, we deduce
	\begin{align}\label{e-Ck-Ck+1 k01}
		\max_{k\in\{0,1\}}\Big|\mathscr{C}_{k}\setminus\mathscr{C}_{k+1}\Big|&\lesssim  \gamma^{\frac{1}{q_{0}}}\Bigg(\sum_{l\in\mathbb{Z}^d}\langle l\rangle^{-\frac{\tau_{1}}{q_{0}}}+\gamma^{-\upsilon}\sum_{l\in\mathbb{Z}^d}\langle l\rangle^{\tau_1-1-\frac{\tau_{2}}{q_{0}}}\Bigg)
		+\gamma^{\frac{\upsilon}{q_0}}\sum_{l\in\mathbb{Z}^d}\langle l\rangle^{-\frac{\tau_{1}}{q_{0}}}\\
		\nonumber &\lesssim  \gamma^{\min\left(\frac{\upsilon}{q_{0}},\frac{1}{q_{0}}-\upsilon\right)}=\gamma^{\frac{\upsilon}{q_0}}.
	\end{align}
	Now fix some $n\geqslant 2.$ Using \eqref{decomp Cnbis} and Lemma \ref{lem empty Cant}, we get
	\begin{align*}
		\Big|\mathscr{C}_{n}\setminus\mathscr{C}_{n+1}\Big|\leqslant \sum_{\underset{|j|\leqslant C_{0}\langle l\rangle,N_{n-1}<|l|\leqslant N_{n}}{(l,j)\in\mathbb{Z}^{d}\times\mathbb{Z}\setminus\{(0,0)\}}}\Big|\mathscr{R}_{l,j}^{(0)}(i_{n})\Big|+\sum_{\underset{\underset{\min(|j|,|j_{0}|)\leqslant c_{2}\gamma_{n+1}^{-\upsilon}\langle l\rangle^{\tau_{1}}}{|j-j_{0}|\leqslant C_{0}\langle l\rangle,N_{n-1}<|l|\leqslant N_{n}}}{(l,j,j_{0})\in\mathbb{Z}^{d}\times(\mathbb{S}_{0}^{c})^{2}}}\Big|\mathscr{R}_{l,j,j_{0}}^{(2)}(i_{n})\Big|+\sum_{\underset{|j|\leqslant C_{0}\langle l\rangle,N_{n-1}<|l|\leqslant N_{n}}{(l,j)\in\mathbb{Z}^{d}\times\mathbb{S}_{0}^{c}}}\Big|\mathscr{R}_{l,j}^{(1)}(i_{n})\Big|.
	\end{align*}
	Therefore, we obtain from \eqref{e-scrR0112} and \eqref{bound maxjj0},
	\begin{align*}
		\Big|\mathscr{C}_{n}\setminus\mathscr{C}_{n+1}\Big|\lesssim 
		\gamma^{\frac{1}{q_{0}}}\Bigg(\sum_{l\in\mathbb{Z}^d \atop|l|>N_{n-1}}\langle l\rangle^{-\frac{\tau_{1}}{q_{0}}}+\gamma^{-\upsilon}\sum_{l\in\mathbb{Z}^d \atop|l|>N_{n-1}}\langle l\rangle^{\tau_1-1-\frac{\tau_{2}}{q_{0}}}\Bigg)
		+\gamma^{\frac{\upsilon}{q_0}}\sum_{l\in\mathbb{Z}^d \atop|l|>N_{n-1}}\langle l\rangle^{-\frac{\tau_{1}}{q_{0}}}.
	\end{align*}
	Hence, we infer
	\begin{align}\label{sumCk-Ck+1}
		\sum_{n=2}^\infty \Big|\mathscr{C}_{n}\setminus\mathscr{C}_{n+1}\Big|&\lesssim  \gamma^{\frac{\upsilon}{q_{0}}}.
	\end{align}
	Plugging \eqref{sumCk-Ck+1} and \eqref{e-Ck-Ck+1 k01} into \eqref{e-Cant comp} and using \eqref{gma N0 NM} yields
	$$\Big|(\alpha_{0},\alpha_{1})\setminus\mathscr{C}_{\infty}^{\varepsilon}\Big|\lesssim  \gamma^{\frac{\upsilon}{q_{0}}}=\varepsilon^{\frac{a\upsilon}{q_0}}.$$
	This achieves the proof of Proposition \ref{prop meas Cant}. 
\end{proof}
To conclude, it remains to prove the following lemma providing necessary constraints between time and space Fourier modes so that the sets in \eqref{split Cn-Cn+1} are not void.
\begin{lem}\label{lem empty Cant}
	There exists $\varepsilon_0$ such that for any $\varepsilon\in[0,\varepsilon_0]$ and $n\in\mathbb{N}$ the following assertions hold true. 
	\begin{enumerate}[label=(\roman*)]
		\item Let $(l,j)\in\mathbb{Z}^{d}\times\mathbb{Z}\setminus\{(0,0)\}.$ If $\,\displaystyle\mathscr{R}_{l,j}^{(0)}(i_{n})\neq\varnothing,$ then $|j|\leqslant C_{0}\langle l\rangle.$
		\item Let $(l,j,j_{0})\in\mathbb{Z}^{d}\times(\mathbb{S}_{0}^{c})^{2}.$ If $\,\displaystyle\mathscr{R}_{l,j,j_{0}}^{(2)}(i_{n})\neq\varnothing,$ then $|j-j_{0}|\leqslant C_{0}\langle l\rangle.$
		\item Let $(l,j)\in\mathbb{Z}^{d}\times\mathbb{S}_{0}^{c}.$ If $\,\displaystyle \mathscr{R}_{l,j}^{(1)}(i_{n})\neq\varnothing,$ then $|j|\leqslant C_{0}\langle l\rangle.$
		\item Let $(l,j,j_{0})\in\mathbb{Z}^{d}\times(\mathbb{S}_{0}^{c})^{2}.$ There exists $c_{2}>0$ such that if $\displaystyle \min(|j|,|j_{0}|)\geqslant c_{2}\gamma_{n+1}^{-\upsilon}\langle l\rangle^{\tau_{1}},$ then 
		$$\mathscr{R}_{l,j,j_{0}}^{(2)}(i_{n})\subset\mathscr{R}_{l,j-j_{0}}^{(0)}(i_{n}).$$
	\end{enumerate}
\end{lem}
\begin{proof}
	\textbf{(i)} Assume that $\mathscr{R}_{l,j}^{(0)}(i_{n})\neq\varnothing.$ Then, we can find $\alpha\in(\alpha_{0},\alpha_{1})$ such that
	$$|\omega(\alpha,\varepsilon)\cdot l+j \mathtt{m}_{n}^{\infty}(\alpha,\varepsilon)|\leqslant\tfrac{4\gamma_{n+1}^{\upsilon}\langle j\rangle}{\langle l\rangle^{\tau_{1}}}.$$ 
	By using the triangle and Cauchy-Schwarz inequalities together with \eqref{def gamman}, \eqref{gma N0 NM} and the fact that $(\alpha,\varepsilon)\mapsto \omega(\alpha, \varepsilon)$ is bounded, we infer 
	\begin{align*}
		|\mathtt{m}_{n}^{\infty}(\alpha,\varepsilon)||j|&\leqslant 4|j|\gamma_{n+1}^{\upsilon}\langle l\rangle^{-\tau_{1}}+|{\omega}(\alpha,\varepsilon)\cdot l|\\
		&\leqslant 4|j|\gamma_{n+1}^{\upsilon}+C\langle l\rangle\\
		&\leqslant 8\varepsilon^{a \upsilon}|j|+C\langle l\rangle.
	\end{align*}
	Now, by \eqref{r1-Vn-rn}, we can write
	$$\mathtt{m}_{n}^{\infty}(\alpha,\varepsilon)=V_0(\alpha)+r^{0,n}(\alpha,\varepsilon).$$
	Hence, applying \eqref{e-r1}, \eqref{gma N0 NM} and Proposition \ref{Nash-Moser}-(i)-(a) we find
	\begin{align}\label{unif e-r1}
		\forall k\in\llbracket 0,q\rrbracket,\quad\sup_{n\in\mathbb{N}}\sup_{\alpha\in(\alpha_{0},\alpha_{1})}|\partial_{\alpha}^{k}r^{0,n}(\alpha,\varepsilon)|&\leqslant\gamma^{-k}\sup_{n\in\mathbb{N}}\| r^{0,n}\|_{q}^{\gamma,\mathcal{O}}\nonumber\\
		&\lesssim\varepsilon\gamma^{-k}\nonumber\\
		&\lesssim\varepsilon^{1-ak}.
	\end{align}
	Thus, for $\varepsilon$ small enough, we obtain by \eqref{LaKa} and the decay property of $I_1K_1$ on $(0,\infty),$
	$$\inf_{n\in\mathbb{N}}\inf_{\alpha\in(\alpha_{0},\alpha_{1})}|\mathtt{m}_{n}^{\infty}(\alpha,\varepsilon)|\geqslant\tfrac{1}{2}V_0\left(\tfrac{1}{\alpha_0}\right).$$
	Hence, a suitable choice of $\varepsilon$ small enough provides the constraint $|j|\leqslant C_{0}\langle l\rangle$ for some $C_{0}>0.$\\
	
	\noindent\textbf{(ii)} Observe that $\mathscr{R}_{l,j_0,j_{0}}^{(2)}(i_{n})=\mathscr{R}^{(0)}_{l,0}(i_{n})$, so the case $j=j_0$ can be included in the previous point. Now we assume $j\neq j_0$ and $\mathscr{R}_{l,j,j_{0}}^{(2)}(i_{n})\neq\varnothing.$ There exists $\alpha\in(\alpha_{0},\alpha_{1})$ such that 
	$$|\omega(\alpha,\varepsilon)\cdot l+\mathrm{d}_{j}^{\infty,n}(\alpha,\varepsilon)-\mathrm{d}_{j_{0}}^{\infty,n}(\alpha,\varepsilon)|\leqslant\tfrac{2\gamma_{n+1}|j-j_0|}{\langle l\rangle^{\tau_2}}.$$
	Applying the triangle and Cauchy-Schwarz inequalities, \eqref{def gamman} and \eqref{gma N0 NM}, we infer
	\begin{align*}
		|\mathrm{d}_{j}^{\infty,n}(\alpha,\varepsilon)-\mathrm{d}_{j_{0}}^{\infty,n}(\alpha,\varepsilon)|&\leqslant 2\gamma_{n+1}|j-j_{0}|\langle l\rangle^{-\tau_{2}}+|{\omega}(\alpha,\varepsilon)\cdot l|\\
		&\leqslant 2\gamma_{n+1}|j-j_{0}|+C\langle l\rangle\\
		&\leqslant 4\varepsilon^a|j-j_{0}|+C\langle l\rangle.
	\end{align*}
	Now, similarly to \eqref{unif e-r1}, we may obtain from \eqref{specDinfty}-\eqref{gma N0 NM},
	\begin{align}\label{unif e-rjfty}
		\forall k\in\llbracket 0,q\rrbracket,\quad\sup_{n\in\mathbb{N}}\sup_{j\in\mathbb{S}_{0}^{c}}\sup_{\alpha\in(\alpha_{0},\alpha_{1})}|j||\partial_{\alpha}^{k}r_{j}^{\infty,n}(\alpha,\varepsilon)|&\leqslant\gamma^{-k}\sup_{n\in\mathbb{N}}\sup_{j\in\mathbb{S}_{0}^{c}}|j|\| r_{j}^{\infty,n}\|_{q}^{\gamma,\mathcal{O}}\nonumber\\
		&\lesssim\varepsilon\gamma^{-1-k}\nonumber\\
		&\lesssim\varepsilon^{1-a(1+k)}.
	\end{align} 
	One obtains from the triangle inequality, Lemma \ref{properties omegajalpha}-(iv), \eqref{unif e-r1} and \eqref{unif e-rjfty}, up to taking $\varepsilon$ sufficiently small
	\begin{align*}
		|\mathrm{d}_{j}^{\infty,n}(\alpha,\varepsilon)-\mathrm{d}_{j_{0}}^{\infty,n}(\alpha,\varepsilon)| & \geqslant  |\Omega_{j}^{\textnormal{\tiny{E}}}(\alpha)-\Omega_{j_{0}}^{\textnormal{\tiny{E}}}(\alpha)|-|r^{0,n}(\alpha,\varepsilon)||j-j_{0}|-|r_{j}^{\infty,n}(\alpha,\varepsilon)|-|r_{j_{0}}^{\infty,n}(\alpha,\varepsilon)|\\
		& \geqslant  \big(\Omega-C\varepsilon^{1-a}\big)|j-j_{0}|\\
		& \geqslant \tfrac{\Omega}{2}|j-j_{0}|.
	\end{align*}
	The foregoing inequalities together give for $\varepsilon $ small enough, the constraint $|j-j_{0}|\leqslant C_{0}\langle l\rangle,$ for some $C_{0}>0.$\\
	
	\noindent \textbf{(iii)} As previously, we can forget the case $j=0.$ Assume that $j\neq 0$ and $\mathscr{R}_{l,j}^{(1)}(i_{n})\neq\varnothing.$ There exists $\alpha\in(\alpha_{0},\alpha_{1})$ such that 
	$$|\omega(\alpha,\varepsilon)\cdot l+\mathrm{d}_{j}^{\infty,n}(\alpha,\varepsilon)|\leqslant\tfrac{\gamma_{n+1}|j|}{\langle l\rangle^{\tau_1}}.$$
	By the same techniques as in the other cases, we find
	\begin{align*}
		|\mathrm{d}_{j}^{\infty,n}(\alpha,\varepsilon)|&\leqslant \gamma_{n+1}|j|\langle l\rangle^{-\tau_{1}}+|{\omega}(\alpha,\varepsilon)\cdot l|\\
		&\leqslant  2\varepsilon^a|j|+C\langle l\rangle.
	\end{align*}
	Now \eqref{mujinftyn}, the triangle inequality, Lemma \ref{properties omegajalpha}-(iii), \eqref{unif e-r1} and \eqref{unif e-rjfty} imply
	\begin{align*}
		|\mathrm{d}_{j}^{\infty,n}(\alpha,\varepsilon)|&\geqslant \Omega|j|-|j||r^{0,n}(\alpha,\varepsilon)|-|r_{j}^{\infty,n}(\alpha,\varepsilon)|\\
		&
		\geqslant \Omega|j|-C\varepsilon^{1-a}|j|.
	\end{align*}
	Gathering the foregoing inequalities gives
	\begin{align*}
		\big( \Omega-C\varepsilon^{1-a}-2\varepsilon^a\big)|j|
		&\leqslant  C\langle l\rangle.
	\end{align*}
	Thus, for $\varepsilon$ small enough we deduce the constraint $|j|\leqslant C_0\langle  l\rangle,$ for some $C_{0}>0.$\\
	
	\noindent \textbf{(iv)} We can forget the case $j=j_0.$ Assume $j\neq j_0.$ The symmetry property $\mathrm{d}_{-j}^{\infty,n}=-\mathrm{d}_{j}^{\infty,n}$ implies that without loss of generality we can assume that $0\leqslant j<j_0.$ Take $\alpha\in\mathscr{R}_{l,j,j_{0}}^{(2)}(i_{n}),$ i.e.
	$$\big|{\omega}(\alpha,\varepsilon)\cdot l+\mathrm{d}_{j}^{\infty,n}(\alpha,\varepsilon)\pm\mathrm{d}_{j_{0}}^{\infty,n}(\alpha,\varepsilon)\big|\leqslant\tfrac{2\gamma_{n+1}\langle j\pm j_{0}\rangle}{\langle l\rangle^{\tau_{2}}}.
	$$
	Putting together \eqref{mujinftyn}, \eqref{Omegajalpha} and the triangle inequality, we find
	\begin{align*}
		\big|{\omega}(\alpha,\varepsilon)\cdot l+(j\pm j_{0})\mathtt{m}_{n}^{\infty}(\alpha,\varepsilon)\big|  &\leqslant  \big|{\omega}(\alpha,\varepsilon)\cdot l+\mathrm{d}_{j}^{\infty,n}(\alpha,\varepsilon)\pm\mathrm{d}_{j_{0}}^{\infty,n}(\alpha,\varepsilon)\big|+\tfrac{1}{2}(1\pm 1)\\
		&\quad+|jI_{j}\left(\tfrac{1}{\alpha}\right)K_{j}\left(\tfrac{1}{\alpha}\right)\pm j_{0}I_{j_{0}}\left(\tfrac{1}{\alpha}\right)K_{j_{0}}\left(\tfrac{1}{\alpha}\right)|+\big|r_{j}^{\infty,n}(\alpha,\varepsilon)\pm r_{j_{0}}^{\infty,n}(\alpha,\varepsilon)\big|.
	\end{align*}
	Hence, we deduce
	\begin{align}\label{e-inc R2 in R1}
		\big|{\omega}(\alpha,\varepsilon)\cdot l+(j\pm j_{0})\mathtt{m}_{n}^{\infty}(\alpha,\varepsilon)\big|\leqslant & \tfrac{2\gamma_{n+1}\langle j\pm j_{0}\rangle}{\langle l\rangle^{\tau_{2}}}+|jI_{j}\left(\tfrac{1}{\alpha}\right)K_{j}\left(\tfrac{1}{\alpha}\right)\pm j_{0}I_{j_{0}}\left(\tfrac{1}{\alpha}\right)K_{j_{0}}\left(\tfrac{1}{\alpha}\right)|\\
		&\quad+\tfrac{1}{2}(1\pm 1)+\big|r_{j}^{\infty,n}(\alpha,\varepsilon)-r_{j_{0}}^{\infty,n}(\alpha,\varepsilon)\big|.\nonumber
	\end{align}
It has been proved in \cite[Lem. 7.2-(iv)]{HR21} that
\begin{align*}
	\forall x>0,\quad|jI_{j}(x)K_{j}(x)\pm j_{0}I_{j_{0}}(x)K_{j_{0}}(x)|\leqslant \tfrac{\langle j\pm j_{0}\rangle}{\min(|j|,|j_{0}|)}\cdot
\end{align*}
	Additionally, \eqref{specDinfty} gives
	\begin{align*}
		\big|r_{j}^{\infty,n}(\alpha,\varepsilon)\pm r_{j_{0}}^{\infty,n}(\alpha,\varepsilon)\big|\leqslant &C \varepsilon^{1-a}\big(|j|^{-1}+|j_0|^{-1}\big)\\
		\leqslant & C \varepsilon^{1-a} \tfrac{\langle j\pm j_{0}\rangle}{\min(|j|,|j_{0}|)}\cdot
	\end{align*}
Also one has the trivial bound
$$\tfrac{1}{2}(1\pm 1)\leqslant\tfrac{\langle j\pm j_0\rangle}{\min(|j|,|j_0|)}.$$
Inserting the foregoing estimates into \eqref{e-inc R2 in R1} yields 
	\begin{align*}
		\nonumber \big|{\omega}(\alpha,\varepsilon)\cdot l+(j\pm j_{0})\mathtt{m}_{n}^{\infty}(\alpha,\varepsilon)\big|  \leqslant & \tfrac{2\gamma_{n+1}\langle j\pm j_{0}\rangle}{\langle l\rangle^{\tau_{2}}}+ C \tfrac{\langle j\pm j_{0}\rangle}{\min(|j|,|j_{0}|)}\cdot
	\end{align*}
	Thus, for $\displaystyle \min(|j|,|j_{0}|)\geqslant \tfrac{1}{2} C\gamma_{n+1}^{-\upsilon}\langle l\rangle^{\tau_{1}},$ then using \eqref{to2}, we get
	$$\big|{\omega}(\alpha,\varepsilon)\cdot l+(j\pm j_{0})\mathtt{m}_{n}^{\infty}(\alpha,\varepsilon)\big| \leqslant  \tfrac{4\gamma_{n+1}^{\upsilon}\langle j\pm j_{0}\rangle}{\langle l\rangle^{\tau_{1}}}\cdot$$
	This proves Lemma \ref{lem empty Cant}, taking $c_{2}\triangleq\frac{C}{2}.$
\end{proof}
\appendix
\section{Properties of modified Bessel functions}\label{appendix Bessel}
We collect here the definitions and useful properties of modified Bessel functions. The literature is huge as regards these special functions and we may refer the reader to \cite{AS64,W95} for a nice introduction. The modified Bessel functions of first and second kind are defined as follows
$$I_{\nu}(z)\triangleq\sum_{m=0}^{+\infty}\frac{\left(\frac{z}{2}\right)^{\nu+2m}}{m!\Gamma(\nu+m+1)},\mbox{ }\quad|\mbox{arg}(z)|<\pi$$
and $$\forall \nu\in\mathbb{C}\backslash\mathbb{Z},\quad K_{\nu}(z)\triangleq\frac{\pi}{2}\frac{I_{-\nu}(z)-I_{\nu}(z)}{\sin(\nu\pi)},\qquad\forall n\in\mathbb{Z},\quad K_{n}(z)\triangleq\displaystyle\lim_{\nu\rightarrow n}K_{\nu}(z), \qquad|\mbox{arg}(z)|<\pi.$$ 
\textbf{Symmetry and positivity properties} (see \cite[p. 375]{AS64}) \textbf{:}
\begin{equation}\label{sym Bessel}
	\forall n\in\mathbb{N},\quad\forall x>0,\quad I_{-n}(x)=I_{n}(x)\in\mathbb{R}_{+}^{*}\quad\mbox{ and }\quad K_{-n}(x)=K_{n}(x)\in\mathbb{R}_{+}^{*}.
\end{equation}
\textbf{Derivatives and Anti-derivatives :}\\
If we set $\mathcal{Z}_{\nu}(z)\triangleq I_{\nu}(z)$ or $e^{i\nu\pi}K_{\nu}(z)$, then for all $\nu\in\mathbb{R}$, we have 
\begin{equation}\label{Bessel derivatives}
	\mathcal{Z}_{\nu}'(z)=\mathcal{Z}_{\nu-1}(z)-\frac{\nu}{z}\mathcal{Z}_{\nu}(z)=\mathcal{Z}_{\nu+1}(z)+\frac{\nu}{z}\mathcal{Z}_{\nu}(z)
\end{equation}
and
\begin{equation}\label{Bessel and anti-derivatives}
\int z^{\nu+1}\mathcal{Z}_{\nu}(z)dz=z^{\nu+1}\mathcal{Z}_{\nu+1}(z).
\end{equation}
\textbf{Power series extension for $K_{n}$} (see \cite[p. 375]{AS64}) \textbf{:}\\
\begin{align*}
	K_{n}(z)=&\frac{1}{2}\left(\frac{z}{2}\right)^{-n}\sum_{k=0}^{n-1}\frac{(n-k-1)!}{k!}\left(\frac{-z}{4}\right)^{k}+(-1)^{n+1}\ln\left(\frac{z}{2}\right)I_{n}(z)\\
	&+\frac{1}{2}\left(\frac{-z}{2}\right)^{n}\sum_{k=0}^{+\infty}\left(\psi(k+1)+\psi(n+k+1)\right)\frac{\left(\frac{z^{2}}{4}\right)^{k}}{k!(n+k)!},
\end{align*}
where 
$$\psi(1)\triangleq-\boldsymbol{\gamma} \textnormal{ (Euler's constant)}\qquad \qquad\forall m\in\mathbb{N}^{*},\quad\psi(m+1)\triangleq\sum_{k=1}^{m}\frac{1}{k}-\boldsymbol{\gamma}.$$
In particular 
\begin{equation}\label{exp K0}
	K_{0}(z)=-\log\left(\frac{z}{2}\right)I_{0}(z)+\sum_{m=0}^{+\infty}\frac{\left(\frac{z}{2}\right)^{2m}}{(m!)^{2}}\psi(m+1),
\end{equation}
so $K_{0}$ behaves like a logarithm at $0.$\\

\noindent\textbf{Decay property for the product $I_{\nu}K_{\nu}$} (see \cite{B09} and \cite{DHR19}) \textbf{:}\\ The application $(x,\nu)\mapsto I_{\nu}(x)K_{\nu}(x)$ is strictly decreasing in each variable $x,\nu>0.$\\

\noindent\textbf{Wronskian} (see \cite[p. 375]{AS64}) \textbf{:}
\begin{equation}\label{wronskian}
I_{\nu}'(z)K_{\nu}(z)-I_{\nu}(z)K_{\nu}'(z)=I_{\nu}(z)K_{\nu+1}(z)+I_{\nu+1}(z)K_{\nu}(z)=\frac{1}{z}.
\end{equation}
\textbf{Ratio bounds} (see \cite{BP13}) \textbf{:}\\
For all $n\in\mathbb{N}$ and $x>0,$ we have 
\begin{equation}\label{ratio bounds with derivatives}
\left\lbrace\begin{array}{l}
\displaystyle\frac{x I_{n}'(x)}{I_{n}(x)}<\sqrt{x^{2}+n^{2}}\\
\\
\displaystyle\frac{x K_{n}'(x)}{K_{n}(x)}<-\sqrt{x^{2}+n^{2}}.
\end{array}\right.
\end{equation}
\textbf{Asymptotic expension of small argument} (see \cite[p. 375]{AS64}) \textbf{:}
\begin{equation}\label{asymp small arg}
	\forall n\in\mathbb{N}^{*},\quad I_{n}(x)\underset{x\rightarrow 0}{\sim}\frac{\left(\frac{1}{2}x\right)^{n}}{\Gamma(n+1)}\quad\mbox{ and }\quad K_{n}(x)\underset{x\rightarrow 0}{\sim}\frac{\Gamma(n)}{2\left(\frac{1}{2}x\right)^{n}}.
\end{equation} 
\textbf{Asymptotic expension of large argument} (see \cite[p. 375]{AS64}) \textbf{:}
\begin{equation}\label{asymp large z}
	\forall n\in\mathbb{N}^{*},\quad I_{n}(x)\underset{x\rightarrow +\infty}{\sim}\frac{e^{x}}{\sqrt{2\pi x}}\quad\mbox{ and }\quad K_{n}(x)\underset{x\rightarrow +\infty}{\sim}\sqrt{\frac{\pi}{2x}}e^{-x}.
\end{equation}
\textbf{Asymptotic of high order} (see \cite[p. 377]{AS64}) \textbf{:}
\begin{equation}\label{asymp high order}
	\forall x>0,\quad\mbox{ }I_{\nu}(x)\underset{\nu\rightarrow+\infty}{\sim}\frac{1}{\sqrt{2\pi\nu}}\left(\frac{ex}{2\nu}\right)^{\nu}\quad\mbox{ and }\quad K_{\nu}(x)\underset{\nu\rightarrow+\infty}{\sim}\sqrt{\frac{\pi}{2\nu}}\left(\frac{ex}{2\nu}\right)^{-\nu}.
\end{equation}
\section{Local bifurcation theorem and singular integrals}
We recall here the Crandall-Rabinowitz's Theorem of local bifurcation theory which was used to find the periodic solutions in Section \ref{sec per}. Its proof can be found in \cite{CR71} and \cite[p.15]{K11}. 
\begin{theo}[Crandall-Rabinowitz]\label{Crandall-Rabinowitz theorem}
	Let $X$ and $Y$ be two Banach spaces. Let $V$ be a neighborhood of $0$ in $X$ and let
	$$F:\begin{array}[t]{rcl}
		\mathbb{R}\times V & \rightarrow  & Y\\
		(\Omega,x) & \mapsto & F(\Omega,x)
	\end{array}$$
	be a function of classe $C^{1}$ with the following properties 
	\begin{enumerate}[label=(\roman*)]
		\item (Trivial solution) $\forall\,\Omega\in\mathbb{R},\,\,F(\Omega,0)=0.$
		\item (Regularity) $\partial_{\Omega}F$, $d_{x}F$ and $\partial_{\Omega}d_{x}F$ exist and are continuous.
		\item (Fredholm property) $d_{x}F(0,0)$ is a Fredholm operator with index $0$ and $\ker\left(d_{x}F(0,0)\right)=\langle x_{0}\rangle.$
		\item (Transversality assumption) $\partial_{\Omega}d_{x}F(0,0)[x_{0}]\not\in R\left(d_{x}F(0,0)\right).$
	\end{enumerate}
	If $\chi$ denotes any complement of $\ker\left(d_{x}F(0,0)\right)$ in $X$, then there exist 
	\begin{enumerate}[label=\textbullet]
		\item $U$ a neighborhood of $(0,0)$ in $\mathbb{R}\times V,$
		\item an interval $(-a,a)$ for some $a>0$,
		\item continuous functions
		$\psi:(-a,a)\rightarrow \mathbb{R}$ and $\phi:(-a,a)\rightarrow\chi$ satisfying $\psi(0)=0$ and $\phi(0)=0$
	\end{enumerate} 
	such that the set of the zeros of $F$ in $U$ can be described as the following two curves intersecting at $(0,0)$
	$$\Big\{(\Omega,x)\in U\quad\textnormal{s.t.}\quad F(\Omega,x)=0\Big\}=\Big\{\big(\psi(s),sx_{0}+s\phi(s)\big)\quad\textnormal{s.t.}\quad|s|<a\Big\}\cup\Big\{(\Omega,0)\in U\Big\}.$$
\end{theo}
Now, we also state some continuity properties of singular integral operators. We may refer to \cite{HH15,H09,K14,LL97} for a proof of the following result.
\begin{lem}\label{lem sing ker}
	Consider a function $\mathtt{K}:\mathbb{T}\times\mathbb{T}\rightarrow\mathbb{C}$ such that for some $C_0>0$ we have
	\begin{enumerate}[label=\textbullet]
		\item $\mathtt{K}$ is measurable on $\mathbb{T}\times\mathbb{T}\setminus\{(w,w),\,w\in\mathbb{T}\}$ and
		$$\forall(w,\tau)\in\mathbb{T}^2,\qquad w\neq\tau\quad\Rightarrow\quad|\mathtt{K}(w,\tau)|\leqslant C_0.$$
		\item For any $\tau\in\mathbb{T}$, the application $w\mapsto\mathtt{K}(w,\tau)$ is differentiable in $\mathbb{T}\setminus\{\tau\}$ and
		$$\forall(w,\tau)\in\mathbb{T}^2,\qquad w\neq\tau\quad\Rightarrow\quad|\partial_{w}\mathtt{K}(w,\tau)|\leqslant\frac{C_0}{|w-\tau|}.$$
	\end{enumerate}
Consider the operator $\mathcal{T}_{\mathtt{K}}$ defined by
$$\mathcal{T}_{\mathtt{K}}(f)(w)=\fint_{\mathbb{T}}\mathtt{K}(w,\tau)f(\tau)\,d\tau.$$
Then, the operator $\mathcal{T}_{\mathtt{K}}$ is continuous from $L^{\infty}(\mathbb{T})$ to $C^{\delta}(\mathbb{T})$ for any $0<\delta<1$ and there exists $C_{\delta}>0$ such that
	$$\|\mathcal{T}_{\mathtt{K}}(f)\|_{C^{\delta}(\mathbb{T})}\leqslant C_{\delta}C_0\|f\|_{L^{\infty}(\mathbb{T})}.$$
\end{lem}

\end{document}